\documentclass[a4paper,11pt]{amsart}
\usepackage{amsmath}
\usepackage{cases}

\usepackage{amsfonts}
\usepackage[colorlinks,linkcolor=blue,citecolor=blue]{hyperref}
\usepackage{latexsym, amssymb, amsmath, amsthm, bbm}
\usepackage[all]{xy}
\usepackage{pgfplots}

\DeclareSymbolFont{EulerExtension}{U}{euex}{m}{n}
\DeclareMathSymbol{\euintop}{\mathop} {EulerExtension}{"52}
\DeclareMathSymbol{\euointop}{\mathop} {EulerExtension}{"48}

\allowdisplaybreaks[4]

\def \id{\operatorname{Id}}
\def \ker{\operatorname{Ker}}
\def \ord{\operatorname{ord}}
\def \mi{\operatorname{min}}
\def \im{\operatorname{im}}
\def \io{\operatorname{io}}

\def \D{\Delta}

\def \e{\varepsilon}

\def \N{\mathbb{N}}
\def \Z{\mathbb{Z}}
\def \unit{\mathbbm{1}}
\def \k{\mathbbm{k}}
\def \To{\longrightarrow}
\def \dim{\operatorname{dim}}

\def \Hom{\operatorname{Hom}}

\def \Id{\operatorname{Id}}
\def \ord{\operatorname{ord}}
\def \Rep{\operatorname{Rep}}
\def \Ext{\operatorname{Ext}}

\def \id{\operatorname{Id}}
\def \ker{\operatorname{Ker}}

\def \im{\operatorname{im}}
\def \io{\operatorname{io}}

\def \D{\Delta}

\def \e{\varepsilon}

\def \N{\mathbb{N}}

\def \Z{\mathbb{Z}}
\def \unit{\mathbbm{1}}

\numberwithin{equation}{section}

\newtheorem{theorem}{Theorem}[section]
\newtheorem{lemma}[theorem]{Lemma}
\newtheorem{proposition}[theorem]{Proposition}
\newtheorem{corollary}[theorem]{Corollary}
\newtheorem{definition}[theorem]{Definition}
\newtheorem{example}[theorem]{Example}
\newtheorem{remark}[theorem]{Remark}
\newtheorem{question}[theorem]{Question}

\newtheorem{conjecture}[theorem]{Conjecture}

\begin{document}
\title{A classification result on prime Hopf algebras of GK-dimension one}\thanks{$^\dag$Supported by NSFC 11722016.}


\author{Gongxiang Liu}
\address{Department of Mathematics, Nanjing University, Nanjing 210093, China} \email{gxliu@nju.edu.cn}
\date{}
\maketitle
\dedicatory{Dedicate to Professor Shao-Xue Liu for his 90th birthday with my deepest admiration}
\begin{abstract} In this paper, we classify all prime Hopf algebras $H$ of GK-dimension one satisfying the following two conditions: 1) $H$ has a 1-dimensional representation of order PI.deg$(H)$  and 2) the invariant components of $H$ with respect to this 1-dimensional representation are domains (see Section 2 for related definitions). As consequences, 1) a number of new Hopf algebras of GK-dimension one are found and some of them are not pointed, 2) we give a partial answer to a question posed in \cite{BZ} and 3) two new series of finite-dimensional Hopf algebras are found which in particular gives us a Hopf algebra of dimension $24$ (see \cite{BGNR}).
\end{abstract}

\section{Introduction}

Throughout this paper, $\k$ denotes an algebraically closed field
of characteristic $0$, all vector spaces are over $\k$. All algebras considered in this paper are noetherian and affine unless stated otherwise. The antipode of a Hopf algebra is assumed to be bijective.
\subsection{Motivation.} We are motivated by the following three seemingly irrelevant but indeed related phenomenons. The first one is based on the next simple observation. It is well-known that the affine line $\mathbb{A}^{1}$ is a commutative algebraic group of dimension one. If we consider the infinite dimensional Taft algebra $T(n, t, \xi)$ (see Subsection \ref{ss2.3} for its definition), then we find that the affine line (here and the following we identify an affine variety with its coordinate algebra) is also a Hopf algebra in the braided tensor category $^{\mathbb{Z}_{n}}_{\mathbb{Z}_n}{\mathcal{YD}}$ of Yetter-Drinfeld modules of $\k \Z_{n}$. Intuitively,
\begin{figure}[hbt]
\begin{picture}(100,60)(0,-60)
\put(0,0){\line(-1,-1){50}}\put(60,-30){\makebox(0,0){$\in {^{\mathbb{Z}_{n}}_{\mathbb{Z}_n}{\mathcal{YD}}}.$}}
\end{picture}
\end{figure}

\noindent From this, a natural question is:
\begin{equation}\label{1.1}\textsf{ Can we realize other irreducible curves as Hopf algebras in } ^{\mathbb{Z}_{n}}_{\mathbb{Z}_n}{\mathcal{YD}}?\end{equation}
In order to answer this question, we need give two remarks at first. Firstly, observe that above line is smooth and thus the infinite dimensional Taft algebra is \emph{regular}, i.e. has finite global dimension. Secondly, it is harmless to assume that the action of $\Z_n$ on the curve is faithful since otherwise one can take a smaller group $\Z_{m}$ with $m|n$ to substitute $\Z_n$. This assumption implies the infinite dimensional Taft algebra is \emph{prime}. Put them together, the the infinite dimensional Taft algebra is prime regular of Gelfand-Kirillov dimension (GK-dimension for short) one. Under this assumption, one can show that the affine line $\k[x]$ and the multiplicative group $\k[x^{\pm1}]$ are the \emph{only} smooth curves which can be realized as Hopf algebras in $^{\mathbb{Z}_{n}}_{\mathbb{Z}_n}{\mathcal{YD}}$ (see Corollary \ref{c2.9}). Therefore, the only left chance is to consider singular curves. We find that at least for some special curves the answer is ``Yes"! As an illustration, consider the example $T(\{2,3\},1,\xi)$ (see Subsection \ref{ss4.1}) and from this example we find the the cusp $y_1^2=y_2^3$ is a Hopf algebra in $^{\mathbb{Z}_{6}}_{\mathbb{Z}_6}{\mathcal{YD}}$. That is,

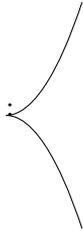
\begin{tikzpicture}
\begin{picture}(100,60)(-100,0)
\draw (0,0)  parabola (1,1.5);\put(60,0){\makebox(0,0){$\in {^{\mathbb{Z}_{6}}_{\mathbb{Z}_6}{\mathcal{YD}}}.$}}
\draw  (0,0)  parabola (1,-1.5);
\end{picture}
\end{tikzpicture}

So above analysis tell us that we need consider the structures of prime Hopf algebras of GK-dimension one which are \emph{not regular} if we want to find the answer to question \eqref{1.1}.

The second one is a wide range of recent researches and interest on the classification of Hopf algebras of finite GK-dimensions. See for instance \cite{AS, AAH, BZ, GZ, Liu, LWZ, WZZ1, WZZ2, WZZ3, WLD}. Up to the authors's knowledge, there are two different lines to classify such Hopf algebras.  One line focuses on pointed versions, in particular about braidings (i.e. Nichols algebras). The first celebrated work in this line is the Rosso's basic observation about the structure of Nichols algebras of finite GK-dimension with positive braiding (see \cite[Theorem 21.]{Ro}). Then the pointed Hopf algebra domains of finite GK-dimension with generic infinitesimal braiding were classified by Andruskiewitsch and Schneider \cite[Theorem 5.2.]{AS} and Andruskiewitsch and Angiono \cite[Theorem 1.1.]{AA}. Recently, Andruskiwwitsch-Angiono-Heckenberger \cite{AAH} conjectured that a Nichols algebra of diagonal type has finite GK-dimension if and only if the corresponding generalized root system is finite, and under assuming the validity of this conjecture they classified a natural class of braided spaces whose Nichols algebra has finite GK-dimension \cite[Theorem 1.10.]{AAH}.  Another line focuses more on algebraic and homological properties of these Hopf algebras, which is motivated by noncommutative algebras and noncommutative algebraic geometry. Historically,  Lu, Wu and Zhang initiated the the program of classifying Hopf algebras of GK-dimension one \cite{LWZ}. Then the author found a new class of examples about prime regular Hopf algebras of GK-dimension one \cite{Liu}. Brown and Zhang \cite[Theorem 0.5]{BZ} made further efforts in this direction and
classified all prime regular Hopf algebras $H$ of GK-dimension one
under an extra hypothesis. In 2016, Wu, Ding and the author \cite[Theorem 8.3]{WLD} removed this hypothesis and gave a complete classification prime regular Hopf algebras of GK-dimension one at last. One interesting fact is that some non-pointed Hopf algebras of GK-dimension one were found in \cite{WLD} and as far as we know they are the only non-pointed Hopf algebras with finite GK-dimension (except GK-dimension zero) until today. For Hopf algebras $H$ of GK-dimension two, all known classification results are given under the condition of $H$ being domains. In \cite[Theorem 0.1.]{GZ}, Goodearl and Zhang classified all Hopf algebras $H$ of
GK-dimension two which are domains and satisfy the condition $\Ext^{1}_{H}(\k, \k)\neq 0$. For those with vanishing Ext-groups, some interesting examples were constructed by Wang-Zhang-Zhuang
\cite[Section 2.]{WZZ2} and they conjectured these examples together with Hopf algebras given in \cite{GZ} exhausted all Hopf algebra domains with GK-dimension two. In order to study Hopf algebras $H$ of GK-dimensions three and four, a more restrictive condition was added: $H$ is connected, that is, the coradical of $H$ is $1$-dimensional.  All connected Hopf algebras with GK-dimension three and four were classified by Zhuang in \cite[Theorem 7.6]{Zh} and Wang, Zhang
and Zhuang \cite[Theorem 0.3.]{WZZ3} respectively. So, as a natural development of this line we want to classify prime Hopf algebras of GK-dimension one without regularity.

The third one is the lack of knowledge about non-pointed Hopf algebras. In the last two decades, the people achieved an essential progress in understanding the structures and even classifications of pointed Hopf algebras under many experts's, like Andruskiewitsch, Schneider, Heckenberger etc., efforts. See for example \cite{AS1,He,He1}.  On the contrast, we know very little about non-pointed Hopf algebras. In fact, we almost can't or are very hard to provide any nontrivial examples of them. The short of examples of non-pointed Hopf algebras obviously hampers our research and understanding of non-pointed Hopf algebras. Inspired by our previous work \cite{WLD} on the classification of prime regular Hopf algebras, which prompted us to find a series of new examples of non-pointed Hopf algebras, we expect to get more examples through classifying prime Hopf algebras of GK-dimension one without regularity.

\subsection{Setting.} As the research continues, we gradually realize that the condition ``regular" is very delicate and strong. The situation becomes much worse if we just remove the regularity condition directly. In another word, we still need some ingredients from regularity at present. To get suitable ingredients, let's go back to the question \eqref{1.1} and in such case the Hopf algebra has a natural projection to the group algebra $\k\Z_n$. The first question is: what is this natural number $n?$  In the Taft algebra $H$ case, it is not hard to see that this $n$ is just the PI degree of $H$, that is, $n=$PI.deg$(H)$. So crudely speaking $n$ measures how far is a Hopf algebra from a commutative one. At the same time, the Hopf algebra who has a projection to $\k\Z_n$ will have a 1-dimensional representation $M$ with order $n$, that is $M^{\otimes n}\cong \k$. Putting them together, we form our first hypothesis about prime Hopf algebras of GK-dimension one:
\textsf{ \begin{itemize}\item[$\;\;$(Hyp1):] The Hopf algebra $H$ has a $1$-dimensional representation $\pi:\;H\to \k$ whose order is equal to PI.deg$(H)$.\end{itemize}}
The second question is: where is the curve? It is not hard to see that the curve is exactly the coinvariant algebra under the projection to $\k\Z_n$. We will see that for each $1$-dimensional representation of $H$ one has an analogue of coinvariant algebras which are called the \emph{invariant components} with respect to this representation (see Subsection \ref{subs2.2} for details.) Due to the (Hyp1), our second hypothesis is:
\textsf{\begin{itemize}\item[(Hyp2):] The invariant components with respect to $\pi_{H}$ are domains.\end{itemize}}

By definition, a Hopf algebra $H$ we considered has two invariant components, that is the left invariant component $H_{0,\pi}^l$ and right invariant component $H_{0,\pi}^r$ (see Definition \ref{de1}). By Lemma \ref{l2.7}, we see  that $H_{0,\pi}^{l}$ is a domain if and only if $H_{0,\pi}^r$ is a domain. So the (Hyp2) can be weakened to require that any one of two invariant components is a domain. But, in practice (Hyp2) is more convenient for us.

Usually, one may wonder that (Hyp1) is strange and strong. Actually, any noetherian affine Hopf algebra $H$ has natural $1$-dimensional representations: the space of right (resp. left) homological integrals. The order of any one of these 1-dimensional modules is called the \emph{integral order} (see Subsection \ref{subs2.2} for related definitions) of $H$ and we denote it by $\io (H)$, which is used widely in the regular case. So a plausible alternative of (Hyp1) is

\textsf{(Hyp1)$'$ $\quad$ $\io(H)=$ PI.deg$(H)$.}

Clearly, (Hyp1)$'$ is stronger than (Hyp1) and should be easier to use (Hyp1)$'$ instead of (Hyp1). But we will see that the (Hyp1)$'$ is not so good because it excludes some nice and natural examples (see Remark \ref{r4.3}).

 Note that all prime regular Hopf algebras of GK-dimension one satisfy both (Hyp1)$'$ and (Hyp2) automatically (see \cite[Theorem 7.1.]{LWZ}). Since we have examples which satisfy (Hyp1) and (Hyp2) while they are not regular (see, say, the example about the cusp given above), regularity is a really more stronger than (Hyp1) $+$ (Hyp2) for prime Hopf algebras of GK-dimension one.

The main result of this paper is to give a classification of all prime Hopf algebras of GK-dimension one satisfying (Hyp1) $+$ (Hyp2) (see Theorem \ref{t7.1}). As byproducts, a number of new Hopf algebras, in particular some non-pointed Hopf algebras, were found and the answer to question \eqref{1.1} was given easily. Moreover, many new, up to the author's knowledge, finite-dimensional Hopf algebras were gotten which in particular helps us to find a Hopf algebra of dimension $24$  (see \cite{BGNR}).

\subsection{Strategy and organization.} In a word, the idea of this paper just is to build a ``relative version" (i.e. with respect to any $1$-dimensional representation rather than just the $1$-dimensional representation of homological integrals) and extend the methods of \cite{BZ,WLD} to our general setting. So the strategy of the proof of the main result is divided into two parts: the ideal case and the remaining case. However, we need point out that the most significant difference between the regular Hopf algebras of GK-dimension one and our setting is: In the regular case, the invariant components are Dedekind domains (see \cite[Theorem 2.5 (f)]{BZ}) while in our case they are just required to be general domains! At the first glance, there is a huge distance between a general domain and a Dedekind domain. A contribution of this paper is to overcome this difficulty and prove that we can classify these domains under the requirement that they are the invariant components of prime Hopf algebra of GK-dimension one. To overcome this difficulty, a new concept called \emph{a fraction of natural number} is introduced (see Definition \ref{d3.1}).

As the first step to realize our idea, we construct a number of new prime Hopf algebras of GK-dimension one which are called the ``fraction versions" of known examples of prime regular Hopf algebras of GK-dimension one. Then we use the concepts so called representation minor, denoted as $\im (\pi)$,  and representation order, denoted as $\ord (\pi)$, of a noetherian affine Hopf algebra $H$  to deal with the ideal case, that is, the case either $\im(\pi)=1$ or $\ord (\pi)=\im(\pi)$. In the ideal case, we proved that every prime Hopf algebras of GK-dimension one satisfying (Hyp1) $+$ (Hyp2) must be isomorphic to either a known regular Hopf algebra given in \cite[Section 3]{BZ} or a fraction version of one of these regular Hopf algebras. Then, we consider the remaining case, that is the case $\ord(\pi)>\im(\pi)>1$ (note that by definition $\im(\pi)|\ord(\pi)$). We show that for each prime Hopf algebra $H$ of GK-dimension one in the remaining case one always can construct a Hopf subalgebra $\widetilde{H}$ which lies in the ideal case. As one of difficult parts of this paper, we show that $\widetilde{H}$ indeed determine the structure of $H$ essentially and from which we can not only get a complete classification of prime Hopf algebras of GK-dimension one satisfying (Hyp1) $+$ (Hyp2) but also find a series of new examples of non-pointed Hopf algebras. At last, we give some applications of our results, in particular the questions \eqref{1.1} is solved, a partial solution to \cite[Question 7.3C.]{BZ} is given and  some new examples of finite dimensional Hopf algebras including semisimple and nonsemisimple Hopf algebras are found. In particular, we provide an example of 24-dimensional Hopf algebra, which seems not written out explicitly in \cite{BGNR}. Moreover,  at the end of the paper we formulate a conjecture (see Conjecture \ref{con7.19}) about the structure of a general prime Hopf algebra of GK-dimension one for further researches and considerations.

The paper is organized as follows. Necessary definitions, known examples and preliminary results are collected in Section 2. In particular, in order to compare regular Hopf algebras and non-regular ones, the widely used tool called homological integral is recalled. The definition of a fraction of natural number, a fraction version of a Taft algebra and some combinatorial relations, which are crucial to the following analysis, will be given in Section 3. Section 4 is devoted to construct new examples of prime Hopf algebras of GK-dimension one which satisfy (Hyp1) and (Hyp2). We should point out that the proof of the example $D(\underline{m},d,\gamma)$, which are not pointed in general, being a Hopf algebra is quite nontrivial. The properties of these new examples are also built in this section and in particular we show that they are pivotal Hopf algebras. The question about the classification of prime Hopf algebras of GK-dimension one satisfying (Hyp1) $+$ (Hyp2) in ideal cases is solved in Section 5, and Section 6 is designed to solve the same question in the remaining case. The main result is formulated in the last section and we end the paper with some consequences, questions and a conjecture on the structure of a general prime Hopf algebra of GK-dimension one. Among of them, a new kinds of semisimple Hopf algebras are found and studied. Their fusion rules are given. We also give another series of finite-dimensional nonsemisimple Hopf algebras in this last section.

\textbf{{Acknowledgments.}} The work started during my visiting to Department of Mathematics, MIT. I would like thank, from the bottom of my heart, Professor Pavel Etingof for his heuristic discussion, encouragements and hospitality. The author also want to thank Professor James Zhang for his continued help and support for the author, and in particular for showing him their examples of non-regular Hopf algebras given in Subsection \ref{sub4.2}. I appreciate Professors Ken Brown, Q.-S. Wu and D.-M Lu for useful communications and in particular thank Ken Brown for showing the author his nice slides on infinite dimensional Hopf algebras.

\section{Preliminaries}
In this section we recall the urgent needs around affine noetherian Hopf algebras for completeness and the convenience of the reader. About general background knowledge, the reader is referred to \cite{Mo} for Hopf algebras, \cite{MR} for noetherian rings, \cite{Br,LWZ,BZ,Go} for exposition about noetherian Hopf algebras and \cite{EGNO} for general knowledge of tensor categories.

Usually we are working on left modules (resp. comodules). Let $A^{op}$ denote the opposite algebra of $A$. Throughout, we use the symbols $\Delta,\epsilon$ and $S$ respectively, for the coproduct, counit and antipode of a Hopf algebra $H$, and the Sweedler's notation for coproduct $\D(h)=\sum h_1\otimes h_2=h_1\otimes h_2=h'\otimes h''\;(h\in H)$ will be used freely. Similarly, the coaction of left comodule $M$ is denoted by $\delta(m)=m_{(-1)}\otimes m_{(0)}\in H\otimes M,\;m\in M.$

 \subsection{Stuffs from ring theory and Homological integrals.}
In this paper, a ring $R$ is called \emph{regular} if it has finite global dimension, it is \emph{prime} if $0$ is a prime ideal and it is \emph{affine} if it is finitely generated.

{$\bullet$ \emph{PI-degree}.}  If $Z$ is an Ore domain, then the {\it rank } of a $Z$-module $M$ is defined to be the $Q(Z)$-dimension of $Q(Z)\otimes_Z M$, where $Q(Z)$ is the quotient
division ring of $Z$.  Let $R$ be an algebra satisfying a polynomial identity (PI for short). The PI-degree of $R$ is defined to be $$\text{PI-deg}(R)=\text{min}\{n|R\hookrightarrow M_n(C)\ \text{for some commutative ring}\ C\}$$(see \cite[Chapter 13]{MR}). If $R$ is a prime PI ring with center $Z$, then the PI-degree of $R$ equals the square root of the rank of
$R$ over $Z$.

{$\bullet$\emph{ Artin-Schelter condition}.}  Recall that an algebra $A$ is said to be \emph{augmented} if there is an algebra morphism $\epsilon:\; A\to \k$. Let $(A,\epsilon)$ be an augmented noetherian algebra. Then $A$ is \emph{Artin-Schelter Gorenstein}, we usually abbreviate to \emph{AS-Gorenstein}, if
\begin{itemize}
\item [(AS1)] injdim$_AA=d<\infty$,
\item [(AS2)] dim$_\k\Ext_A^d(_A\k, \;_AA)=1$ and
dim$_\k\Ext_A^i(_A\k,\;
_AA)=0$ for all $i\neq d$,
\item [(AS3)] the right $A$-module versions of (AS1, AS2) hold.\end{itemize}

The following result is the combination of \cite[Theorem 0.1]{WZ} and  \cite[Theorem 0.2  (1)]{WZ}, which shows that a large number of Hopf algebras are AS-Gorenstein.
\begin{lemma}\label{l2.1} Each affine noetherian  PI Hopf algebra is AS-Gorenstein.
\end{lemma}
{$\bullet$\emph{ Homological integral}.} The concept \emph{homological integral} can be defined for an AS-Gorenstein augmented algebra.
\begin{definition}\cite[Definition 1.3]{BZ}
\emph{Let $(A, \epsilon)$ be a noetherian augmented algebra and suppose that $A$ is  AS-Gorenstein of injective dimension $d$. Any non-zero element of the 1-dimensional $A$-bimodule $\Ext_A^d( _A\k,\; _AA)$ is called a \emph{left homological integral} of $A$. We write $\int_A^l=\Ext_A^d(_A\k,\; _AA)$. Any non-zero element in $\Ext_{A^{op}}^d(\k_A, A_A)$ is called a \emph{right homological integral} of $A$. We write $\int_A^r=\Ext_{A^{op}}^d(\k_A, A_A)$. By abusing the language we also call $\int_A^l$ and $\int_A^r$ the left and the right homological integrals of $A$ respectively.}
\end{definition}
\subsection{Relative version.}\label{subs2.2}
Assuming that a Hopf algebra $H$ has a $1$-dimensional representation $\pi: H\to \k$, we give some results according to this $\pi$, most of them coming from \cite[Section 2]{BZ}, by using slightly different notations with \cite{BZ}. Throughout this subsection, we fix this representation $\pi$.

\noindent{$\bullet$ \emph{Winding automorphisms}.}  We write $\Xi_\pi^l$ for the \emph{left winding automorphism} of $H$ associated to $\pi$, namely
$$\Xi_\pi^l(a):=\sum\pi(a_1)a_2 \;\;\;\;\;\;\;\textrm{for} \;a\in H.$$ Similarly we use $\Xi_\pi^r$ for the right winding automorphism of $H$ associated to $\pi$, that is,
$$\Xi_\pi^r(a):=\sum a_1\pi(a_2)\;\;\;\;\;\;\textrm{for}\; a\in H.$$

Let $G_\pi^l$ and $G_\pi^r$ be the subgroups of $\text{Aut}_{\k\text{-alg}}(H)$ generated by $\Xi_\pi^l$ and $\Xi_\pi^r$, respectively. Define:
$$G_\pi:=G_\pi^l\bigcap G_\pi^r.$$

The following is some parts of \cite[Propostion 2.1.]{BZ}.

\begin{lemma}\label{l2.3} Let $H_{0,\pi}^{l},H_{0,\pi}^{r}$ and $H_{0,\pi}$ be the subalgebra of invariants $H^{G_\pi^l},H^{G_\pi^r}$ and $H^{G_\pi}$ respectively. Then we have
\begin{itemize} \item[(1)] $H_{0,\pi}=H_{0,\pi}^{l}\bigcap H_{0,\pi}^{r}.$
\item[(2)] $\Xi_\pi^l\Xi_\pi^r=\Xi_\pi^r\Xi_\pi^l.$
\item[(3)] $\Xi_\pi^r\circ S=S\circ (\Xi_\pi^l)^{-1}$. Therefore, $S(H_{0,\pi}^{l})\subseteq H_{0,\pi}^{r}$ and $S(H_{0,\pi}^{r})\subseteq H_{0,\pi}^{l}$.
\end{itemize}
\end{lemma}

\emph{$\bullet$ $\pi$-order and $\pi$-minor.} With the same notions as above, the {\it $\pi$- order} (denoted as $\ord(\pi)$) of $H$ is defined by the order of the group $G_\pi^l$ :
\begin{equation}\ord(\pi):=|G_\pi^l|.\end{equation}
\begin{lemma} We always have $|G_\pi^l|=|G_\pi^r|$.
\end{lemma}
\begin{proof} Assume that $|G_\pi^l|=n$, and then by the definition we know that
$$a=\sum \pi^{n}(a_1)a_2$$
for all $a\in H$. Therefore, $\pi^{n}=\e$ (because above formula implies that $\pi^{n}$ is the left counit) and thus $a=\sum a_1\pi^{n}(a_2)$ for all $a$. So, $|G_\pi^l|\geq |G_\pi^r|$. Similarly, we have $|G_\pi^r|\geq |G_\pi^l|$.
\end{proof}
By this lemma,  the above definition is independent of the choice of $G_\pi^l$ or $G_\pi^r$.

The \emph{$\pi$-minor} (denoted by $\mi(\pi)$) of $H$ is defined by
\begin{equation}\mi(\pi):=|G_\pi^l/G_\pi^l\cap G_\pi^r|.\end{equation}

\begin{remark}\emph{In particular, if the 1-dimensional representation is given by the (right module structure) of left integrals, then the corresponding representation order and representation minor are called \emph{integral order} and \emph{integral minor}, denoted as
 $$\io(H)\quad\quad \textrm{and}\quad\quad \im(H),$$  respectively. Both the integral order and integral minor are used widely in \cite{BZ,WLD}. Therefore, we can consider a general $1$-dimensional representation as a relative version of homological integrals. Note that the notations $\io(H)$ and $\im(H)$ will be used freely in this paper too.}
\end{remark}

$\bullet$ \emph{Invariant components and strongly graded property}.
Let $H$ be a prime Hopf algebra of GK-dimension one. By a fundamental results of Small, Stafford and Warfield \cite{SSW}, a semiprime affine algebra of GK-dimension one is a finite module over its center. Therefore, it is PI and has finite PI-order. Now we assume that $H$ satisfies the (Hyp1) (see Subsection 1.1)  and  thereby $|G_\pi^l|=$ PI-deg$(H)$ is finite, say $n$. Moreover, since $G_\pi^l$ is a cyclic group, its character group $\widehat{G_\pi^l}:=\text{Hom}_{\k\text{-alg}}(\k G_\pi^l, \k)$ is isomorphic to itself. Similarly, the character group $\widehat{G_\pi^r}$ of $G_\pi^r$ is isomorphic to $G_\pi^r$.

%Let $H_0^l$ and $H_0^r$ be the subalgebras of invariants of
%$H^{G_\pi^l}$ and $H^{G_\pi^l}$, respectively.

Fix a primitive $n$th root $\zeta$ of 1 in $\k$, and define $\chi\in\widehat{G_\pi^l}$ and $\eta\in \widehat{G_\pi^r}$ by setting $$\chi(\Xi_\pi^l)=\zeta \quad \text{and} \quad
\eta(\Xi_\pi^r)=\zeta.$$ Thus $\widehat{G_\pi^l}=\{\chi^i|0\leqslant i\leqslant n-1\}$ and $ \widehat{G_\pi^r}=\{\eta^j|0\leqslant j\leqslant n-1\}$.

For each $0\leqslant i, j\leqslant n-1$, let
$$H_{i,\pi}^l:=\{a\in H|\Xi_\pi^l(a)=\chi^i(\Xi_\pi^l)a\} \;\;\textrm{and}\;\;
H_{j,\pi}^r:=\{a\in H|\Xi_\pi^r(a)=\eta^j(\Xi_\pi^r)a\}.$$

The  following lemma is  \cite[Theorem 2.5 (b)]{BZ} (Note that for the part (b) of \cite[Theorem 2.5.]{BZ} we don't need the condition about regularity).

\begin{lemma}\label{l2.5}
\begin{itemize} \item[(1)] $H=\bigoplus_{\chi^i\in\widehat{G_\pi^l}}H_{i,\pi}^l$ is strongly $\widehat{G_\pi^l}$-graded.
\item[(2)] $H=\bigoplus_{\eta^j\in\widehat{G_\pi^r}}H_{j,\pi}^r$ is strongly $\widehat{G_\pi^r}$-graded.
\end{itemize}
\end{lemma}

\begin{definition}\label{de1} \emph{The subalgebra $H_{0,\pi}^{l}$ (resp. $H_{0,\pi}^{r}$) is called the left (resp. right) \emph{invariant component} of $H$ with respect to $\pi$.}
\end{definition}

Therefore, (Hyp2) just says that both $H_{0,\pi}^{l}$ and $H_{0,\pi}^{r}$ are domains. In fact, these two algebras are closely related.
\begin{lemma}\label{l2.7} Let $H$ be a prime Hopf algebra of GK-dimension one. Then
\begin{itemize}\item[(1)] As algebras, we have $H_{0,\pi}^{l}\cong (H_{0,\pi}^r)^{op}$.
\item[(2)] If moreover either $H_{0,\pi}^{l}$ or $H_{0,\pi}^{r}$ is a domain, then both $H_{0,\pi}^{l}$ and $H_{0,\pi}^{r}$ are commutative domains and thus $H_{0,\pi}^{l}\cong H_{0,\pi}^r$.
\end{itemize}
\end{lemma}
\begin{proof} By Lemma \ref{l2.3}. (3), we have $S(H_{0,\pi}^{l})\subseteq H_{0,\pi}^{r}$ and $S(H_{0,\pi}^{r})\subseteq H_{0,\pi}^{l}$. Now (1) is proved.

For (2), it is harmless to assume that $H_{0,\pi}^l$ is a domain. By $H$ is of GK-dimension one and $H=\bigoplus_{\chi^i\in\widehat{G_\pi^l}}H_{i,\pi}^l$ is strongly graded (see Lemma \ref{l2.5}), $H_{0,\pi}^l$ has GK-dimension one too. Now it is well-known that a domain with GK-dimension one must be commutative (see for example \cite[Lemma 4.5]{GZ}). Therefore $H_{0,\pi}^l$ is commutative and  $H_{0,\pi}^{l}\cong H_{0,\pi}^r$ by (1). So $H_{0,\pi}^r$ is a commutative domain too.
\end{proof}

By Lemma \ref{l2.3}. (2), $\Xi_\pi^l \Xi_\pi^r=\Xi_\pi^r \Xi_\pi^l$, and thus $H_{i,\pi}^l$ is stable under the action of $G_\pi^r$. Consequently, the $\widehat{G_\pi^l}$- and $\widehat{G_\pi^r}$-gradings on $H$ are
\emph{compatible} in the sense that
$$H_{i,\pi}^l=\bigoplus_{0\leqslant j \leqslant n-1}(H_{i,\pi}^l\cap H_{j,\pi}^r)\quad \text{and}\quad H_{j,\pi}^r=\bigoplus_{0\leqslant i \leqslant n-1}(H_{i,\pi}^l\cap H_{j,\pi}^r)$$
for all $i, j$.
 Then $H$ is a bigraded algebra:
\begin{equation}\label{eq2.3} H=\bigoplus_{0\leqslant i,j\leqslant n-1} H_{ij,\pi},\end{equation}
where $H_{ij,\pi}=H_{i,\pi}^l\cap H_{j,\pi}^r$. And we write $H_{0,\pi}:=H_{00,\pi}$ for convenience.

For later use, we collect some more properties about $H$ which were proved in \cite{BZ} without the requirement about regularity. For details, see \cite[Proposition 2.1 (c)(e)]{BZ} and \cite[Lemma 6.3]{BZ}.
\begin{lemma}\label{l2.8} Let $H$ be a prime  Hopf algebra of GK-dimensional one satisfying (Hyp1). Then
\begin{itemize}
\item[(1)] $\Delta(H_{i,\pi}^l)\subseteq H_{i,\pi}^l\otimes H$ and $\Delta
(H_{j,\pi}^r)\subseteq H\otimes H_j^r$; thus $H_{i,\pi}^l$ is a right coideal of
$H$ and $H_{j,\pi}^r$ is a left coideal of $H$ for all $0\leq i,j\leq n-1$;

\item[(2)] $\Xi_\pi^r \circ S=S\circ(\Xi_\pi^l)^{-1},$
where $(\Xi_\pi^l)^{-1}=\Xi_{\pi\circ S}^l.$
\item[(3)] $S(H_{i,\pi}^l)=H_{-i,\pi}^{r}$ and $S(H_{ij,\pi})=H_{-j,-i,\pi}$.
\item[(4)] If $i\neq j$, then $\e(H_{ij,\pi})=0$.
\item[(5)] If $i=j$, then $\e(H_{ii,\pi})\neq 0.$
\end{itemize}
\end{lemma}

\begin{remark} \emph{(1) In the regular case, that is, $H$ is a prime regular Hopf algebra of GK-dimension one, the set of all right homological integrals forms a  $1$-dimensional representation whose order is equal to the PI.deg$(H)$. In such case, the invariant components are called \emph{classical components} by \cite[Section 2]{BZ}. }

\emph{(2) In the following of this paper, we will omit the notation $\pi$ when the representation is clear from context. Therefore, say, sometimes we just write $H_{0,\pi}$ as $H_0$ when there is no confusion about which representation we are considering.}
\end{remark}

The following result is the combination of some parts of \cite[Proposition 5.1, Corollary 5.1]{BZ}, which is very useful for us.
\begin{lemma}\label{l2.10} Let $A$ be a $\k$-algebra and let $G$ be a finite abelian group of order $n$ acting faithfully on $A$. So $A$ is $\widehat{G}$-graded, $A=\bigoplus_{\chi\in \widehat{G}}A_{\chi}$. Assume that \emph{1)} this grading is strong and \emph{2)} the invariant component $A_0$ is a commutative domain. Then we have
\begin{itemize}
\item[(a)] Every non-zero homogeneous element is a regular element of $A$ and PI.deg$(A)\leq n$.
\item[(b)] There is an action $\triangleright$ of $\widehat{G}$ on $A_0$ with the following property: For any $\chi\in \widehat{G}$ and $a\in A_0$,
    \begin{equation}\label{eq2.4}(\chi\triangleright a)u_{\chi} =u_{\chi}a\end{equation}
    where $u_{\chi}$ is an arbitrary nonzero element belonging to $A_{\chi}.$
\item[(c)] \emph{PI.deg}$(A)=n$ if and only if the action $\triangleright$ is faithful.
\item[(d)] If \emph{PI.deg}$(A)=n$, then $A$ is prime.
\item[(e)] Let $K< G$ be a subgroup $\widehat{G}$ and let $B$ be the subalgebra $\bigoplus_{\chi\in K}A_{\chi}$. If \emph{PI.deg}$(A)=n$, then $B$ is prime with PI-degree $|K|.$
\end{itemize}
\end{lemma}

\subsection{Known examples.}\label{ss2.3}

The following examples appeared in \cite{BZ,WLD} already and we recall them for completeness.

\noindent$\bullet$ \emph{Connected algebraic groups of dimension one}. It is well-known that there are precisely two connected algebraic groups of dimension one (see, say \cite[Theorem 20.5]{Hu}) over an algebraically closed field $\k$. Therefore, there are precisely two commutative $\k$-affine domains of GK-dimension one which admit a structure of Hopf algebra, namely $H_1=\k[x]$ and $H_2=\k[x^{\pm 1}]$. For $H_1$, $x$ is a primitive element, and for $H_2$, $x$ is a group-like element. Commutativity and cocommutativity imply that $\io(H_i)=\im(H_i)=1$ for $i=1, 2$.

\noindent$\bullet$ \emph{Infinite dihedral group algebra}. Let $\mathbb{D}$ denote the infinite dihedral group $\langle g, x | g^2 = 1, gxg=x^{-1}\rangle$. Both $g$ and $x$ are group-like elements in the group algebra $\k\mathbb{D}$. By cocommutativity, $\im(\k\mathbb{D})=1$. Using \cite[Lemma 2.6]{LWZ}, one sees that as a right $H$-module, $\int_{\k\mathbb{D}}^l\cong \k\mathbb{D}/\langle x-1, g+1\rangle.$ This implies $\io(\k\mathbb{D})=2$.

\noindent$\bullet$ \emph{Infinite dimensional Taft algebras}. Let $n$ and $t$ be integers with $n>1$ and $0\leqslant t \leqslant n-1$. Fix a primitive $n$th root $\xi$ of
$1$. Let $T=T(n, t, \xi)$ be the algebra generated by $x$ and $g$ subject to the relations $$g^n=1\quad \text{and} \quad xg=\xi gx.$$
Then $T(n, t, \xi)$ is a Hopf algebra with coalgebra structure given by
$$\D(g)=g\otimes g,\ \epsilon(g)=1 \quad \text{and} \quad \D(x)=x\otimes g^t+1\otimes x,\
\epsilon(x)=0,$$ and with $$S(g)=g^{-1}\quad \text{and} \quad S(x)=-xg^{-t}.$$

As computed in \cite[Subsection 3.3]{BZ}, we have $\int_T^l \cong T/\langle x, g-\xi^{-1}\rangle,$ and the corresponding homomorphism $\pi$ yields left and right winding automorphisms \[{\Xi_{\pi}^l:}
\begin{cases}
x\longmapsto x, &\\
g\longmapsto \xi^{-1}g, &
\end{cases} \textrm{and} \;\;\;\;\;
\Xi_{\pi}^r:
\begin{cases}
x\longmapsto \xi^{-t}x, &\\
g\longmapsto \xi^{-1}g. &
\end{cases}\]
So that $G_\pi^l=\langle \Xi_{\pi}^l\rangle$ and $G_\pi^r=\langle \Xi_{\pi}^r\rangle$ have order $n$. If gcd$(n, t)=1$, then $G_\pi^l\cap G_\pi^r=\{1\}$ and \cite[Propositon 3.3]{BZ} implies
that there exists a primitive $n$th root $\eta$ of 1 such that $T(n, t, \xi)\cong T(n, 1, \eta)$ as Hopf algebras. If gcd$(n,t)\neq 1$, let $m:=n/\text{gcd}(n, t)$, then $G_\pi^l\cap G_\pi^r=\langle (\Xi_{\pi}^l)^m\rangle$.
Thus we have $\io(T(n, t, \xi))=n$ and $\im(T(n, t, \xi))=m$ for any $t$. In particular, $\im(T(n, 0, \xi))=1$, $\im(T(n, 1, \xi))=n$ and $\im(T(n, t, \xi))=m=n/t$ when $t|n$.

\noindent$\bullet$ \emph{Generalized Liu algebras}. Let $n$ and $\omega$ be positive integers. The generalized Liu algebra, denoted by $B(n, \omega, \gamma)$, is generated by $x^{\pm 1}, g$ and $y$, subject to the relations
\begin{equation*}
\begin{cases}
xx^{-1}=x^{-1}x=1,\quad  xg=gx,\quad  xy=yx, & \\
yg=\gamma gy, & \\
y^n=1-x^\omega=1-g^n, & \\
\end{cases}
\end{equation*}
where  $\gamma$ is a primitive $n$th root of 1. The comultiplication, counit and antipode of $B(n, \omega, \gamma)$  are given by
$$\Delta(x)=x\otimes x,\quad  \Delta(g)=g\otimes g, \quad   \Delta(y)=y\otimes g+1\otimes y,$$
$$\epsilon(x)=1,\quad  \epsilon(g)=1,\quad    \epsilon(y)=0,$$ and
$$S(x)=x^{-1},\quad  S(g)=g^{-1} \quad S(y)=-yg^{-1}.$$

Let $B:=B(n, \omega, \gamma)$. Using \cite[Lemma 2.6]{LWZ}, we get $\int_B^l=B/\langle y, x-1, g-\gamma^{-1}\rangle$. The corresponding homomorphism $\pi$ yields left and right winding automorphisms
\[{\Xi_{\pi}^l:}
\begin{cases}
x\longmapsto x, &\\
g\longmapsto \gamma^{-1}g, &\\
y\longmapsto y, &
\end{cases}
\textrm{and}\;\;\;\; \Xi_{\pi}^r:
\begin{cases}
x\longmapsto x, &\\
g\longmapsto \gamma^{-1}g, &\\
y\longmapsto \gamma^{-1}y. &
\end{cases}\]
Clearly these automorphisms have order $n$ and $G_\pi^l\cap
G_\pi^r=\{1\}$, whence $\io(B)=\im(B)=n$.

\noindent$\bullet$ \emph{The Hopf algebras $D(m,d,\gamma)$}. Let $m,d$ be two natural numbers satisfying that $(1+m)d$ is even and $\gamma$ a primitive $m$th root of $1$. Define
$$\omega:=md,\;\;\;\;\xi:=\sqrt{\gamma}.$$
As an algebra, $D=D(m, d, \gamma)$ is generated by $x^{\pm 1}, g^{\pm 1}, y, u_0, u_1, \cdots,
u_{m-1}$, subject to the following relations
\begin{align*}
&xx^{-1}=x^{-1}x=1,\quad\quad gg^{-1}=g^{-1}g=1,\quad\quad xg=gx,\\
&xy=yx,\quad\quad\quad\quad\quad\quad yg=\gamma gy,\ \quad\quad\quad\quad\quad y^m=1-x^\omega=1-g^m,\\
& xu_i=u_ix^{-1},\ \quad \quad\quad\quad yu_i=\phi_iu_{i+1}=\xi x^d u_i y,\quad u_i g=\gamma^i x^{-2d}gu_i,\\
& u_iu_j=\left \{
\begin{array}{lll} (-1)^{-j}\xi^{-j}\gamma^{\frac{j(j+1)}{2}}\frac{1}{m}x^{-\frac{1+m}{2}d}\phi_i\phi_{i+1}\cdots \phi_{m-2-j}y^{i+j}g, & i+j \leqslant m-2,\\
(-1)^{-j}\xi^{-j}\gamma^{\frac{j(j+1)}{2}}\frac{1}{m}x^{-\frac{1+m}{2}d}y^{i+j}g, & i+j=m-1,\\
(-1)^{-j}\xi^{-j}\gamma^{\frac{j(j+1)}{2}}\frac{1}{m}x^{-\frac{1+m}{2}d}\phi_i \cdots \phi_{m-1}\phi_0\cdots
\phi_{m-2-j}y^{i+j-m}g, &
\textrm{otherwise},
\end{array}\right.\end{align*}
where $\phi_i=1-\gamma^{-i-1}x^d$ and $0 \leqslant i, j \leqslant m-1$.

The coproduct $\D$, the counit $\epsilon$ and the antipode $S$ of $D(m,d,\gamma)$ are given by
\begin{align*}
&\D(x)=x\otimes x,\;\; \D(g)=g\otimes g, \;\;\D(y)=y\otimes g+1\otimes y,\\
&\D(u_i)=\sum_{j=0}^{m-1}\gamma^{j(i-j)}u_j\otimes x^{-jd}g^ju_{i-j};\\
&\epsilon(x)=\epsilon(g)=\epsilon(u_0)=1,\;\;\epsilon(y)=\epsilon(u_s)=0;\\
&S(x)=x^{-1},\;\; S(g)=g^{-1}, \;\;S(y)=-yg^{-1},\\
&S(u_i)=(-1)^i\xi^{-i}\gamma^{-\frac{i(i+1)}{2}}x^{id+\frac{3}{2}(1-m)d}g^{m-i-1}u_i,
\end{align*}
for $0\leq i\leq m-1$ and $1\leqslant s\leqslant m-1$. Direct computation shows that $\int_D^l=D/(y, x-1,
g-\gamma^{-1}, u_0-\xi^{-1}, u_1, u_2, \cdots, u_{m-1}),$  and the left and right winding automorphisms are:
\[{\Xi_{\pi}^l:}
\begin{cases}
x\longmapsto x, &\\
y\longmapsto y, &\\
g\longmapsto \gamma^{-1}g, &\\
u_i\longmapsto \xi^{-1}u_i, &
\end{cases}
\textrm{and}\;\;\;\; \Xi_{\pi}^r:
\begin{cases}
x\longmapsto x, &\\
y\longmapsto \gamma^{-1}y. &\\
g\longmapsto \gamma^{-1}g, &\\
u_i\longmapsto \xi^{-(2i+1)} u_i. &
\end{cases}\]
From these, we know that $\io(D)=2m$ and $\im(D)=m$.

\begin{remark} In \cite{WLD}, the authors used the notation $D(m,d,\xi)$ rather than $D(m,d,\gamma)$ used here. We will see that the notation $D(m,d,\gamma)$ is more convenient for us.
\end{remark}

Up to an isomorphism of Hopf algebras, all of above examples form a complete list of prime regular Hopf algebras of GK-dimension one (see \cite[Theorem 8.3.]{WLD}).

\begin{lemma}\label{l2.12} Let $H$ be a prime regular Hopf algebra of GK-dimension one, then it is isomorphic to one of Hopf algebras listed above.
\end{lemma}
\subsection{Yetter-Drinfeld modules.}\label{ss2.4} This subsection is just a preparation for the question \eqref{1.1} and will not be used in the proof of our main result. Let $H$ be an arbitrary Hopf algebra. By definition, a \emph{left-left Yetter-Drinfeld module} $V$ over $H$ is a left $H$-module and a left $H$-comodule such that
$$\delta(h\cdot v)=h_1v_{(-1)}S(h_3)\otimes h_2\cdot v_{(0)}$$
for $h\in H, v\in V.$ The category of left-left Yetter-Drinfeld modules over $H$ is denoted by ${^{H}_{H}\mathcal{YD}}$. It is a braided tensor category. In particular, when $H=\k G$ a group algebra, we denote this category by ${^{G}_{G}\mathcal{YD}}.$

We briefly summarize results from \cite{Rad}, see also \cite{Maj}. Let $A$ be a Hopf algebra provided with Hopf algebra maps $\pi:\; A\to H.\;\iota: H\to A$, such that $\pi\iota=\id_{H}.$ Let $R=A^{coH}=\{a\in A|(\in\otimes \pi)\D(a)=a\otimes 1\}.$ Then $R$ is a braided Hopf algebra in ${^{H}_{H}\mathcal{YD}}$ through
 \begin{align*} &h\cdot r:= h_1rS(h_2),\\
 & r_{(-1)}\otimes r_{(0)}:=\pi(r_1)\otimes r_2,\\
 & r^1\otimes r^2:=\vartheta(r_1)\otimes r_2
 \end{align*} for $r\in R, \;h\in H$, $\D(r)=r^1\otimes r^2$ denote the coproduct of $r\in R$ in the category ${^{H}_{H}\mathcal{YD}}$ and $\vartheta(a):=a_1\iota\pi(S(a_2))$ for $a\in A.$

 Conversely, let $R$ be a Hopf algebra in ${^{H}_{H}\mathcal{YD}}$. A construction discovered by Radford, and interpreted in terms of braided tensor categories by Majid, produces a Hopf algebra $R\# H$ through: As a vector space $R\# H=R\otimes H;$ if $r\# h:=r\otimes h, \; r\in R, h\in H$, the multiplication and coproduct are given by
 \begin{align*} &(r\# h)(s\# f)=r(h_1\cdot s)\# h_2f,\\
 &\D(r\# h)= r^{1}\#(r^2)_{(-1)}h_{1}\otimes (r^2)_{(0)}\#h_2.
 \end{align*}
 The resulted Hopf algebra $R\# H$ is called a Radford's biproduct or Majid's bosonization.

Now go back to the situation of $\pi:\; A\to H.\;\iota: H\to A$ such that $\pi\iota=\id_H$. In such case we have $A\cong R\# H$ and
 \begin{equation}\label{eq2.5} r_{1}\otimes r_{2}=r^{1}(r^{2})_{(-1)}\otimes (r^{2})_{(0)}
 \end{equation} for $r\in R.$

With these preparations, we can set the question of \eqref{1.1} for smooth curves at first.

\begin{corollary}\label{c2.9} The affine line and $\k[x^{\pm 1}]$ are the only irreducible smooth curves which can be realized as Hopf algebras in ${^{\Z_n}_{\Z_n}{\mathcal{YD}}}$ for some $n$.
\end{corollary}
\begin{proof} Let $C$ be an irreducible smooth curve which can be realized as a Hopf algebra in ${^{\Z_n}_{\Z_n}{\mathcal{YD}}}$ for some $n$. There is no harm to assume that the action of $\Z_n$ on this curve (more precisely,  on the coordinate algebra $\k[C]$ of this curve) is faithful. Therefore, the Radford's biproduct $$A:=\k[C]\# \k \Z_n$$
constructed above is a Hopf algebra of GK-dimension one. We claim that it is  prime and regular. Primeness is gotten from Lemma \ref{l2.10}: Clearly
 $$A=\bigoplus_{i=0}^{n-1} \k[C]g^{i}.$$
 From this, $A$ is a strongly $\widehat{\Z_n}=\langle \chi|\chi^n=1 \rangle$-graded algebra through $\chi(ag^i)=\xi^{i}$ for any $a\in \k[C]$ and $0\leq i\leq n-1.$ Therefore, the conditions 1) and 2) of Lemma \ref{l2.10} are fulfilled. By part (b) of Lemma \ref{l2.10}, the action of $\widehat{\Z_n}$ is just the adjoint action of $\Z_{n}=\langle g|g^n=1\rangle$ on $\k[C]$ which by definition is faithful. Therefore, PI.deg$(A)=n$ by part (c) of Lemma \ref{l2.10}. In addition, the part (d) of Lemma \ref{l2.10} implies that $A$ is prime now. Regularity is clear since the smoothness of $C$ implies the regularity of $\k[C]$ and thus regularity of $A$. In one word, $A$ is a prime regular Hopf algebra of GK-dimension one.

   Therefore, the result is followed from above classification stated in Lemma \ref{l2.12} by checking it one by one.
\end{proof}
\subsection{Pivotal tensor categories.} The only purpose of this subsection is just to tell us that the representation categories of our new examples stated in Section 4 are quite delightful: they are pivotal. The readers can refer \cite[Section 4.7]{EGNO} for details of the following content of this subsection.

Recall that a tensor category $\mathcal{C}=(\mathcal{C},\otimes,\Phi,\unit, l,r)$ is called \emph{rigid} if every object in $\mathcal{C}$ has a left and a right dual. By definition, a left dual object of $V\in \mathcal{C}$ is a triple $(V^{*},\textrm{ev}_{V}, \textrm{coev}_{V})$ with an object $V^{*}\in \mathcal{C}$ and morphisms $\textrm{ev}_{V}:\; V^{*}\otimes V\to \unit$ and $\textrm{coev}_{V}:\;\unit\to V\otimes V^{*}$ such that the compositions
\begin{figure}[hbt]
\begin{picture}(300,60)(0,0)
\put(0,40){\makebox(0,0){$ V$}}
\put(10,40){\vector(1,0){60}}\put(40,50){\makebox(0,0){$ \textrm{coev}_V\otimes \id_V$}}
\put(105,40){\makebox(0,0){$(V\otimes V^{*})\otimes V$}} \put(140,40){\vector(1,0){20}}
\put(150,50){\makebox(0,0){$ \Phi$}}\put(200,40){\makebox(0,0){$ V\otimes(V^{*}\otimes V)$}}
\put(240,40){\vector(1,0){60}}\put(270,50){\makebox(0,0){$ \id_V \otimes \textrm{ev}_V$}}
\put(320,40){\makebox(0,0){$V$,}}

\put(0,0){\makebox(0,0){$ V^{*}$}}
\put(10,0){\vector(1,0){60}}\put(40,10){\makebox(0,0){$ \id_{V^{*}}\otimes \textrm{coev}_V $}}
\put(105,0){\makebox(0,0){$V^{*}\otimes (V\otimes V^{*})$}} \put(140,0){\vector(1,0){20}}
\put(150,10){\makebox(0,0){$ \Phi^{-1}$}}\put(200,0){\makebox(0,0){$ (V^{*}\otimes V)\otimes V^{*}$}}
\put(240,0){\vector(1,0){60}}\put(270,10){\makebox(0,0){$ \textrm{ev}_V\otimes \id_{V^{*}}$}}
\put(320,0){\makebox(0,0){$V^{*}$,}}\end{picture}
\end{figure}

are identities. The right dual can be defined similarly. Then we have the following functor $$(-)^{**}:\mathcal{C}\to \mathcal{C},\;\;V\mapsto V^{**}$$
which is a tensor autoequivalence of $\mathcal{C}.$
\begin{definition}\emph{ Let $\mathcal{C}$ be a rigid tensor category. A \emph{pivotal structure} on $\mathcal{C}$ is an isomorphism $j$ of tensor functors $j_{V}:V\mapsto V^{**}.$ A rigid tensor category $\mathcal{C}$ is said pivotal if it has a pivotal structure.}
\end{definition}

As nice properties of a pivotal tensor category, one can define categorical dimensions \cite[Section 4.7]{EGNO}, the Frobenius-Schur indicators \cite{NS}, semisimplifications \cite{EO} etc. The following result is well-known.

\begin{lemma}\label{l2.16} Let $H$ be a Hopf algebra. If $S^2(h)=ghg^{-1}$ for a group-like element $g\in H$ and any $h\in H$, then the representation category of $H$ is pivotal.
\end{lemma}
\begin{proof} Let Rep$(H)$ be the tensor category of representations of $H$. Clearly, the map $$V\to V^{**}=V,\;v\mapsto g\cdot v, \;\;\;V\in \Rep(H), \;v\in V$$
gives us the desired pivotal structure on Rep$(H)$.
\end{proof}

 \section{Fractions of a number}

 As a necessary ingredient to define new examples, we give the definition of a fraction of a natural number firstly in this section. Then we use it to ``fracture" the Taft algebra and thus we get the fraction version of a Taft algebra. At last, some combinatorial identities are collected for the future analysis.

 \subsection{Fraction.}  Let $m$ be a natural number and $m_1,m_2,\ldots,m_\theta$ be $\theta$ number of natural numbers. For each $m_i\;(1\leq i\leq \theta)$, we have many natural numbers $a$ such that $m|am_i$. Among of them, we take the smallest one and denote it by $e_i$, that is, $e_i$ is the smallest natural number such that $m|e_im_i.$  Define
  $$A:=\{\underline{a}=(a_1,\ldots,a_\theta)|0\leq a_i< e_i, \;1\leq i\leq \theta\}.$$
  With these notations, we give the definition of a fraction as follows.
 \begin{definition} \label{d3.1} \emph{We call $m_1,\ldots,m_\theta$ is} a fraction of $m$ of length $\theta$ \emph{if the following conditions are satisfied:}
 \begin{itemize}\item[(1)] \emph{For each $1\leq i\leq \theta$, $e_i$ is coprime to $m_i$, i.e. $(e_i,m_i)=1$;}
 \item[(2)] \emph{For each pair $1\leq i\neq j\leq \theta$, $m|m_im_j$;}
 \item[(3)] \emph{The production of $e_i$ is equal to $m$, that is, $m=e_1e_2\cdots e_\theta$;}
 \item[(4)] \emph{For any two elements $\underline{a},\underline{b}\in A$, we have  $\sum_{i=1}^{\theta}a_im_i\not\equiv \sum_{i=1}^{\theta}b_im_i\;(\emph{mod}\; m)$ if $\underline{a}\neq \underline{b}$.}
 \end{itemize}
 \emph{The set of all fractions of $m$ of length $\theta$ is denoted by $F_{\theta}(m)$ and let $\mathcal{F}(m):=\bigcup_{\theta}F_{\theta}(m),\;\mathcal{F}=\bigcup_{m\in \mathbb{N}}\mathcal{F}(m).$}
 \end{definition}

 \begin{remark}\label{r3.2} \emph{(1) Conditions (3) and (4) in this definition is equivalent to say that up to modulo $m$, each number $0\leq j\leq m-1 $ can be represented \emph{uniquely} as a linear combination of $m_1,\ldots,m_\theta$ with coefficients in $A$. That is, under basis $m_1,\ldots,m_\theta$, $j$ has a coordinate and we denote this coordinate by $(j_1,\ldots,j_\theta)$, i.e.}
 $$j\equiv j_1m_1+j_2m_2+\ldots+j_\theta m_\theta\;(\emph{mod}\;m).$$
 \emph{Moreover, for any $j\in \Z$ it has a unique remainder $\overline{j}$ in $\Z_{m}$ and thus we can define the coordinate for any integer accordingly, that is, $j_i:=\overline{j}_i$ for $1\leq i\leq \theta$.} In the following of this paper, this expression will be used freely.

 \emph{(2) For each $1\leq i\leq \theta$, we call $e_i$ \emph{the exponent} of $m_i$ with respect to $m$. Intuitively, it seems more natural to call these exponents $e_1,\ldots,e_\theta$ a fraction of $m$ due to the condition (3). However, there are as least two reasons forbidding us to do it. The first one is that we will meet $m_i$'s rather than $e_i$'s in the following analysis. The second reason is that the exponents can not determine $m_i$'s uniquely. As an example, let $m=6$, we see that both $\{2,3\}$ and $\{4,3\}$ have the same set of exponents. }

 \emph{(3) It is not hard to see that $\theta=1$ if and only if $(m,m_1)=1$.}

 \emph{(4) Usually, we use the notation such as $\underline{m},\underline{m}'\cdots$ to denote a fraction of $m$, that is, $\underline{m},\underline{m}'\in \mathcal{F}(m)$.}

 \end{remark}

 \subsection{Fraction version of a Taft algebra.}
Now let $m_1,\ldots,m_\theta$ be a fraction of $m$, $m_0:=(m_1,\ldots,m_{\theta})$ biggest common divisor of $m_1,\ldots,m_\theta$ and fix a primitive $m$th root of unity $\xi$. We want to define a Hopf algebra $T(m_1,\ldots,m_\theta,\xi)$ as follows. As an algebra, it is generated by $g,y_{m_1},\ldots,y_{m_\theta}$ and subject to the following relations:
 \begin{equation}\label{eq3.1} g^m=1,\;\;y_{m_i}^{e_i}=0,\;\;y_{m_i}y_{m_j}=y_{m_j}y_{m_i},\;\;
 y_{m_i}g=\xi^{\frac{m_i}{m_0}}gy_{m_i},
 \end{equation}
 for $1\leq i,j\leq \theta.$  The coproduct $\D$, the counit $\epsilon$ and the antipode $S$ of $T(m_1,\ldots,m_\theta,\xi)$ are given by
 $$\D(g)=g\otimes g,\quad \D(y_{m_i})=1\otimes y_{m_i}+y_{m_i}\otimes g^{m_i},$$
 $$\e(g)=1,\quad\e(y_{m_i})=0,$$
 $$S(g)=g^{-1},\quad S(y_{m_i})=-y_{m_i}g^{-m_i}$$
 for $ 1\leq i\leq \theta.$

Since $(m_0,m)=1$, if we take $\xi':=\xi^{m_0}$ in the above definition then it is not hard to see that $\xi'$ is still a primitive $m$th root of unity. So in \eqref{eq3.1} we can substitute the relation $y_{m_i}g=\xi^{\frac{m_i}{m_0}}gy_{m_i}$ by a more convenient version
$$y_{m_i}g=\xi^{m_i}gy_{m_i},\;\;\;\;1\leq i\leq \theta.$$

 \begin{lemma}\label{l3.3} The algebra $T(m_1,\ldots,m_\theta,\xi)$ defined above is an $m^2$-dimensional Hopf algebra.
 \end{lemma}
 \begin{proof} This is clear. We just point out that: The condition (1) of Definition \ref{d3.1} ensures that each $y_{m_i}^{e_i}$ is a primitive element and the condition (2) of Definition \ref{d3.1} ensures that $y_{m_i}y_{m_j}-y_{m_j}y_{m_i}$ is a skew-primitive element for all $1\leq i,j\leq \theta.$
 \end{proof}

 \begin{proposition}\label{p3.4} Let $m'$ be another natural number and $\underline{m'}=\{m_1',\ldots,m'_{\theta'}\}$ be a fraction of $m'$. Then as Hopf algebras, $T(m_1,\ldots,m_\theta,\xi)\cong T(m_1',\ldots,m_\theta',\xi')$ if and only if $m=m',\;\theta=\theta'$ and there exists $x_0\in \N$ which is relatively prime to $m$ such
 that up to an order of $m_1,\ldots,m_\theta$ we have $m_i'\equiv m_ix_0$ $($\emph{mod} $m)$ and $\xi=\xi'^{x_0}.$
 \end{proposition}
 \begin{proof} We denote the generators and numbers of $T(m_1',\ldots,m_\theta',\xi')$ by adding the symbol $'$ to that of $T(m_1,\ldots,m_\theta,\xi)$ for convenience. The sufficiency of the proposition is clear. We only prove the necessity. Assume that we have an isomorphism of Hopf algebras
 $$\varphi:\;T(m_1,\ldots,m_\theta,\xi)\stackrel{\cong}{\longrightarrow} T(m_1',\ldots,m_\theta',\xi').$$
 By this isomorphism, they have the same dimension and thus $m=m'$ according to Lemma \ref{l3.3}.
Comparing the number of nontrivial skew primitive elements, we know that $\theta=\theta'.$ Up to an order of $m_1,\ldots,m_\theta$, there is no harm to assume that $\varphi(y_{m_i})=y_{m_i'}$ for $1\leq i\leq \theta.$ (More precisely, we should take $\varphi(y_{m_i})=y_{m_i'}+c(1-(g')^{m_i'})$ at first. But through the relation $y_{m_i}g=\xi^{m_i}gy_{m_i}$ we have $c=0$.) Since $\varphi(g)$ is a group-like and generates all group-likes, $\varphi(g)=g'^{x_0}$ for some $x_0\in \N$ and $(x_0,m)=1$. Due to
$$\D(\varphi(y_{m_i}))=\D(y_{m_i'})=1\otimes y_{m_i'}+y_{m_i'}\otimes (g')^{m_i'}$$
which equals to
$$(\varphi\otimes \varphi)(\D(y_{m_i}))=1\otimes  y_{m_i'}+y_{m_i'}\otimes (g')^{m_ix_0}.$$
Therefore, $m_i'\equiv m_ix_0$ (mod $m$). By this, we can assume that $(m_1',\ldots,m_\theta')=(m_1,\ldots,m_\theta)x_0$, that is, $m_0'=m_0x_0.$
So $\varphi(y_{m_i}g)=\varphi(\xi^{\frac{m_i}{m_0}}gy_{m_i})$ implies that
$$\xi'^{\frac{m_i'}{m'_0}x_0}(g')^{x_0}y_{m_i'}=\xi^{\frac{m_i}{m_0}}(g')^{x_0}y_{m_i'}$$
which implies that $\xi^{\frac{m_i}{m_0}}=\xi'^{x_0\frac{m_i}{m_0}}$ for all $1\leq i\leq \theta.$ Since by definition $(\frac{m_1}{m_0},\ldots,\frac{m_\theta}{m_0})=1$, there exist $c_1,\ldots,c_{\theta}$ such that $\sum_{i=1}^{\theta}c_i\frac{m_i}{m_0}=1$. Therefore,
$$\xi=\xi^{\sum_{i=1}^{\theta}c_i\frac{m_i}{m_0}}=
\xi'^{x_0\sum_{i=1}^{\theta}c_i\frac{m_i}{m_0}}=\xi'^{x_0}.$$
 \end{proof}

\subsection{Some combinatorial identities.} Firstly, we will rewrite some combinatorial identities appeared in \cite[Section 3]{WLD} in a suitable form for our purpose. Secondly, we prove some more identities which are not included in \cite[Section 3]{WLD}.  Let $m,d$ be two natural numbers. As before, let $\underline{m}=\{m_1,\ldots,m_{\theta}\}\in \mathcal{F}(m)$ be a fraction of $m$ and $e_i$ the exponent of $m_i$ with respect to $m$ for $1\leq i\leq \theta.$ Let $\gamma$ be a primitive $m$th root of unity. By definition, we know that $$\gamma_i:=\gamma^{-{m_i^2}}$$ is a primitive $e_i$th root of unity. For any $j\in \Z$, the polynomial $\phi_{m_i,j}$ is defined through
\begin{equation}\phi_{m_i,j}:=1-\gamma^{-{m_i}(m_i+j)}x^{m_id}
=1-\gamma^{-m_i^2(1+j_i)}x^{m_id}=1-\gamma_{i}^{(1+j_i)}x^{m_id}\end{equation}
for any $1\leq i\leq \theta$ and the second equality is due to the (2) of the definition of the fraction. In the following of this subsection, we fix an $1\leq i\leq \theta.$

Take $j$ to be an arbitrary integer, define $\bar{j}$ to be the unique element in $\{0,1,\ldots,e_i-1\}$ satisfying $\bar{j}\equiv j$ {(mod}\,$e_i$). Then we have $$\phi_{m_i,j}=\phi_{m_i,\bar{j}}$$
since $\gamma_i^{e_i}=1$.

With this observation, we can use $${]s,t[_{m_i}}$$ to denote the resulted polynomial by omitting all items \emph{from} $\phi_{m_i,\overline{s}m_i}$ \emph{to} $\phi_{m_i,\overline{t}m_i}$ in $$\phi_{m_i,0}\phi_{m_i,m_i}\cdots \phi_{m_i,(e_i-1)m_i},$$ that is

\begin{equation}\label{eqomit} {]s,t[_{m_i}}=\begin{cases}
\phi_{m_i,(\bar{t}+1)m_i}\cdots \phi_{m_i,(e_i-1)m_i}\phi_{m_i,0}\cdots\phi_{m_i,(\bar{s}-1)m_i}, & \textrm{if}\; \bar{t}\geqslant \bar{s}
\\1, & \textrm{if}\; \bar{s}=\overline{t}+1 \\
\phi_{m_i,(\bar{t}+1)m_i}\cdots \phi_{m_i,(\bar{s}-1)m_i}, & \textrm{if}\;
\overline{s}\geqslant \bar{t}+2.
\end{cases} \end{equation}

For example, ${]{-1},{-1}[_{m_i}}={]{e_i-1},{e_i-1}[}_{m_i}=\phi_{m_i,0}\phi_{m_i,m_i}\cdots
\phi_{m_i,(e_i-2)m_i}$.

In practice, in particular to formulate the multiplication of our new examples of Hopf algebras, the next notation is also useful for us, which can be considered as the resulted polynomial (except the case $\bar{s}=\bar{t}+1$) by preserving all items \emph{from} $\phi_{m_i,\overline{s}m_i}$
\emph{to} $\phi_{m_i,\overline{t}m_i}$ in $\phi_{m_i,0}\phi_{m_i,m_i}\cdots \phi_{m_i,(e_i-1)m_i}$.

\begin{equation}\label{eqpre}  {[s,t]_{m_i}}:=\begin{cases}
\phi_{m_i,\bar{s}m_i}\phi_{m_i,(\bar{s}+1)m_i}\cdots \phi_{m_i,\bar{t}m_i}, & \textrm{if}\; \bar{t}\geqslant \bar{s}
\\1, & \textrm{if}\; \bar{s}=\overline{t}+1 \\
\phi_{m_i,\bar{s}m_i}\cdots \phi_{m_i,(e_i-1)m_i}\phi_{m_i,0}\cdots \phi_{m_i,\bar{t}m_i}, & \textrm{if}\;
\overline{s}\geqslant \bar{t}+2.
\end{cases} \end{equation}

So, by definition, we have
\begin{equation}\label{eqrel}
{[i, m-2-j]_{m_i}}={]{-1-j},{i-1}[_{m_i}}.
\end{equation}

Due to the equality \eqref{eqrel}, we just study equations with omitting items. The following formulas already were proved or already implicated in \cite[Section 3]{WLD} in different forms. So we just state them in our forms without proofs.

\begin{lemma}\label{l3.4} With notions defined as above, we have
\begin{itemize}
\item[(1)]$\sum_{j=0}^{e_i-1}{]{j-1},{j-1}[_{m_i}}\;=e_i.$
\item[(2)] $\phi_{m_i,0}\phi_{m_i,m_i}\cdots \phi_{m_i,(e_i-1)m_i}=1-x^{e_im_id}.$
\item[(3)] $\sum_{j=0}^{e_i-1}\gamma_i^{j}\;{]{j-1},{j-1}[_{m_i}}\;=e_ix^{(e_i-1)m_id}.$
\item[(4)] $\sum_{j=0}^{e_i-1}\gamma_i^{j}\;{]{j-2},{j-1}[_{m_i}}\;=0.$
\item[(5)] Fix $k$ such that $1\leqslant k\leqslant e_i-1$ and let $1\leqslant i'\leqslant k$. Then
$$\sum_{j=0}^{e_i-1}\gamma_i^{i'j}\;{]{j-1-k},{j-1}[_{m_i}}=0.$$
\item[(6)] Let $0\leq t\leq j+l\leq e_i-1,\; 0\leq \alpha \leq e_i-1-j-l$. Then
\begin{align*}
&\quad (-1)^{\alpha+t}\gamma_i^{\frac{(\alpha+t)(\alpha+t+1)}{2}+t(j+l-t)}
\binom{e_i-1-t}{\alpha}_{\gamma_i}\binom{e_i-1+t-j-l}{\alpha+t}_{\gamma_i} \\
&=\binom{j+l}{t}_{\gamma_i}\binom{m-1-j-l}{\alpha}_{\gamma_i}.
\end{align*}
\end{itemize}
\end{lemma}

We still need two more observations which were not included in \cite[Section 3]{WLD}.

\begin{lemma}\label{l3.5} With notations as above. Then
\begin{itemize} \item[(1)] For any $e_i$th root of unity $\xi$, we have
$$\sum_{j=0}^{e_i-1}\xi^{j}\;{]{j-1},{j-1}[_{m_i}}\neq 0.$$
\item[(2)] Let $\xi$ be an $e_i$th root of unity. Then $\sum_{j=0}^{e_i-1}\xi^{j}\;{]{j-2},{j-1}[_{m_i}}=\;0$ if and only if $\xi=\gamma_i.$
\end{itemize}
\end{lemma}
\begin{proof} (1) Otherwise, we assume that $\sum_{j=0}^{e_i-1}\xi^{j}\;{]{j-1},{j-1}[_{m_i}}= 0.$ From this, we know that $\xi\neq 1$ by (3) of Lemma \ref{l3.4}. By the definition of $]{j-1},{j-1}[_{m_i}$, we know that
\begin{align*}
& \quad\sum_{j=0}^{e_i-1}\;\xi^{j}{]{j-1},{j-1}[_{m_i}}\;-
\sum_{j=0}^{e_i-1}\xi^{j}\gamma_i^{j}x^{m_id}\;{]{j-1},{j-1}[_{m_i}}\;\\
&=\sum_{j=0}^{e_i-1}\xi^{j}(1-\gamma_i^{j}x^{m_id})\;{]{j-1},{j-1}[_{m_i}}\;\\
&=\sum_{j=0}^{e_i-1}\xi^{j}\phi_{m_i,0}\phi_{m_i,m_i}\cdots\phi_{m_i,(e_i-1)m_i}\\
&=\sum_{j=0}^{e_i-1}\xi^{j}(1-x^{e_im_id})\\
&=0.
\end{align*}
where the third equality is due to (2) of Lemma \ref{l3.4} and the last equality follows from $\xi\neq 1$ being an $e_i$th root of unity. Therefore, $\sum_{j=0}^{e_i-1}\xi^{j}\gamma_i^{j}x^{m_id}\;{]{j-1},{j-1}[_{m_i}}=0$ and thus
$\sum_{j=0}^{e_i-1}(\gamma_i\xi)^{j}\;{]{j-1},{j-1}[_{m_i}}=0$. Repeat above process, we know that for any $k$
$$\sum_{j=0}^{e_i-1}(\gamma_i^k\xi)^{j}\;{]{j-1},{j-1}[_{m_i}}=0.$$
 Since $\xi$ is an $e_i$th root of unity while $\gamma_i$ is a primitive $e_i$th root of unity, there exists a $k$ such that $\gamma_i^k\xi=1$. But in this case $\sum_{j=0}^{e_i-1}(\gamma_i^k\xi)^{j}\;{]{j-1},{j-1}[_{m_i}}=e_i\neq 0.$
 That is a contradiction.

 (2) ``$\Leftarrow$" This is just the (4) of Lemma \ref{l3.4}.

 ``$\Rightarrow$" Before prove this part, we recall a formula (see \cite[Proposition IV.2.7]{Kas}) at first: $$(a-z)(a-qz)\cdots (a-q^{n-1}z)=\sum_{l=0}^{n}(-1)^l\binom{n}{l}_q q^{\tfrac{l(l-1)}{2}}a^{n-l}z^l,$$
 where $q$ is a nonzero element in $\k$ and any $a\in \k$. From this,
 \begin{align*}
{]{j-2},{j-1}[_{m_i}}&=(1-\gamma_i^{j+1}x^{m_id})(1-\gamma_i^{j+2}x^{m_id})\cdots(1-\gamma_i^{e_i+j-2}x^{m_id})\\
               &=\sum_{l=0}^{e_i-2}(-1)^l\binom{e_i-2}{l}_{\gamma_{i}}\gamma_i^{\tfrac{l(l-1)}{2}}(\gamma_i^{j+1}x^{m_id})^l\\
               &=\sum_{l=0}^{e_i-2}(-1)^l\binom{e_i-2}{l}_{\gamma_i}\gamma_i^{\tfrac{l(l+1)}{2}+lj}x^{lm_id}.
\end{align*}
So from this, we have
\begin{align*} \sum_{j=0}^{e_i-1}\xi^{j}\;{]{j-2},{j-1}[_{m_i}}&
=\sum_{l=0}^{e_i-2}(-1)^l\binom{e_i-2}{l}_{\gamma_i}\gamma_i^{\tfrac{l(l+1)}{2}}
\sum_{j=0}^{e_i-1}\xi^{j}\gamma_i^{lj}x^{lm_id}.
\end{align*}

Therefore assumption implies that
$$\sum_{j=0}^{e_i-1}\xi^{j}\gamma_i^{lj}=0$$
for all $0\leqslant l \leqslant e_i-2$. So we see that the only possibility is $\xi=\gamma_i$.
\end{proof}

\section{New examples}
In this section, we will introduce the fraction versions of infinite dimensional Taft algebras, generalized Liu algebras and the Hopf algebras $D(m,d,\gamma)$ respectively. Some properties of them are listed. Most of these Hopf algebras, as far as we know, are knew.
\subsection{Fraction of infinite dimensional Taft algebra $T(\underline{m},t,\xi)$. }\label{ss4.1}

 Let $m,t$ be two natural numbers and set $n=mt$. Let $\underline{m}=\{m_1,\ldots,m_\theta\}$ be a fraction of $m$ and $m_0=(m_1,\ldots,m_{\theta})$ the greatest common divisor. So it is not hard to see that $(m,m_0)=1$.  Now fix a primitive $n$th root of unity $\xi$ satisfying
   $$\xi^{e_1\frac{m_1}{m_0}}=\xi^{e_2\frac{m_2}{m_0}}=\cdots=\xi^{e_\theta \frac{m_\theta}{m_0}}.$$
 Note that such $\xi$ does not always exist (for example, taking $m=6,\;t=2$ and $\{4,3\}$ be a fraction of $6$, we find that we have no such $\xi$). If it exists, then we can define a Hopf algebra $T(\underline{m},t,\xi)$ as follows. As an algebra, it is generated by $g,y_{m_1},\ldots,y_{m_\theta}$ and subject to the following relations:
 \begin{equation} g^n=1,\;\;y_{m_i}^{e_i}=y_{m_j}^{e_j},\;\;y_{m_i}y_{m_j}=y_{m_j}y_{m_i},\;\;
 y_{m_i}g=\xi^{\frac{m_i}{m_0}}gy_{m_i},
 \end{equation}
 for $1\leq i,j\leq \theta.$  The coproduct $\D$, the counit $\epsilon$ and the antipode $S$ of $T(\underline{m},t,\xi)$ are given by
 $$\D(g)=g\otimes g,\quad \D(y_{m_i})=1\otimes y_{m_i}+y_{m_i}\otimes g^{tm_i},$$
 $$\e(g)=1,\quad\e(y_{m_i})=0,$$
 $$S(g)=g^{-1},\quad S(y_{m_i})=-y_{m_i}g^{-tm_i}$$
 for $ 1\leq i\leq \theta.$

 \begin{proposition}\label{p4.1} Let the $\k$-algebra $T=T(\{m_1,\ldots,m_\theta\}, t, \xi)$ be the algebra defined as above. Then \begin{itemize}
\item[(1)] The algebra $T$ is a Hopf algebra of GK-dimension one, with center $\k[y_{m_1}^{e_1t}]$.
\item[(2)] The algebra $T$ is prime and PI-deg $(T)=n$.
\item[(3)] The algebra $T$ has a $1$-dimensional representation whose order is $n$. \end{itemize}
\end{proposition}
 \begin{proof} (1) Since the proof of $T(\underline{m},t,\xi)$ being a Hopf algebra is routine, we leave it to the readers. (In fact, since for each $1\leq i\leq \theta$ the subalgebra generated by $g, y_{m_i}$ is just a generalized infinite dimensional Taft algebra, one can reduce the proof to just considering the mixed relation $y_{m_i}y_{m_j}=y_{m_j}y_{m_i}$ and $y_{m_i}^{e_i}=y_{m_j}^{e_j}$ for $1\leq i,j\leq \theta$.)
 Through direct computations, one can see that the subalgebra $\k[y_{m_1}^{e_1t}]\cong \k[x]$ is the center of $T(\underline{m},t,\xi)$ and $T$ is finite module over $\k[y_{m_1}^{e_1t}]$. This means the GK-dimension of $T(\underline{m},t,\xi)$ is one.

 (2) We want to apply Lemma \ref{l2.10} to prove this result and we use similar argument developed in the proof of Corollary \ref{c2.9}. At first, let $T_0$ be the subalgebra generated by $y_{m_1},\ldots,y_{m_\theta}$. Then clearly
 $$T=\bigoplus_{i=0}^{n-1} T_{0}g^{i}.$$
 From this, $T$ is a strongly $\widehat{\Z_n}=\langle \chi|\chi^n=1 \rangle$-graded algebra through $\chi(ag^i)=\xi^{i}$ for any $a\in T_0$ and $0\leq i\leq n-1.$ Therefore, the conditions 1) and 2) of Lemma \ref{l2.10} are satisfied. By part (b) of Lemma \ref{l2.10}, the action of $\widehat{\Z_n}$ is just the adjoint action of $\Z_{n}=\langle g|g^n=1\rangle$ on $T_0$ which by definition is faithful. Therefore, PI.deg$(T)=n$ by part (c) of Lemma \ref{l2.10}. In addition, the part (d) of Lemma \ref{l2.10} implies that $T$ is prime now.

 (3)  By the definition of $T(\underline{m},t,\xi)$, it has a $1$-dimensional representation $$\pi:\;T(\underline{m},t,\xi)\to \k, \;\;y_{m_i}\mapsto 0,\;g\mapsto \xi\;\;(1\leq i\leq \theta).$$
 It's order is clear $n$.
 \end{proof}

 \begin{remark}\label{r4.3} \emph{
 We call the representation in Proposition \ref{p4.1} (c)  the \emph{canonical representation} of $T(\underline{m},t,\xi)$. Since $\ord(\pi)=n$ which is same as the PI-degree of $T(\underline{m},t,\xi)$, the Hopf algebra $T(\underline{m},t,\xi)$ satisfies the (Hyp1). At the same time, let $\{2,5\}$ be a fraction of $10$ and consider the example $T=T(\{2,5\},3,\xi)$ where $\xi$ is a primitive $30$th root of unity. Applying \cite[Lemma 2.6]{LWZ}, we find that the right module structure of the left homological integrals is given by
 $$\int_{T}^{l}=T/(y_{m_i}\;(1\leq i\leq \theta),g-\xi^{10-7}).$$
 Therefore $\io(T)=10$ which does \emph{not} equal the PI-degree of $T$, which is $30$. So, $T(\underline{m},t,\xi)$ only satisfies (Hyp1) rather than (Hyp1)$'$, that is, $\io(T)\neq$ PI.deg$(T)$ in general.}
 \end{remark}

 The canonical representation of $T=T(\underline{m},t,\xi)$ yields the corresponding left and right winding automorphisms
 \[{\Xi_{\pi}^l:}
\begin{cases}
y_{m_i}\longmapsto y_{m_i}, &\\
g\longmapsto \xi g, &
\end{cases} \textrm{and} \;\;\;\;\;
\Xi_{\pi}^r:
\begin{cases}
y_{m_i}\longmapsto \xi^{m_it}y_{m_i}, &\\
g\longmapsto \xi g, &
\end{cases}\]
for $1\leq i\leq \theta$.

Using above expression of $\Xi_{\pi}^{l}$ and $\Xi_{\pi}^{r}$, it is not difficult to find that
 \begin{equation} T_{i}^{l}=\k[y_{m_1},\ldots,y_{m_\theta}]g^{i}\;\;\;\;\textrm{and} \;\;\;\; T_{j}^{r}=\k[g^{-m_1t}y_{m_1},\ldots,g^{-m_{\theta}t}y_{m_\theta}]g^{j}
 \end{equation}
 for all $0\leq i,j\leq n-1.$ Thus we have
 \begin{equation} T_{00}=\k[y_{m_1}^{e_1}]\;\;\;\;\textrm{and} \;\;\;\;T_{i,i+jt}=\k[y_{m_1}^{e_1}]y_{j}g^{i}
 \end{equation}
 for all $0\leq i\leq n-1$, $0\leq j\leq m-1$ where $y_{j}=y_{m_1}^{j_1}\cdots y_{m_\theta}^{j_\theta}$ (see (1) of Remark \ref{r3.2}). Moreover, we can see that
 $$T_{ij}=0\;\;\textrm{if}\;\; i-j\not\equiv 0\;(\textrm{mod}\; t)$$
 for all $0\leq i,j\leq n-1.$

 As a concluding remark of this subsection, we want to discriminate these fractions of infinite dimensional Taft algebras.

 \begin{proposition}\label{p4.3} Keep above notations. Let $\underline{m'}=\{m_1',\ldots,m'_{\theta'}\}$ be a fraction of another integer $m'$. Then $T(\underline{m},t, \xi)\cong T(\underline{m'},t',\xi')$ if and only if $m=m',\;\theta=\theta',\;t=t'$ and there exists $x_0\in \N$ which is relatively prime to $n=mt$ such
 that up to an order of $m_1,\ldots,m_\theta$ we have $m_i'\equiv m_ix_0$ $($\emph{mod} $n)$ and $\xi=\xi'^{x_0}.$
 \end{proposition}
 \begin{proof} We write the proof out for completeness. We denote the corresponding generators and numbers of $T(\underline{m'},t',\xi')$ by adding the symbol $'$ to that of $T(\underline{m},t, \xi)$.  The sufficiency is clear (for example, just take $\varphi:\;T(\underline{m},t, \xi)\to T(\underline{m'},t',\xi')$ through $g\mapsto g'^{x_0},\;\;y_{m_i}\mapsto y'_{m_i'}$ for $1\leq i\leq \theta$. Then one can $\varphi$ gives the desired isomorphism). We next prove the necessity. Assume that we have an isomorphism of Hopf algebras
 $$\varphi:\;T(\underline{m},t, \xi)\stackrel{\cong}{\longrightarrow} T(\underline{m'},t',\xi').$$
 By this isomorphism, they have the same number of group-likes which implies that $n=mt=m't'=n'$ and $\varphi(g)=(g')^{x_0}$ for some $x_0\in \N$ satisfying $x_0$ and $n$ are coprime.
Comparing the number of nontrivial skew primitive elements, we know that $\theta=\theta'.$ Up to an order of $m_1,\ldots,m_\theta$, there is no harm to assume that $\varphi(y_{m_i})=y_{m_i'}$ for $1\leq i\leq \theta.$ (Just as the case of a fraction of a Taft algebra, one should take $\varphi(y_{m_i})=y_{m_i'}+c_i(1-(g')^{m_i'})$ at the beginning for some $c_i\in \k$. Then through the relation $y_{m_i}g=\xi^{\frac{m_i}{m_0}}gy_{m_i}$ we can find that $c_i=0.$) Since both $y_{m_i}^{e_i}$ and $y_{m_i'}^{e_i'}$ are primitive, $e_i=e_i'.$ Therefore $m=e_1\cdots e_{\theta}=e'_1\cdots e'_{\theta}=m'$ and thus $t=t'.$ Then one can repeat the proof of Proposition \ref{p3.4} and get that $m_i'\equiv m_ix_0$ (mod $n$) and $\xi=\xi'^{x_0}.$
 \end{proof}

 \subsection{$T(\underline{m},t,\xi)$ vs the Brown-Goodearl-Zhang's example.}\label{sub4.2} In the paper of Goodeal and Zhang \cite[Section 2]{GZ}, they found a new kind of Hopf domains of GK-dimension two. From these Hopf domains, one can get some Hopf algebras of GK-dimension one through quotient method. In fact, through this way Brown and Zhang \cite[Example 7.3]{BZ} got the first example of a prime Hopf algebra of GK-dimension one which is not regular. Let's recall their construction at first.

 \begin{example}[Brown-Goodearl-Zhang's example] \emph{Let $n,p_0,p_1,\ldots,p_s$ be positive integers and $a\in \k^{\times}$ with the following properties:
 \begin{itemize}\item[(a)] $s\geq 2$ and $1<p_1<p_2<\cdots<p_s;$
 \item[(b)] $p_0|n$ and $p_0,p_1,\ldots,p_s$ are pairwise relatively prime;
 \item[(c)] $q$ is a primitive $l$th root of unity, where $l=(n/p_0)p_1p_2\cdots p_s.$
 \end{itemize}
 Set $m_i=p_i^{-1}\prod_{j=1}^{s}p_j$ for $i=1,\ldots,s$. Let $A$ be the subalgebra of $\k[y]$ generated by $y_i:=y^{m_i}$ for $i=1,\ldots,s$. The $\k$-algebra automorphism of $\k[y]$ sending $y\mapsto qy$ restricts to an algebra automorphism $\sigma$ of $A$. There is a unique Hopf algebra structure on the Laurent polynomial ring $B=A[x^{\pm 1};\sigma]$ such that $x$ is group-like and the $y_i$ are skew primitive, with
 $$\D(y_i)=1\otimes y_i+y_i\otimes x^{m_i n}$$
 for $i=1,\ldots,s$. It is a PI Hopf domain of GK-dimension two, and is denoted by $B(n,p_0,p_1,\ldots,p_s,q)$. Now let $$\overline{B}(n,p_0,p_1,\ldots,p_s,q):=B(n,p_0,p_1,\ldots,p_s,q)/(x^l-1).$$
 Then Brown-Zhang proved that the quotient Hopf algebra $\overline{B}(n,p_0,p_1,\ldots,p_s,q)$ is a prime Hopf algebra of GK-dimension one.}
 \end{example}

 There is a close relationship between the Brown-Goodearl-Zhang's example and the fractions of infinite dimensional Taft algebras.
 \begin{proposition} The Hopf algebra $\overline{B}(n,p_0,p_1,\ldots,p_s,q)$ is a fraction of an infinite dimensional Taft algebra, that is, $\overline{B}(n,p_0,p_1,\ldots,p_s,q)=T(\underline{m},t,\xi)$ for some $\underline{m}\in \mathcal{F},t\in \mathbb{N}$ and $\xi$ a root of unity.
 \end{proposition}
 \begin{proof} By definition of $\overline{B}=\overline{B}(n,p_0,p_1,\ldots,p_s,q)$, we know that $y_i=y^{m_i}$ (we also use the same notation as $B(n,p_0,p_1,\ldots,p_s,q)$) and thus the following relation is satisfied
 $$y_i^{p_i}=y_j^{p_j}$$
 for all $1\leq i,j\leq s.$ At the same time, in $\overline{B}$ the group like element $x$ satisfying the following relations
 $$x^{l}=1, \;\;y_ix=q^{m_i}xy_i$$
 for $i=1,\ldots,s.$ By these observations, define
 $$m_i':=p_0m_i,\;\;\;\;1\leq i\leq s.$$
 Then it is tedious to show that $m_1',m_2',\ldots,m_s'$ is a fraction of $m:=\prod_{i=1}^{s} p_i$. Moreover, let $t:=n/p_0$. Now we see that the Hopf algebra $T(\{m_1',m_2',\ldots,m_s'\},t,q)$ is generated by $y_{m_1'},\ldots,y_{m_s'},\;g$ and satisfies the following relations
 $$g^l=1,\;\;y_{m_i'}^{p_i}=y_{m_j'}^{p_j},\;\;y_{m_i'}y_{m_j'}=y_{m_j'}y_{m_i'},\;\;
 y_{m_i'}g=q^{\frac{m_i'}{p_0}}gy_{m_i'}=q^{m_i}gy_{m_i'}.$$
 From this, there is an algebra epimorphism
 $$f:\;T(\{m_1',m_2',\ldots,m_s'\},n/p_0,q)\to \overline{B}(n,p_0,p_1,\ldots,p_s,q),\;\;
 y_{m_i'}\mapsto y_i, \;g\mapsto x$$
  which is clear a Hopf epimorphism. Since both of them are prime of GK-dimension one, $f$ must be an isomorphism.
 \end{proof}

 But not all fractions of infinite dimensional Taft algebras belong to the class of Brown-Goodearl-Zhang's examples.
 \begin{example} \emph{Let $5,12$ be a fraction of $30$ and $\xi$ a primitive $30$th root of unity. Then the corresponding $T(\{12,5\},1,\xi)$ is generated by $y_{5},y_{12},g$ satisfying
 $$y_{12}^{5}=y_{5}^{6},\;\;y_{12}y_{5}=y_{5}y_{12},\;\;y_{12}g=\xi^{12}gy_{12},
 \;\;y_{5}g=\xi^{5}gy_{5},\;\;g^{30}=1.$$
 If there is an isomorphism between this Hopf algebra and a Brown-Goodearl-Zhang's example $$f:\; T(\{12,5\},1,\xi)\stackrel{\cong}{\longrightarrow}\overline{B}(n,p_0,p_1,\ldots,p_s,q),$$ then clearly $s=2$ (by the number of non-trivial skew primitive elements) and $l=(n/{p_0})p_1p_2=30$ (due to they have the same group of group-likes). Therefore, $f(g)=x^{t}$ with $(t,30)=1$. By
 $$\D(y_5)=1\otimes y_5+g^{5}\otimes y_5,\;\;\;\;\D(y_{12})=1\otimes y_{12}+g^{12}\otimes y_{12},$$
 we know that $np_1\equiv 5t, np_2\equiv 12t$ (mod $30$). Since $p_1,p_2$ are factors of $30$ and $t$ is coprime to $30$, $p_1=5$ and thus $n\equiv t$ (mod $30$), $p_2=12$. This contradicts to $l=(n/{p_0})p_1p_2=30$.}
 \end{example}

This example also shows that not every fraction version of infinite dimensional Taft algebra can be realized as a quotient of a Hopf domain of GK-dimension two.

\subsection{Fraction of generalized Liu algebra $B(\underline{m},\omega,\gamma)$.}\label{ss4.3}
Let $m,\omega$ be positive integers and $m_1,\ldots,m_\theta$ a fraction of $m$. A fraction of a generalized Liu algebra, denoted by $B(\underline{m}, \omega, \gamma)=B(\{m_1,\ldots,m_\theta\}, \omega, \gamma)$, is generated by $x^{\pm 1}, g$ and $y_{m_1},\ldots,y_{m_{\theta}}$, subject to the relations
\begin{equation}
\begin{cases}
xx^{-1}=x^{-1}x=1,\quad  xg=gx,\quad  xy_{m_i}=y_{m_i}x, & \\
y_{m_i}g=\gamma^{m_i} gy_{m_i},\quad y_{m_i}y_{m_j}=y_{m_j}y_{m_i} & \\
y_{m_i}^{e_i}=1-x^{\omega\frac{e_im_i}{m}},\quad g^m=x^{\omega}, & \\
\end{cases}
\end{equation}
where  $\gamma$ is a primitive $m$th root of 1 and $1\leq i, j\leq \theta$. The comultiplication, counit and antipode of $B(\{m_1,\ldots,m_\theta\}, \omega, \gamma)$  are given by
$$\Delta(x)=x\otimes x,\quad  \Delta(g)=g\otimes g, \quad   \Delta(y_{m_i})=y_{m_i}\otimes g^{m_i}+1\otimes y_{m_i},$$
$$\epsilon(x)=1,\quad  \epsilon(g)=1,\quad    \epsilon(y_{m_i})=0,$$ and
$$S(x)=x^{-1},\quad  S(g)=g^{-1} \quad S(y_{m_i})=-y_{m_i}g^{-m_i},$$
for $1\leq i\leq \theta$.
\begin{proposition}\label{p4.6} Let the $\k$-algebra $B=B(\{m_1,\ldots,m_\theta\}, \omega, \gamma)$ be defined as above. Then \begin{itemize}
\item[(1)] The algebra $B$ is a Hopf algebra of GK-dimension one, with center $\k[x^{\pm 1}]$.
\item[(2)] The algebra $B$ is prime and PI-deg $(B)=m$.
\item[(3)] The algebra $B$ has a $1$-dimensional representation whose order is $m$.
\item[(4)] $\io(B)=m.$
 \end{itemize}
\end{proposition}
\begin{proof} (1) It is not hard to see that the center of $B$ is $\k[x^{\pm 1}]$ and $B$ is a free module over $\k[x^{\pm 1}]$ with finite rank. Actually, through a direct computation one can find that $\{y_jg^{i}|0\leq i,j\leq m-1\}$ is a basis of $B$ over $\k[x^{\pm 1}]$. Here recall that if $j\equiv j_1m_1+\ldots +j_{\theta}m_{\theta}$ (mod $m$) then $y_j=\prod_{i=1}^{\theta}y_{m_i}^{j_i}.$ Therefore, it has GK-dimension one. Similar to the case of $T(\underline{m},t,\xi)$, we leave the task to the readers to check that $B$ is a Hopf algebra. Actually, the same as the case of Taft algebras, since for each $1\leq i\leq \theta$ the subalgebra generated by $x^{\pm 1},g, y_{m_i}$ is just a similar kind of generalized Liu algebra which may be not prime now, one can reduce the proof to just considering the mixed relation $y_{m_i}y_{m_j}=y_{m_j}y_{m_i}$ and $y_{m_i}^{e_i}=y_{m_j}^{e_j}$ for $1\leq i,j\leq \theta$.

(2)  As the case of $T(\underline{m},t,\xi)$, we want to apply Lemma \ref{l2.10} to prove that $B$ is prime with PI-degree $m$. At first, let $B_0$ be the subalgebra generated by $y_{m_1},\ldots,y_{m_\theta}$ and $x^{\pm 1}$. Clearly, $B_0$ is a domain and
 $$B=\bigoplus_{i=0}^{m-1} B_{0}g^{i}.$$
 From this, $B$ is a strongly $\widehat{\Z_m}=\langle \chi|\chi^m=1 \rangle$-graded algebra through $\chi(ag^i)=\gamma^{i}$ for any $a\in B_0$ and $0\leq i\leq m-1.$ Therefore, the conditions 1) and 2) of Lemma \ref{l2.10} are fulfilled. By part (b) of Lemma \ref{l2.10}, the action of $\widehat{\Z_m}$ is just the adjoint action of $\Z_{m}=\langle g|g^m=1\rangle$ on $B_0$ which by definition of a fraction of $m$ is faithful. Therefore, PI.deg$(B)=m$ by part (c) of Lemma \ref{l2.10}. In addition, the part (d) of Lemma \ref{l2.10} implies that $B$ is prime now.

 (3)  By the definition of $B$, it has a $1$-dimensional representation $$\pi:\;B\to \k, \;\;x\mapsto 1,\;y_{m_i}\mapsto 0,\;g\mapsto \gamma\;\;(1\leq i\leq \theta).$$
 It's order is clear $m$.

 (4) Using \cite[Lemma 2.6]{LWZ}, we have the right module structure of the left integrals is
 $$\int_{B}^{l}=B/(x-1,\;y_{m_i},\;g-\gamma^{-\sum_{i=1}^{\theta}m_{i}},\;1\leq i\leq \theta).$$
 Next, we want to show that $\sum_{i=1}^{\theta}m_{i}$ is coprime to $m$. Recall that in the definition of a fraction (see Definition \ref{d3.1}), we ask that $(m_i,e_i)=1$ and $m|m_im_j$ for all $1\leq i,j\leq \theta$. Thus $$(e_i,e_j)=1,\;\;\;\;e_i|m_j$$
 for all $1\leq i\neq j\leq \theta$. By (3) of Definition \ref{d3.1}, $m=e_1\cdots e_{\theta}$. On the contrary, assume that $(\sum_{i=1}^{\theta}m_{i},m)\neq 1$. Then there exists $1\leq i\leq \theta$ and a prime factor $p_i| e_i$ such that $p_i|m$ and $p_i|\sum_{i=1}^{\theta}m_{i}.$
 Since $e_i|m_j$ for all $j\neq i$, $p_i|m_j$ for all $j\neq i$. Therefore, $p_i|m_i$ which is impossible since $(m_i,e_i)=1$.

 Therefore, we know that $(\sum_{i=1}^{\theta}m_{i},m)= 1$ and thus $\gamma^{-\sum_{i=1}^{\theta}m_{i}}$ is still a primitive $m$th root of unity which implies that $\io(B)=m.$
\end{proof}

We also call the $1$-dimensional representation stated in (3) of Proposition \ref{p4.6} the \emph{canonical representation} of $B=B(\{m_1,\ldots,m_\theta\}, \omega, \gamma)$. This canonical representation of $B$ yields the corresponding left and right winding automorphisms
 \[{\Xi_{\pi}^l:}
\begin{cases}x\longmapsto x,&\\
y_{m_i}\longmapsto y_{m_i}, &\\
g\longmapsto \gamma g, &
\end{cases} \textrm{and} \;\;\;\;\;
\Xi_{\pi}^r:
\begin{cases}x\longmapsto x,&\\
y_{m_i}\longmapsto \gamma^{m_i}y_{m_i}, &\\
g\longmapsto \gamma g, &
\end{cases}\]
for $1\leq i\leq \theta$.

Using above expression of $\Xi_{\pi}^{l}$ and $\Xi_{\pi}^{r}$, it is not difficult to find that
 \begin{equation} B_{i}^{l}=\k[x^{\pm 1},y_{m_1},\ldots,y_{m_\theta}]g^{i}\;\;\;\;\textrm{and} \;\;\;\; B_{j}^{r}=\k[x^{\pm 1},g^{-m_1}y_{m_1},\ldots,g^{-m_{\theta}}y_{m_\theta}]g^{j}
 \end{equation}
 for all $0\leq i,j\leq m-1.$ Thus we have
 \begin{equation}\label{eq4.5} B_{00}=\k[x^{\pm 1}]\;\;\;\;\textrm{and} \;\;\;\;B_{i,i+j}=\k[x^{\pm 1}]y_{j}g^{i}
 \end{equation}
 for all $0\leq i,j\leq m-1$ where $y_{j}=y_{m_1}^{j_1}\cdots y_{m_\theta}^{j_\theta}$ (see (1) of Remark \ref{r3.2}).

 At the end of this subsection, we also want to consider when two fractions of generalized Liu algebras are the same. To do that, let $m'\in \N$ and $\{m_1',\ldots,m'_{\theta'}\}$ a fraction of $m'$. As before, we denote the corresponding generators and numbers of $B(\underline{m'},\omega',\gamma')$ by adding the symbol $'$ to that of $B(\underline{m},\omega,\gamma).$
\begin{proposition}\label{pliu} As Hopf algebras, if $B(\underline{m},\omega,\gamma)\cong B(\underline{m'},\omega',\gamma')$, then $m=m', \theta=\theta'$ and up to an order of $m_i$'s, $\omega m_i=\omega'm_i'$ for all $1\leq i\leq \theta.$
\end{proposition}
\begin{proof} Since they have the same PI-degrees, $m=m'$. We know the the center of $B(\underline{m},\omega,\gamma)$ is $\k [x^{\pm 1}]$ and thus $\varphi(x)=x'$ or $\varphi(x)=(x')^{-1}$. Also, as before, through comparing the nontrivial skew primitive elements, $\theta=\theta'$ and after a reordering the generators we can assume that $\varphi(y_{m_i})=y'_{m_i'}.$ The relation $y_{m_i}^{e_i}=1-x^{\omega\frac{e_im_i}{m}}$ implies that $e_i=e_i'$ and $\varphi(x)=x'$ since by assumption all $e_i,m_i$ and $m$ are positive. From which one has $$\omega\frac{e_im_i}{m}=\omega'\frac{e_i'm_i'}{m'}.$$
Since $m=m'$ and $e_i=e_i'$, $\omega m_i=\omega'm_i'$ for all $1\leq i\leq \theta.$
\end{proof}

It is a pity that the conditions in above proposition is only a necessary condition for $B(\underline{m},\omega,\gamma)\cong B(\underline{m'},\omega',\gamma')$. To get a sufficient one, or an equivalent condition, we need the following observation.

\begin{lemma}\label{l4.9} Any fraction of generalized Liu algebra $B(\underline{m},\omega,\gamma)$ is isomorphic to a \emph{unique} $B(\underline{m'},\omega',\gamma')$ satisfying $(m_1',\ldots,m'_{\theta'})=1$.
\end{lemma}
\begin{proof} We prove the existence at first and then prove the uniqueness. Take an arbitrary $B(\underline{m},\omega,\gamma)$. Let $m_0=(m_1,\ldots,m_{\theta})$. Above proposition suggests us to construct the following algebra
$$B(\{\frac{m_1}{m_0},\ldots, \frac{m_\theta}{m_0}\}, \omega m_0,\gamma^{m^2_0}).$$
Clearly, $\{\frac{m_1}{m_0},\ldots, \frac{m_\theta}{m_0}\}$ is a fraction of $m$ with length $\theta$ and $(\frac{m_1}{m_0},\ldots, \frac{m_\theta}{m_0})=1$.

\emph{Claim 1: As Hopf algebras, $B(\underline{m},\omega,\gamma)\cong B(\{\frac{m_1}{m_0},\ldots, \frac{m_\theta}{m_0}\}, \omega m_0,\gamma^{m^2_0}).$}

\emph{Proof of the claim 1.}  Since $(m_0, m)=1$, there exist $a\in \N, b\in \Z$ such that $am_0+bm=1$. Define the following map
\begin{align*}\varphi:\;& B(\underline{m},\omega,\gamma)\longrightarrow B(\{\frac{m_1}{m_0},\ldots, \frac{m_\theta}{m_0}\}, \omega m_0,\gamma^{m^2_0}),\\
&\;\;x\mapsto x',\;g\mapsto (g')^{a}(x')^{b\omega},\;y_{m_i}\mapsto y'_{\frac{m_i}{m_0}},\;\;(1\leq i\leq \theta).\end{align*}
Since
    \begin{align*}\varphi(g^{m_i})&=\varphi(g)^{m_i}=((g')^{a}(x')^{b\omega})^{m_i}
    =(g')^{am_0\frac{m_i}{m_0}}(x')^{b\omega' \frac{m_i}{m_0}}\\
    &=(g')^{am_0\frac{m_i}{m_0}}(g')^{bm \frac{m_i}{m_0}}=(g')^{(am_0+bm)\frac{m_i}{m_0}}\\
    &=(g')^{\frac{m_i}{m_0}}
    \end{align*}
    and
    \begin{align*}\varphi(y_{m_i}g)&=\varphi(y_{m_i})\varphi(g)
    =y'_{\frac{m_i}{m_0}}(g')^{a}(x')^{b\omega}\\
    &=\gamma^{am^2_0\frac{m_i}{m_0}}(g')^{a}(x')^{b\omega}y'_{\frac{m_i}{m_0}}
    =\gamma^{m_i}\varphi(g)\varphi(y_{m_i})\\
    &=\varphi(\gamma^{m_i}gy_{m_i}),
    \end{align*} for all $1\leq i\leq \theta$, it is not hard to prove that $\varphi$ gives the desired isomorphism.

    Next, let's show that uniqueness. To prove it, it is enough to built the following statement.

    \emph{Claim 2: Let $\{m_1,\ldots,m_\theta\}$ and $\{m'_1,\ldots,m'_\theta\}$ be two fractions of $m$ with length $\theta$ satisfying $(m_1,\ldots,m_\theta)=(m'_1,\ldots,m'_\theta)=1$. If $B(\underline{m},\omega,\gamma)$ is isomorphic to $B(\underline{m'},\omega',\gamma')$, then up to an order of $m_i$'s we have $m_i=m_i'$, $\omega=\omega'$ and $\gamma=\gamma'$ for $1\leq i\leq \theta$.}

    \emph{Proof of Claim 2.} By Proposition \ref{pliu}, $\omega m_i=\omega' m_i'$. Since $$(m_1,\ldots,m_\theta)=(m'_1,\ldots,m'_\theta)=1,$$ $\omega|\omega'$ and $\omega'|\omega$. Therefore $\omega=\omega'$ and thus $m_i=m_i'$ for all $1\leq i\leq \theta.$ From this, we know the isomorphism given in the proof of Proposition \ref{pliu} must sent $g^{m_i}$ to $(g')^{m_i}$, i.e., keeping the notations used in the proof of Proposition \ref{pliu}, we have $\varphi(g^{m_i})=(g')^{m_i}$ for all $1\leq i\leq\theta$. Since $(m_1,\ldots, m_\theta)=1$, there exist $a_i\in \Z$ such that $\sum_{i=1}^{\theta}a_im_i=1$.
    Thus $$\varphi(g)=\varphi(g^{\sum_{i=1}^{\theta}a_im_i})=(g')^{\sum_{i=1}^{\theta}a_im_i}=g'.$$
    This implies that $$\gamma^{m_i}=(\gamma')^{m_i}$$
    through using the relation $y_{m_i}g=\gamma^{m_i}gy_{m_i}$. So,
    $$\gamma=\gamma^{\sum_{i=1}^{\theta}a_im_i}=(\gamma')^{\sum_{i=1}^{\theta}a_im_i}=\gamma'.$$
\end{proof}

\begin{definition}\label{d4.10} \emph{We call the Hopf algebra $B(\{\frac{m_1}{m_0},\ldots, \frac{m_\theta}{m_0}\}, \omega m_0,\gamma^{m^2_0})$ the \emph{basic form} of $B(\underline{m},\omega,\gamma)$.}
\end{definition}
By this lemma, we can tell when two fractions of generalized Liu algebras are isomorphic now. Keeping notations before, let $m,m'\in \N$ and $\{m_1,\ldots,m_{\theta}\},\;$ $\{m_1',\ldots,m'_{\theta'}\}$ be fractions of $m$ and $m'$ respectively. Let $m_0:=(m_1,\ldots,m_\theta)$ and $m_0':=(m_1',\ldots,m'_{\theta'}).$

\begin{proposition}\label{iliu} Retain above notations. As Hopf algebras, $B(\underline{m},\omega,\gamma)\cong B(\underline{m'},\omega',\gamma')$ if and only if $m=m', \theta=\theta',\; \omega m_0=\omega'm_0'$ and $\gamma^{m^2_0}=\gamma'^{(m_0')^2}.$
\end{proposition}
\begin{proof} Note that $B(\underline{m},\omega,\gamma)\cong B(\underline{m'},\omega',\gamma')$ if and only if they have the same basic forms by above lemma. Now the condition listed in the proposition is clearly equivalent to say that the basic forms of them are same.
\end{proof}

\subsection{Fraction of the Hopf algebra $D(\underline{m},d,\gamma)$.}\label{ss4.4}
 Let $m,d$ be two natural numbers, $m_1,\ldots,m_\theta$ a fraction of $m$ satisfying the following two conditions:
 \begin{equation}\label{eq4.7} 2|\sum_{i=1}^{\theta}(m_i-1)(e_i-1)\quad \textrm{and}\quad 2|\sum_{i=1}^{\theta}(e_i-1)m_id.
  \end{equation}
   Let $\gamma$ a primitive $m$th root of unity and define
 \begin{equation}\xi_{m_i}:=\sqrt{\gamma^{m_i}},\;\;\;\;1\leq i\leq \theta.\end{equation}
  That is, $\xi_{m_i}$ is a primitive square root of $\gamma^{m_i}$. Therefore in particular, one has \begin{equation}\xi_{m_i}^{e_i}=-1\end{equation}
  for all $1\leq i\leq \theta$.

 In order to give the definition of the Hopf algebra $D(\underline{m},d,\gamma)$, we still need recall two notations introduced in Section 3:
\begin{equation} {]s,t[_{m_i}}=\begin{cases}
\phi_{m_i,(\bar{t}+1)m_i}\cdots \phi_{m_i,(e_i-1)m_i}\phi_{m_i,0}\cdots\phi_{m_i,(\bar{s}-1)m_i}, & \textrm{if}\; \bar{t}\geqslant \bar{s}
\\1, & \textrm{if}\; \bar{s}=\overline{t}+1 \\
\phi_{m_i,(\bar{t}+1)m_i}\cdots \phi_{m_i,(\bar{s}-1)m_i}, & \textrm{if}\;
\overline{s}\geqslant \bar{t}+2.
\end{cases} \end{equation}
and
\begin{equation} {[s,t]_{m_i}}:=\begin{cases}
\phi_{m_i,\bar{s}m_i}\phi_{m_i,(\bar{s}+1)m_i}\cdots \phi_{m_i,\bar{t}m_i}, & \textrm{if}\; \bar{t}\geqslant \bar{s}
\\1, & \textrm{if}\; \bar{s}=\overline{t}+1 \\
\phi_{m_i,\bar{s}m_i}\cdots \phi_{m_i,(e_i-1)m_i}\phi_{m_i,0}\cdots \phi_{m_i,\bar{t}m_i}, & \textrm{if}\;
\overline{s}\geqslant \bar{t}+2.
\end{cases} \end{equation}
where $\phi_{m_i,j}=1-\gamma^{-m_i^2(j_i+1)}x^{m_id}$ for all $1\leq i\leq \theta$. See \eqref{eqomit} and \eqref{eqpre} for details. Now we are in the position to give the definition of $D(\underline{m},d,\gamma)$.

$\bullet$ As an algebra, $D=D(\underline{m},d,\gamma)$ is generated by $x^{\pm 1}, g^{\pm 1}, y_{m_1},\ldots, y_{m_\theta}, u_0, u_1, \cdots,
u_{m-1}$, subject to the following relations
\begin{eqnarray}
&&xx^{-1}=x^{-1}x=1,\quad gg^{-1}=g^{-1}g=1,\quad xg=gx,\quad xy_{m_i}=y_{m_i}x\\
&&y_{m_i}y_{m_k}=y_{m_k}y_{m_i},\quad y_{m_i}g=\gamma^{m_i} gy_{m_i},\quad y_{m_i}^{e_i}=1-x^{e_im_id},\quad g^{m}=x^{md},\\
&&xu_j=u_jx^{-1},\quad y_{m_i}u_j=\phi_{m_i,j}u_{j+m_i}=\xi_{m_i}x^{m_id}u_jy_{m_i}\quad u_j g=\gamma^j x^{-2d}gu_j,\\
\label{eq4.14}&&u_ju_l=(-1)^{\sum_{i=1}^{\theta}l_i}\gamma^{\sum_{i=1}^{\theta}m_i^2\frac{l_i(l_i+1)}{2}}
\frac{1}{m}x^{-\frac{2+\sum_{i=1}^{\theta}(e_i-1)m_i}{2}d}
\\
\notag &&\quad \quad\quad\prod_{i=1}^{\theta}\xi_{m_i}^{-l_i}[j_i,e_i-2-l_i]_{m_i}y_{\overline{j+l}}g
\end{eqnarray}
for $1\leq i,k\leq \theta,$ and $0\leq j,l\leq m-1$ and here for any integer $n$,  $\overline{n}$ means remainder of division of $n$ by $m$ and as before $n\equiv\sum_{i=1}^{\theta} n_im_i$ (mod $m$) by Remark \ref{r3.2}.

$\bullet$ The coproduct $\D$, the counit $\epsilon$ and the antipode $S$ of $D(\underline{m},d,\gamma)$ are given by
\begin{eqnarray}
&&\D(x)=x\otimes x,\;\; \D(g)=g\otimes g, \;\;\D(y_{m_i})=y_{m_i}\otimes g^{m_i}+1\otimes y_{m_i},\\
&&\D(u_j)=\sum_{k=0}^{m-1}\gamma^{k(j-k)}u_k\otimes x^{-kd}g^ku_{j-k};\\
&&\epsilon(x)=\epsilon(g)=\epsilon(u_0)=1,\;\;\epsilon(y_{m_i})=\epsilon(u_s)=0;\\
&&S(x)=x^{-1},\;\; S(g)=g^{-1}, \;\;S(y_{m_i})=-y_{m_i}g^{-m_i},\\
&&\label{eq4.19} S(u_j)=(-1)^{\sum_{i=1}^{\theta}j_i}\gamma^{-\sum_{i=1}^{\theta}
m_i^2\frac{j_i(j_i+1)}{2}}
x^{b+\sum_{i=1}^{\theta}j_im_id}
g^{m-1-(\sum_{i=1}^{\theta}j_im_i)}\prod_{i=1}^{\theta}\xi_{m_i}^{-j_i}u_j,
\end{eqnarray}

for $1\leq i\leq \theta,\;1\leq s\leq m-1\;,0\leq j\leq m-1$ and $b=(1-m)d-\frac{\sum_{i=1}^{\theta}(e_i-1)m_i}{2}d$.

Before we prove that $D(\underline{m},d,\gamma)$ is a Hopf algebra, which is highly nontrivial, we want to express the formula \eqref{eq4.14} and \eqref{eq4.19} in a more convenient way.

On one hand, we find that
\begin{equation}\label{eq4.20}(-1)^{-ke_i-j_i}\xi_{m_i}^{-ke_{i}-j_{i}}
\gamma^{m_i^{2}\tfrac{(ke_i+j_i)(ke_i+j_i+1)}{2}}=(-1)^{-j_i}\xi_{m_i}^{-j_i}
\gamma^{m_i^{2}\tfrac{j_i(j_i+1)}{2}}\end{equation}
for any $k\in \mathbbm{Z}$. Therefore, if we define
$$u_{s}:=u_{\overline{s}},$$
where $\overline{s}$ means the remainder of $s$ modulo $m$, then the relation \eqref{eq4.14} can be replaced by
\begin{align}
\notag u_ju_l&=(-1)^{\sum_{i=1}^{\theta}l_i}\gamma^{\sum_{i=1}^{\theta}m_i^2\frac{l_i(l_i+1)}{2}}
\frac{1}{m}x^{-\frac{2+\sum_{i=1}^{\theta}(e_i-1)m_i}{2}d}
\\
\notag &\quad\quad\quad\quad\quad\quad\quad\quad\quad\quad\quad\quad\quad\quad\quad\quad
\prod_{i=1}^{\theta}\xi_{m_i}^{-l_i}[j_i,e_i-2-l_i]_{m_i}y_{\overline{j+l}}g\\
   \notag  &=(-1)^{\sum_{i=1}^{\theta}l_i}\gamma^{\sum_{i=1}^{\theta}m_i^2\frac{l_i(l_i+1)}{2}}
\frac{1}{m}x^{-\frac{2+\sum_{i=1}^{\theta}(e_i-1)m_i}{2}d}
\\
\notag&\quad\quad\quad\quad\quad\quad\quad\quad\quad\quad\quad\quad\quad\quad\quad\quad
\prod_{i=1}^{\theta}\xi_{m_i}^{-l_i}]-1-l_i,j_i-1[_{m_i}y_{\overline{j+l}}g\\
\label{eq4.21} &=\frac{1}{m}x^{-\frac{2+\sum_{i=1}^{\theta}(e_i-1)m_i}{2}d} \prod_{i=1}^{\theta}
(-1)^{l_i}\xi_{m_i}^{-l_i}\gamma^{m_i^2\frac{l_i(l_i+1)}{2}}]-1-l_i,j_i-1[_{m_i}y_{\overline{j+l}}g\\
\notag &=\frac{1}{m}x^{-\frac{2+\sum_{i=1}^{\theta}(e_i-1)m_i}{2}d} \prod_{i=1}^{\theta}
(-1)^{l_i}\xi_{m_i}^{-l_i}\gamma^{m_i^2\frac{l_i(l_i+1)}{2}}[j_i,e_i-2-l_i]_{m_i}y_{\overline{j+l}}g
\end{align} for all $j, l\in \mathbbm{Z}$, that is, we need not always ask that $0\leq j,l\leq m-1$.

On other hand, since $g^m=x^{md}$ and \eqref{eq4.20} , the definition about $S(u_{j})$ still holds for any integer $j$, that is, \eqref{eq4.19}
can be replaced in the following way:
\begin{align}\notag S(u_j)&=(-1)^{\sum_{i=1}^{\theta}j_i}\prod_{i=1}^{\theta}\xi_{m_i}^{-j_i}\gamma^{-\sum_{i=1}^{\theta}
m_i^2\frac{j_i(j_i+1)}{2}}
x^{\sum_{i=1}^{\theta}j_im_id}x^{b}
g^{m-1-(\sum_{i=1}^{\theta}j_im_i)}u_j\\
&= x^{b}g^{m-1}\prod_{i=1}^{\theta}(-1)^{j_i}\xi_{m_i}^{-j_i}\gamma^{-
m_i^2\frac{j_i(j_i+1)}{2}}
x^{j_im_id}
g^{-j_im_i}u_j
\end{align}
for all $j\in \mathbbm{Z}$.

We also need to give a bigrading on this algebra for the proof. Let $\xi:=\sqrt{\gamma}$ and define the following two algebra automorphisms of $D(\underline{m},d,\gamma)$:
\[{\Xi_{\pi}^l:}
\begin{cases}
x\longmapsto x, &\\
y_{m_i}\longmapsto y_{m_i}, &\\
g\longmapsto \gamma g, &\\
u_i\longmapsto \xi u_i, &
\end{cases}
\textrm{and}\;\;\;\; \Xi_{\pi}^r:
\begin{cases}
x\longmapsto x, &\\
y_{m_i}\longmapsto \gamma^{m_i}y_{m_i}, &\\
g\longmapsto \gamma g, &\\
u_j\longmapsto \xi^{2j+1} u_j,&
\end{cases}\]
for $1\leq i\leq \theta$ and $0\leq j\leq m-1.$ It is straightforward to show that $\Xi_{\pi}^l$ and $\Xi_{\pi}^r$ are indeed algebra automorphisms of $D(\underline{m},d,\gamma)$ and these automorphisms have order $2m$ by noting that $\xi$ is a primitive $2m$th root of 1. Define
\[{D_i^l=}
\begin{cases}
\k [x^{\pm 1}, y_{m_1},\ldots,y_{m_\theta}] g^{\tfrac{i}{2}}, & i=\textrm{even},\\
\sum_{s=0}^{m-1}\k[ x^{\pm 1}] g^{\tfrac{i-1}{2}}u_s,
& i=\textrm{odd},
\end{cases}\]
and \[ D_j^r=
\begin{cases}
\k[ x^{\pm 1}, y_{m_1}g^{-m_1},\dots,y_{m_\theta}g^{-m_\theta}] g^{\tfrac{j}{2}}, & j=\textrm{even},\\
\sum_{s=0}^{m-1}\k[x^{\pm 1}]
g^su_{\tfrac{j-1}{2}-s}, & j=\textrm{odd}.
\end{cases}\]
Therefore \begin{equation}\label{eqD} D_{ij}:=D_i^l\cap D_j^r=
\begin{cases}
\k[x^{\pm 1}]y_{\overline{\frac{j-i}{2}}}g^{\frac{i}{2}}, & i, j=\textrm{even},\\
\k[x^{\pm 1}]g^{\tfrac{i-1}{2}}u_{{\tfrac{j-i}{2}}}, & i, j=\textrm{odd},\\
0, & \textrm{otherwise}.
\end{cases}\end{equation}
Since $\sum_{i,j}D_{ij}=D(\underline{m},d,\gamma)$, we have
\begin{equation}\label{eq4.24}D(\underline{m},d,\gamma)=\bigoplus_{i,j=0}^{2m-1}D_{ij}\end{equation}
 which is a bigrading on $D(\underline{m},d,\gamma)$ automatically.

 Let $D:= D(\underline{m},d,\gamma)$, then $D\otimes D$ is graded naturally by inheriting the grading defined above. In particular,
 for any $h\in D\otimes D$,  we use
 $$h_{(s_1,t_1)\otimes (s_2,t_2)}$$
 to denote the homogeneous part of $h$ in $D_{s_{1},t_{1}}\otimes D_{s_{2},t_{2}} $. This notion will be used freely  in the proof of the following desired proposition.

\begin{proposition}\label{p4.7} The algebra $D(\underline{m},d,\gamma)$ defined above is a Hopf algebra.
\end{proposition}

\emph{Proof:} The proof is standard but not easy. We are aware that one can not apply the fact that the non-fraction version $D(m,d,\gamma)$ (see Subsection 2.3) is already a Hopf algebra to simply the proof although we can do this in the proofs of Proposition \ref{p4.6} and \ref{p4.1}. The reason is that if we consider the subalgebra generated by $x^{\pm 1},g,u_{0},\ldots, u_{m-1}$ together with a single $y_{m_i}$ (this is the case of $D(m,d,\gamma)$) then we can find that the other $y_{m_j}$'s will be created naturally. So, one has to prove it step by step. Since the subalgebra generated by $x^{\pm 1}, y_{m_1},\dots,y_{m_{\theta}}, g$ is just a fraction version of generalized Liu algebra $B(\underline{m}, \omega, \gamma)$, which is a Hopf algebra already (by Proposition \ref{p4.6}), we only need to verify the related relations in $D(\underline{m},d,\gamma)$ where $u_j$ are involved.

\noindent $\bullet$ \emph{Step} 1 ($\D$ and $\epsilon$ are algebra homomorphisms).

First of all, it is clear that $\epsilon$ is an algebra
homomorphism. Since $x$ and $g$ are group-like elements, the
verifications of $\D(x)\D(u_i)=\D(u_i)\D(x^{-1})$ and
$\D(u_i)\D(g)=\gamma^{i}\D(x^{-2d})\D(g)\D(u_i)$ are simple and so they are omitted.

\noindent (1) \emph{The proof of} $\D(\phi_{m_i,j})\D(u_{m_i+j})=\D(y_{m_i})\D(u_j)=\xi_{m_i}\D(x^{m_id})\D(u_j)\D(y_{m_i})$.

Define $$\gamma_i:=\gamma^{-m_i^2}$$ for all $1\leq i\leq \theta.$

By definition $\D(u_j)=\sum_{k=0}^{m-1}\gamma^{k(k-j)}u_k\otimes
x^{-kd}g^ku_{j-k}$ for all $0 \leqslant j \leqslant m-1$, we have
\begin{align*}
\D(\phi_{m_i,j})\D(u_{m_i+j})&=(1\otimes 1 - \gamma_i^{1+j_i}x^{m_id} \otimes x^{m_id})\sum_{k=0}^{m-1}\gamma^{k(j+m_i-k)}u_k\otimes x^{-kd}g^ku_{j+m_i-k}\\
                 &=\sum_{k=0}^{m-1}\gamma^{k(j+m_i-k)}u_k\otimes x^{-kd}g^ku_{j+m_i-k}
                 \\
                 &\quad -\sum_{k=0}^{m-1}\gamma^{-m_i^2(1+j_i)+k(j+m_i-k)}x^{m_id}u_k\otimes x^{m_id-kd}g^ku_{j+m_i-k}.
\end{align*}
And
\begin{align*}
\D(y_{m_i})\D(u_j)&=(1\otimes y_{m_i}+y_{m_i}\otimes g^{m_i})(\sum_{k=0}^{m-1}\gamma^{k(k-j)}u_k\otimes
x^{-kd}g^ku_{j-k})\\
&=\sum_{k=0}^{m-1}\gamma^{k(j-k)}u_k\otimes x^{-kd}g^{k}\gamma^{km_i}\phi_{m_i,j-k}u_{j+m_i-k}\\
& \quad + \sum_{k=0}^{m-1}\gamma^{k(j-k)}\phi_{m_i,k}u_{m_i+k}\otimes x^{-kd}g^{m_i+k}u_{j-k}\\
&=\sum_{k=0}^{m-1}\gamma^{k(j-k)+km_i}u_{k}\otimes x^{-kd}g^{k}u_{j+m_i-k}\\
&\quad - \sum_{k=0}^{m-1}\gamma^{k(j-k)}u_{k}\otimes\gamma^{-m_i^2(j_i+1-2k_i)}
            x^{(m_i-k)d}g^{k}u_{j+m_i-k}\\
&\quad +\sum_{k=0}^{m-1}\gamma^{k(j-k)}u_{m_i+k}\otimes x^{-kd}g^{m_i+k}u_{j-k}\\
&\quad - \sum_{k=0}^{m-1}\gamma^{k(j-k)-m_i^2(1+k_i)}x^{m_id}u_{m_i+k}\otimes
            x^{-kd}g^{m_i+k}u_{j-k}\\
 &=\sum_{k=0}^{m-1}\gamma^{k(j-k)+km_i}u_{k}\otimes x^{-kd}g^{k}u_{j+m_i-k}\\
 &\quad - \sum_{k=0}^{m-1}\gamma^{k(j-k)-m_i^2(j_i+1-2k_i)}u_{k}\otimes
            x^{(m_i-k)d}g^{k}u_{j+m_i-k}\\
 &\quad +\sum_{k=0}^{m-1} \gamma^{(k-m_i)(j-k+m_i)}u_{k}\otimes x^{-(k-m_i)d}g^ku_{j+m_i-k}\\
 &\quad -\sum_{k=0}^{m-1}\gamma^{(k-m_i)(j+m_i-k)-m_i^2k_i}x^{m_id}u_k\otimes x^{-(k-m_i)d}g^{k}u_{j+m_i-k}\\
 &= \sum_{k=0}^{m-1}\gamma^{k(j+m_i-k)}u_k\otimes x^{-kd}g^ku_{j+m_i-k}
                 \\
                 &\quad -\sum_{k=0}^{m-1}\gamma^{-m_i^2(1+j_i)+k(j+m_i-k)}x^{m_id}u_k\otimes x^{m_id-kd}g^ku_{j+m_i-k}.
            \end{align*}
Here we use the following equalities
$$\gamma^{(k-m_i)(j-k+m_i)}=\gamma^{k(j-k)+km_i-m_i(j-k)-m_{i}^2}=\gamma^{k(j-k)+2k_im_i^2-m_i^2(1+j_i)},$$
and
$$\gamma^{(k-m_i)(j+m_i-k)-m_i^2k_i}=\gamma^{-m_i^2(1+j_i)+k(j+m_i-k)}.$$
 Hence $\D(\phi_{m_i,j})\D(u_{m_i+j})=\D(y_{m_i})\D(u_j)$. Similarly,

 $\xi_{m_i}\D(x^{m_id})\D(u_j)\D(y_{m_i})$
\begin{align*}
&=\xi_{m_i}(x^{m_id}\otimes x^{m_id})(\sum_{k=0}^{m-1}\gamma^{k(j-k)}u_k\otimes x^{-kd}g^ku_{j-k})(1\otimes y_{m_i} +y_{m_i}\otimes g^{m_i})\\
&=\sum_{k=0}^{m-1}\xi_{m_i}\gamma^{k(j-k)}x^{m_id}u_k\otimes x^{(m_i-k)d}g^{k}u_{j-k}y_{m_i}\\
&\quad + \sum_{k=0}^{m-1}\xi_{m_i}\gamma^{k(j-k)}x^{m_id}u_ky_{m_i}\otimes x^{(m_i-k)d}g^ku_{j-k}g^{m_i}\\
&=\sum_{k=0}^{m-1}\gamma^{k(j-k)}x^{m_id}u_k\otimes x^{-kd}g^{k}\phi_{m_i,j-k}u_{j+m_i-k}\\
&\quad +\sum_{k=0}^{m-1}\gamma^{k(j-k)}\phi_{m_i,k}u_{k+m_i}\otimes \gamma^{(j-k)m_i}x^{(-m_i-k)d}g^{k+m_i}u_{j-k}\\
&=\sum_{k=0}^{m-1}\gamma^{k(j-k)}x^{m_id}u_k\otimes x^{-kd}g^{k}u_{j+m_i-k}\\
&\quad - \sum_{k=0}^{m-1}\gamma^{k(j-k)-m_i^2(1+j_i-k_i)}x^{m_id}u_k\otimes x^{(-k+m_i)d}g^{k}u_{j+m_i-k}\\
&\quad + \sum_{k=0}^{m-1}\gamma^{(k-m_i)(j-k+m_i)}(1-\gamma^{-m_i^2k_i}x^{m_i}d)u_{k}\otimes \gamma^{(j-k+m_i)m_i}x^{-kd}g^{k}u_{j+m_i-k}\\
&=\sum_{k=0}^{m-1}\gamma^{k(j-k)}x^{m_id}u_k\otimes x^{-kd}g^{k}\phi_{m_i,j-k}u_{j+m_i-k}\\
&\quad +\sum_{k=0}^{m-1}\gamma^{k(j-k)}\phi_{m_i,k}u_{k+m_i}\otimes \gamma^{(j-k)m_i}x^{(-m_i-k)d}g^{k+m_i}u_{j-k}\\
&=\sum_{k=0}^{m-1}\gamma^{k(j-k)}x^{m_id}u_k\otimes x^{-kd}g^{k}u_{j+m_i-k}\\
&\quad - \sum_{k=0}^{m-1}\gamma^{k(j-k)-m_i^2(1+j_i-k_i)}x^{m_id}u_k\otimes x^{(-k+m_i)d}g^{k}u_{j+m_i-k}\\
&\quad + \sum_{k=0}^{m-1}\gamma^{k(j-k+m_i)}u_{k}\otimes x^{-kd}g^{k}u_{j+m_i-k}\\
&-\sum_{k=0}^{m-1}\gamma^{k(j-k)}x^{m_id}u_k\otimes x^{-kd}g^{k}u_{j+m_i-k}\\
&= \sum_{k=0}^{m-1}\gamma^{k(j-k+m_i)}u_k\otimes x^{-kd}g^{k}u_{j+m_i-k}\\
&\quad -\sum_{k=0}^{m-1}\gamma^{k(j-k)-m_i^2(1+j_i-k_i)}x^{m_id}u_k\otimes x^{m_id-kd}g^ku_{j+m_i-k}\\
&= \D(\phi_{m_i,j})\D(u_{m_i+j}).
\end{align*}

\noindent (2)\emph{ The proof of} $\D(u_ju_l)=\D(u_j)\D(u_l)$.

Direct computation shows that
\begin{align*}
\D(u_j)\D(u_l)&=\sum_{s=0}^{m-1}\gamma^{s(j-s)}u_s\otimes x^{-sd}g^su_{j-s} \sum_{t=0}^{m-1}\gamma^{t(l-t)}u_t\otimes x^{-td}g^tu_{l-t}\\
              &=\sum_{t=0}^{m-1}\sum_{s=0}^{m-1}\gamma^{s(j-s)}u_s\gamma^{(t-s)(l-t+s)}u_{t-s}
              \otimes x^{-sd}g^su_{j-s}x^{-(t-s)d}g^{t-s}u_{l-t+s}\\
              &=\sum_{t=0}^{m-1}\sum_{s=0}^{m-1}\gamma^{(t-s)(l-t+s)+(j-s)t}u_su_{t-s}\otimes x^{-td}g^tu_{j-s}u_{l-t+s}.
\end{align*}
By the bigrading given in \eqref{eq4.24}, we can find that for each $0\leqslant t \leqslant m-1$,
$$\sum_{s=0}^{m-1}\gamma^{(t-s)(l-t+s)+(j-s)t}u_su_{t-s}\otimes
x^{-td}g^tu_{j-s}u_{l-t+s}\in D_{2, 2+2t}\otimes D_{2+2t,
2+2(j+l)},$$ where the suffixes in $ D_{2, 2+2t}\otimes D_{2+2t,
2+2(j+l)}$ are interpreted mod $2m$.

Using equation \eqref{eq4.21}, we get that
$$u_su_{t-s}=\frac{1}{m}x^{a}\prod_{i=1}^{\theta}(-1)^{(t-s)_i}\xi_{m_i}^{-(t-s)_i}\gamma^{m_i^2
\frac{(t-s)_i((t-s)_i+1)}{2}}\;{[s_i, e_i-2-(t-s)_i]_{m_i}}\; y_{t}g$$
 and
\begin{eqnarray*}
u_{j-s}u_{l-t+s}&=&\frac{1}{m}x^{a}\prod_{i=1}^{\theta}
(-1)^{(l-t+s)_i}\xi_{m_i}^{-(l-t+s)_i}\gamma^{m_i^{2}\frac{(l-t+s)_i[(j-t+s)_i+1]}{2}}\\
&&{[(j-s)_i, e_i-2-(l-t+s)_i]_{m_i}}\;y_{\overline{j+l-t}}g\end{eqnarray*}
here and the following of this proof $a=-\frac{2+\sum_{i=1}^{\theta}(e_i-1)m_i}{2}d$.

Using \cite[Proposition IV.2.7]{Kas}, for each $1\leq i\leq \theta$
\begin{align*}
\;{[s_i, e_i-2-(t-s)_i]_{m_i}}&=(1-\gamma_i^{s+1}x^{m_id})(1-\gamma_i^{s+2}x^{m_id})\cdots (1-\gamma_i^{(e_i-1-t_i+s_i)}x^{m_id})\\
           &=\sum_{\alpha_i=0}^{e_i-1-t_i}(-1)^{\alpha_i}\binom{e_i-1-t_i}{\alpha_i}_{\gamma_i}
           \gamma_i^{\tfrac{\alpha_i(\alpha_i-1)}{2}}(\gamma_i^{s+1}x^{m_id})^{\alpha_i}\\
           &=\sum_{\alpha_i=0}^{e_i-1-t_i}(-1)^{\alpha_i}\binom{e_i-1-t_i}{\alpha_i}_{\gamma_i}
           \gamma_i^{\tfrac{\alpha_i(\alpha_i+1)}{2}+s_i\alpha_i}x^{m_id\alpha_i},
\end{align*} and
\begin{eqnarray*}
\;&&{[(j-s)_i, e_i-2-(l-t+s)_i]_{m_i}}\\
&=&(1-\gamma_i^{j_i-s_i+1}x^{m_id})(1-\gamma_i^{j_i-s_i+2}x^{m_id})\cdots (1-\gamma_i^{j_i-s_i+e_i-1-\overline{(j_i+l_i-t_i)}}x^{m_id})\\
&=&\sum_{\beta_i=0}^{e_i-1-\overline{(j_i+l_i-t_i)}}
(-1)^{\beta_i}\binom{e_i-1-\overline{(j_i+l_i-t_i)}}{\beta_i}_{\gamma_i}
\gamma_i^{\tfrac{\beta_i(\beta_i-1)}{2}}(\gamma_i^{j_i-s_i+1}x^{m_id})^{\beta_i}\\
&=&\sum_{\beta_i=0}^{e_i-1-\overline{(j_i+l_i-t_i)}}
(-1)^{\beta_i}\binom{e_i-1-\overline{(j_i+l_i-t_i)}}{\beta_i}_{\gamma_i}
\gamma_i^{\tfrac{\beta_i(\beta_i+1)}{2}+(j_i-s_i)\beta_i}x^{m_id\beta_i},
\end{eqnarray*}
where $\overline{(j_i+l_i-t_i)}$ is the remainder of $j_i+l_i-t_i$ divided by $e_i$.

Then for each $0\leqslant t \leqslant m-1$,
\begin{eqnarray}
&& \D(u_j)\D(u_l)_{{(2, 2+2t)}\otimes {(2+2t,2+2(j+l))}}\\
\notag &=&\sum_{s=0}^{m-1}\gamma^{(t-s)(l-t+s)+(j-s)t}u_su_{t-s}\otimes x^{-td}g^tu_{j-s}u_{l-t+s}\\
\notag &=&\sum_{s=0}^{m-1}\gamma^{(t-s)(l-t+s)+(j-s)t}\frac{1}{m}x^{a}\prod_{i=1}^{\theta}(-1)^{(t-s)_i}\xi_{m_i}^{-(t-s)_i}\gamma^{m_i^2
\frac{(t-s)_i((t-s)_i+1)}{2}}\\
\notag&&{[s_i, e_i-2-(t-s)_i]_{m_i}} y_{t}g\\
\notag&&\otimes x^{-td}g^t\frac{1}{m}x^{a}\prod_{i=1}^{\theta}
(-1)^{(l-t+s)_i}\xi_{m_i}^{-(l-t+s)_i}\gamma^{m_i^{2}\frac{(l-t+s)_i[(j-t+s)_i+1]}{2}}\\
\notag&& {[(j-s)_i, e_i-2-(l-t+s)_i]_{m_i}}\;y_{\overline{j+l-t}}g\\
\notag &=&[\sum_{s=0}^{m-1}\gamma^{(j-s)t-t(j+l-t)}\frac{1}{m^2}\prod_{i=1}^{\theta}
(-1)^{l_i}\xi_{m_i}^{-l_i}\gamma^{m_i^2
\frac{l_i(l_i+1)}{2}}[s_i, e_i-2-(t-s)_i]_{m_i}\\
\notag&& \otimes x^{-td}\prod_{i=1}^{\theta}[(j-s)_i, e_i-2-(l-t+s)_i]_{m_i}](x^{a}y_t g\otimes x^a y_{\overline{j+l-t}}g^{t+1})\\
\notag &=& [\sum_{s=0}^{m-1}\gamma^{(j-s)t-t(j+l-t)}\frac{1}{m^2}\prod_{i=1}^{\theta}
(-1)^{l_i}\xi_{m_i}^{-l_i}\gamma^{m_i^2
\frac{l_i(l_i+1)}{2}}\\
\notag&&\sum_{\alpha_i=0}^{e_i-1-t_i}(-1)^{\alpha_i}\binom{e_i-1-t_i}{\alpha_i}_{\gamma_i}
           \gamma_i^{\tfrac{\alpha_i(\alpha_i+1)}{2}+s_i\alpha_i}\\
           \notag&\otimes&\prod_{k=1}^{\theta} \sum_{\beta_k=0}^{e_k-1-\overline{(j_k+l_k-t_k)}}
(-1)^{\beta_k}\binom{e_k-1-\overline{(j_k+l_k-t_k)}}{\beta_k}_{\gamma_k}\\
\notag&&\gamma_k^{\tfrac{\beta_k(\beta_k+1)}{2}+(j_k-s_k)\beta_k}
(x^{m_id\alpha_i}\otimes x^{m_kd\beta_k-td})](x^{a}y_t g\otimes x^a y_{\overline{j+l-t}}g^{t+1})\\
\notag&=& \frac{1}{m^2}\prod_{i=1}^{\theta}
(-1)^{l_i}\xi_{m_i}^{-l_i}\gamma^{m_i^2
\frac{l_i(l_i+1)}{2}} \prod_{i,k=1}^{\theta}\\
\notag&& [\sum_{\alpha_i=0}^{e_i-1-t_i}\sum_{\beta_{k}=1}^{e_k-1-\overline{j_k+l_k-t_k}}
(-1)^{\alpha_i+\beta_k}\binom{e_i-1-t_i}{\alpha_i}_{\gamma_i}
\binom{e_k-1-\overline{(j_k+l_k-t_k)}}{\beta_k}_{\gamma_k}\\
\notag \label{eq4.26}&& \gamma_i^{\tfrac{\alpha_i(\alpha_i+1)}{2}} \gamma_k^{\tfrac{\beta_k(\beta_k+1)}{2}+j_k\beta_k}(x^{m_id\alpha_i}\otimes x^{m_kd\beta_k-td})\\
&& \gamma^{t(t-l)}\sum_{s=0}^{m-1}\gamma^{-ts}\gamma^{-m_i^2s_i\alpha_i+m_k^2s_k\beta_k}](x^{a}y_t g\otimes x^a y_{\overline{j+l-t}}g^{t+1}).
\end{eqnarray}

Meanwhile,
$u_ju_l=\frac{1}{m}x^{a}\prod_{i=1}^{\theta}
(-1)^{l_i}\xi_{m_i}^{-l_i}\gamma^{m_i^2\frac{l_i(l_i+1)}{2}}\frac{1}{m}{[j_i, e_i-2-l_i]_{m_i}}\;y_{\overline{j+l}}g$. By definition,
$$y_{\overline{j+l}}=y_{m_1}^{\overline{j_1+l_1}}y_{m_2}^{\overline{j_2+l_2}}\cdots y_{m_\theta}^{\overline{j_\theta+l_\theta}}$$
where $\overline{j_i+l_i}$ is the remainder of $j_i+l_i$ divided by $e_i$ for $1\leq i\leq \theta.$ Therefore,
\begin{align*}
\D(y_{\overline{j+l}})&=\prod_{i=1}^{\theta}(1\otimes y_{m_i}+ y_{m_i}\otimes g^{m_i})^{\overline{j_i+l_i}}\\
&=\prod_{i=1}^{\theta}\sum_{t_i=0}^{\overline{j_i+l_i}}
\binom{\overline{j_i+l_i}}{t_i}_{\gamma_i}(1\otimes y_{m_i})^{\overline{j_i+l_i}-t_i} (y_{m_i}\otimes g^{m_i})^{t_i}\\
&=\prod_{i=1}^{\theta}\sum_{t_i=0}^{\overline{j_i+l_i}}
\binom{\overline{j_i+l_i}}{t_i}_{\gamma_i} y_{m_i}^{t_i}\otimes y_{m_i}^{\overline{j_i+l_i}-t_i}g^{m_it_i}.
\end{align*}
and
\begin{eqnarray*}
&&\D({[j_i, e_i-2-l_i]_{m_i}})\\
&=&(1\otimes 1 - \gamma_i^{j_i+1}x^{m_id}\otimes x^{m_id})
\cdots (1\otimes 1 - \gamma_i^{e_i-1+j_i-\overline{j_i+l_i}}x^{m_id}\otimes x^{m_id})\\
&=&\sum_{\alpha_i=0}^{e_i-1-\overline{j_i+l_i}}(-1)^{\alpha_i}
\binom{e_i-1-\overline{j_i+l_i}}{\alpha_i}_{\gamma_i}\gamma_{i}^{\tfrac{\alpha_i(\alpha_i-1)}{2}}
(\gamma_i^{j_i+1}x^{m_id}\otimes x^{m_id})^{\alpha_{i}}\\
&=&\sum_{\alpha_i=0}^{e_i-1-\overline{j_i+l_i}}(-1)^{\alpha_i}
\binom{e_i-1-\overline{j_i+l_i}}{\alpha_i}_{\gamma_i}\gamma_{i}^{\tfrac{\alpha_i(\alpha_i+1)}{2}+j_i\alpha_i}
(x^{m_id\alpha_i}\otimes x^{m_id\alpha_i}),
\end{eqnarray*}
we get
\begin{eqnarray*}
\D(u_ju_l)&=&\frac{1}{m}\D(x^a)\prod_{i=1}^{\theta}
(-1)^{l_i}\xi_{m_i}^{-l_i}\gamma^{m_i^2\frac{l_i(l_i+1)}{2}}\D({[j_i, e_i-2-l_i]_{m_i}})
\D(y_{\overline{j+l}})\D(g)\\
&=&\frac{1}{m} \prod_{i=1}^{\theta}[(-1)^{l_i}\xi_{m_i}^{-l_i}\gamma^{m_i^2\frac{l_i(l_i+1)}{2}}
\sum_{\alpha_i=0}^{e_i-1-\overline{j_i+l_i}}(-1)^{\alpha_i}
\binom{e_i-1-\overline{j_i+l_i}}{\alpha_i}_{\gamma_i}
\gamma_{i}^{\tfrac{\alpha_i(\alpha_i+1)}{2}+j_i\alpha_i}\\
&&\sum_{t_i=0}^{\overline{j_i+l_i}}
\binom{\overline{j_i+l_i}}{t_i}_{\gamma_i}(x^a\otimes x^{a})(x^{m_id\alpha_i}\otimes x^{m_id\alpha_i})(y_{m_i}^{t_i}\otimes y_{m_i}^{\overline{j_i+l_i}-t_i}g^{m_it_i})](g\otimes g)\\
&=& \frac{1}{m} \prod_{i=1}^{\theta}[(-1)^{l_i}\xi_{m_i}^{-l_i}\gamma^{m_i^2\frac{l_i(l_i+1)}{2}}
\sum_{t_i=0}^{\overline{j_i+l_i}}\sum_{\alpha_i=0}^{e_i-1-\overline{j_i+l_i}}(-1)^{\alpha_i}
\binom{e_i-1-\overline{j_i+l_i}}{\alpha_i}_{\gamma_i}\binom{\overline{j_i+l_i}}{t_i}_{\gamma_i}\\
&& \gamma_{i}^{\tfrac{\alpha_i(\alpha_i+1)}{2}+j_i\alpha_i}(x^{m_id\alpha_i}\otimes x^{m_id\alpha_i})(x^ay_{m_i}^{t_i}g\otimes x^ay_{m_i}^{\overline{j_i+l_i}-t_i}g^{m_it_i+1})].
\end{eqnarray*}

Clearly, for each $t$ satisfying $0\leq t_i\leq \overline{j_i+l_i}$,
\begin{eqnarray}
\label{eq5.26}&&\D(u_ju_l)_{(2,2+2t)\otimes (2+2t,2+2(j+l))}\\
\notag&=&  \frac{1}{m} \prod_{i=1}^{\theta}[(-1)^{l_i}\xi_{m_i}^{-l_i}\gamma^{m_i^2\frac{l_i(l_i+1)}{2}}
\sum_{\alpha_i=0}^{e_i-1-\overline{j_i+l_i}}(-1)^{\alpha_i}
\binom{e_i-1-\overline{j_i+l_i}}{\alpha_i}_{\gamma_i}\binom{\overline{j_i+l_i}}{t_i}_{\gamma_i}\\
\notag&& \gamma_{i}^{\tfrac{\alpha_i(\alpha_i+1)}{2}+j_i\alpha_i}(x^{m_id\alpha_i}\otimes x^{m_id\alpha_i})(x^ay_{m_i}^{t_i}\otimes x^ay_{m_i}^{\overline{j_i+l_i}-t_i}g^{m_it_i})](g\otimes g).
\end{eqnarray}

By the graded structure of $D\otimes D$,
$\D(u_i)\D(u_j)=\D(u_iu_j)$ if and only if
\begin{equation}\label{eq4.27}\D(u_i)\D(u_j)_{{(2, 2+2t)}\otimes {(2+2t,2+2(j+l))}}=0\end{equation} for all $t$ satisfying there is an $1\leq i\leq \theta$ such that $\overline{j_i+l_i}+1\leqslant t_i \leqslant e_i-1$ and
\begin{equation}\label{eq4.28}\D(u_iu_j)_{{(2, 2+2t)}\otimes
{(2+2t,2+2(j+l))}}=\D(u_i)\D(u_j)_{{(2, 2+2t)}\otimes
{(2+2t,2+2(j+l))}} \end{equation} for all $t$ satisfying $0\leqslant t_i \leqslant
\overline{j_i+l_i}$ for all $1\leq i\leq \theta$.

Now let's go back to equation \eqref{eq4.26} in which there is an item
\begin{eqnarray}\label{eq4.29}&&\sum_{s=0}^{m-1} \gamma^{-ts}\gamma^{-m_i^2s_i\alpha_i+m_k^2s_k\beta_k}\\
\notag&=&\prod_{z=1}^{\theta}\sum_{s_z=0}^{e_z-1} \gamma^{-t_zs_zm_z^2}\gamma^{-m_i^2s_i\alpha_i+m_k^2s_k\beta_k}\\
\notag&=&\left \{
\begin{array}{ll} \sum_{s_i=0}^{e_i-1} \gamma^{-s_im_i^2(\alpha_i+t_i)}\sum_{s_k=0}^{e_k-1}\gamma^{-s_km_m^2(\beta_k-t_k)}
\prod_{z\neq i,k}\sum_{s_z=0}^{e_z-1} \gamma^{-t_zs_zm_z^2} & \;\;\;\;i\neq k\\
\sum_{s_i=0}^{e_i-1}\gamma^{-m_i^2s_i(t_i+\alpha_i-\beta_i)}\prod_{z\neq i}\sum_{s_z=0}^{e_z-1} \gamma^{-t_zs_zm_z^2} &
\;\;\;\;i=k
\end{array}\right.
\end{eqnarray}

Therefore, in order to make this equality \eqref{eq4.29} not zero, we must have
$$\left \{
\begin{array}{ll} \alpha_i=-t_i,\;\;\beta_k=t_k & \;\;\;\;i\neq k
\\ \beta_i=\alpha_i+t_i &
\;\;\;\;i=k
\end{array}\right.$$
But in the expression of equality \eqref{eq4.26} one always have $0\leq \alpha_i\leq e_i-1-t_i$ which implies that $\alpha_i\neq -t_i$. Thus, as a conclusion, in the equality \eqref{eq4.26} we can assume that
$$i=k,\;\;\;\;\beta_i=\alpha_i+t_i,\;\;(1\leq i\leq \theta).$$
So, the equality can be simplified as

\begin{eqnarray}
\notag &&\frac{1}{m^2}\prod_{i=1}^{\theta}
(-1)^{l_i}\xi_{m_i}^{-l_i}\gamma^{m_i^2
\frac{l_i(l_i+1)}{2}} \prod_{i=1}^{\theta}
\sum_{\alpha_i=0}^{e_i-1-t_i}\sum_{\beta_{i}=0}^{e_i-1-\overline{j_i+l_i-t_i}}
(-1)^{\alpha_i+\beta_i}\binom{e_i-1-t_i}{\alpha_i}_{\gamma_i}\\
\notag &&
\binom{e_i-1-\overline{(j_i+l_i-t_i)}}{\beta_i}_{\gamma_i}\gamma_i^{\tfrac{\alpha_i(\alpha_i+1)}{2}
+\tfrac{\beta_i(\beta_i+1)}{2}+j_i\beta_i} (x^{m_id\alpha_i}\otimes x^{m_id\beta_i-t_im_id})\\ \notag&&\gamma^{t(t-l)}\sum_{s=0}^{m-1}\gamma^{-ts}
\gamma^{-m_i^2s_i(\alpha_i-\beta_i)}(x^{a}\prod_{i=1}^{\theta}y_{m_i}^{t_i} \otimes x^a \prod_{i=1}^{\theta}y_{m_i}^{\overline{j_i+l_i}-t_i}g^{m_it_i})(g\otimes g).
\end{eqnarray}

From this, we find the following fact: if $t_i\geq \overline{j_i+l_i}+1$ for some $i$, then $e_i-1-\overline{j_i+l_i-t_i}=t_i-1-\overline{j_i+l_i}$. So, $0\leq \beta_i\leq t_i-1-\overline{j_i+l_i}$ and thus $1-e_i\leq \beta_i-\alpha_i-t_i\leq -1-\overline{j_i+l_i}$ which contradicts to $\beta_i=\alpha_i+t_i$. So the equation \eqref{eq4.27} is proved. Under $\beta_i=\alpha_i+t_i$, we know that $$\prod_{i=1}^{\theta}\sum_{s=0}^{m-1} \gamma^{-ts}\gamma^{-m_i^2s_i\alpha_i+m_k^2s_i\beta_i}=e_1e_2\cdots e_{\theta}=m$$
and \eqref{eq4.26} can be simplified further
\begin{eqnarray}
\notag &&\frac{1}{m}\prod_{i=1}^{\theta}
(-1)^{l_i}\xi_{m_i}^{-l_i}\gamma^{m_i^2
\frac{l_i(l_i+1)}{2}} \prod_{i=1}^{\theta}
\sum_{\alpha_i=0}^{e_i-1-\overline{j_i+l_i}}
(-1)^{t_i}\binom{e_i-1-t_i}{\alpha_i}_{\gamma_i}\\
\notag &&
\binom{e_i-1-\overline{(j_i+l_i-t_i)}}{\alpha_i+t_i}_{\gamma_i}
\gamma_i^{\tfrac{\alpha_i(\alpha_i+1)}{2}+\tfrac{(\alpha_i+t_i)(\alpha_i+t_i+1)}{2}
+j_i(\alpha_i+t_i)+t_i(l_i-t_i)}\\ \notag&&(x^{m_id\alpha_i}\otimes x^{m_id\alpha_i}) (x^a y_{m_i}^{t_i} \otimes x^a y_{m_i}^{\overline{j_i+l_i}-t_i}g^{m_it_i})(g\otimes g).
\end{eqnarray}

Comparing with equation \eqref{eq5.26}, to prove the desired equation \eqref{eq4.28} it is enough to show the following combinatorial identity
\begin{eqnarray*} &&(-1)^{t_i+\alpha_i}\gamma_i^{\tfrac{(\alpha_i+t_i)(\alpha_i+t_i+1)}{2}
+t_i(j_i+l_i-t_i)}\binom{e_i-1-t_i}{\alpha_i}_{\gamma_i}\binom{e_i-1-\overline{(j_i+l_i-t_i)}}{\alpha_i+t_i}_{\gamma_i}\\
&=& \binom{e_i-1-\overline{j_i+l_i}}{\alpha_i}_{\gamma_i}
\binom{\overline{j_i+l_i}}{t_i}_{\gamma_i}
\end{eqnarray*}
which is true by (6) of Lemma \ref{l3.4}.\\

\noindent{$\bullet$ \emph{Step} 2 (Coassociative and couint).

Indeed, for each $0\leqslant j\leqslant m-1$
\begin{align*}
(\D\otimes \Id)\D(u_j)&=(\D\otimes \Id)(\sum_{k=0}^{m-1}\gamma^{k(j-k)}u_k\otimes x^{-kd}g^ku_{j-k})\\
                     &=\sum_{k=0}^{m-1}\gamma^{k(j-k)}(\sum_{s=0}^{m-1}\gamma^{s(k-s)}u_s\otimes x^{-sd}g^su_{k-s})\otimes x^{-kd}g^ku_{j-k}\\
                     &=\sum_{k,s=0}^{m-1}\gamma^{k(j-k)+s(k-s)}u_s\otimes x^{-sd}g^su_{k-s}\otimes x^{-kd}g^ku_{j-k},
\end{align*}
and
\begin{align*}
(\Id\otimes \D)\D(u_j)&=(\Id\otimes \D)(\sum_{s=0}^{m-1}\gamma^{s(j-s)}u_s\otimes x^{-sd}g^su_{j-s})\\
                     &=\sum_{s=0}^{m-1}\gamma^{s(j-s)}
                     u_s\otimes (\sum_{t=0}^{m-1}\gamma^{t(j-s-t)}x^{-sd}g^su_{t}\otimes x^{-sd}g^s x^{-td}g^{t}u_{j-s-t})\\
                     &=\sum_{s,t=0}^{m-1}\gamma^{s(j-s)+t(j-s-t)}u_s\otimes x^{-sd}g^su_{t}\otimes x^{-(s+t)d}g^{(s+t)}u_{j-s-t}.
\end{align*}

It is not hard to see that $(\D\otimes \Id)\D(u_j)=(\Id\otimes
\D)\D(u_j)$ for all $0\leqslant j \leqslant m-1$.  The verification of $(\epsilon\otimes \Id)\D(u_j)=(\Id\otimes
\epsilon)\D(u_j)=u_j$ is easy and it is omitted.\\

\noindent $\bullet$ \emph{Step} 3 (Antipode is an algebra anti-homomorphism).

Because $x$ and $g$ are group-like elements, we only check
$$S(u_{j+m_i})S(\phi_{m_i,j})=S(u_j)S(y_{m_i})=\xi_{m_i} S(y_{m_i})S(u_j)S(x^{m_id})$$ and
$$S(u_ju_l)=S(u_l)S(u_j)$$
for $1\leq i\leq \theta$ and $1\leq j,l\leq m-1$ here.

\noindent (1) \emph{The proof of} $S(u_{j+m_i})S(\phi_{m_i,j})=S(u_j)S(y_{m_i})=\xi_{m_i} S(y_{m_i})S(u_j)S(x^{m_id}).$

Clearly $u_jS(\phi_{m_i,j})=\phi_{m_i,j}u_j$ for all $i, j$ and thus
\begin{eqnarray*}&&S(u_{j+m_i})S(\phi_{m_i,j})\\
&=&x^{b}g^{m-1}\prod_{i=1}^{\theta}(-1)^{j_i}\xi_{m_i}^{-j_i}\gamma^{-
m_i^2\frac{j_i(j_i+1)}{2}}
x^{j_im_id}
g^{-j_im_i}u_j\\
&=&\phi_{m_i,j}S(u_{j+m_i})\end{eqnarray*}
here and the following of this proof $b=(1-m)d-\frac{\sum_{i=1}^{\theta}(e_i-1)m_i}{2}d.$

Through direct calculation, we have
\begin{align*}
&S(u_j)S(y_{m_i})\\
&=x^{b}g^{m-1}\prod_{i=1}^{\theta}[(-1)^{j_i}\xi_{m_i}^{-j_i}\gamma^{-
m_i^2\frac{j_i(j_i+1)}{2}}x^{j_im_id}g^{-j_im_i}]u_j\cdot (-y_{m_i}g^{-m_i})\\
&=-x^{b}g^{m-1}\prod_{i=1}^{\theta}[(-1)^{j_i}\xi_{m_i}^{-j_i}\gamma^{-
m_i^2\frac{j_i(j_i+1)}{2}}x^{j_im_id}g^{-j_im_i}](\xi_{m_i}^{-1}\gamma^{-jm_i}x^{m_id}y_{m_i}g^{-m_i}u_j)\\
&=-x^{b}g^{m-1}\prod_{i=1}^{\theta}[(-1)^{j_i}\xi_{m_i}^{-j_i}\gamma^{-
m_i^2\frac{j_i(j_i+1)}{2}}x^{j_im_id}g^{-j_im_i}]
(\xi_{m_i}^{-1}\gamma^{-m_i^2(j_i+1)}x^{m_id}g^{-m_i}y_{m_i}u_j)\\
&=x^{b}g^{m-1}(-1)^{j_1+\cdots+(j_i+1)+\cdots+j_{\theta}}\xi_{m_1}^{-j_1}\cdots \xi_{m_i}^{-(j_i+1)}\cdots \xi_{m_\theta}^{-j_\theta}\gamma_{1}^{\frac{j_1(j_1+1)}{2}}
\cdots \gamma_{i}^{\frac{(j_i+1)(j_1+2)}{2}}\cdots \gamma_{\theta}^{\frac{j_\theta(j_\theta+1)}{2}}\\
& \quad x^{j_1m_1d}\cdots x^{(j_i+1)m_id}\cdots x^{j_\theta m_\theta d}
g^{-j_1m_1}\cdots g^{-(j_i+1)m_i}\cdots g^{-j_\theta m_\theta}\phi_{m_i,j}u_{j+m_i}\\
&=\phi_{m_i,j}S(u_{j+m_i})
\end{align*}
and
\begin{align*}
&\xi_{m_i}S(y_{m_i})S(u_j)S(x^{m_id})\\
&= \xi_{m_i}(-y_{m_i}g^{-m_i})g^{m-1}x^b\prod_{i=1}^{\theta}[(-1)^{j_i}\xi_{m_i}^{-j_i}\gamma^{-
m_i^2\frac{j_i(j_i+1)}{2}}x^{j_im_id}g^{-j_im_i}]u_jx^{-m_id}\\
&=x^{b}g^{m-1}(-1)^{j_1+\cdots+(j_i+1)+\cdots+j_{\theta}}\xi_{m_1}^{-j_1}\cdots \xi_{m_i}^{-(j_i+1)}\cdots \xi_{m_\theta}^{-j_\theta}\gamma_{1}^{\frac{j_1(j_1+1)}{2}}
\cdots \gamma_{i}^{\frac{(j_i+1)(j_1+2)}{2}}\cdots \gamma_{\theta}^{\frac{j_\theta(j_\theta+1)}{2}}\\
& \quad x^{j_1m_1d}\cdots x^{(j_i+1)m_id}\cdots x^{j_\theta m_\theta d}
g^{-j_1m_1}\cdots g^{-(j_i+1)m_i}\cdots g^{-j_\theta m_\theta}\phi_{m_i,j}u_{j+m_i}\\
&=\phi_{m_i,j}S(u_{j+m_i}).
\end{align*}

\noindent (2) \emph{The proof of}  $S(u_ju_l)=S(u_l)S(u_j)$.

Define $\overline{\phi_{m_i,s}}:=1-\gamma_i^{s_i+1}x^{-m_id}$ for all $s\in \mathbb{Z}$. Using this notion,
\begin{align*}
x^{m_id}\overline{\phi_{m_i,s}}&=x^{m_id}(1-\gamma_i^{s_i+1}x^{-m_id})\\
&=-\gamma_i^{s_i+1}(1-\gamma_i^{(e_i-s_i-2)+1}x^{m_id})\\
&=-\gamma_i^{s_i+1}\phi_{m_i,e_i-s_i-2}.\end{align*}
And so

\begin{align*}
S(u_ju_l)&=S(\frac{1}{m}x^{a} \prod_{i=1}^{\theta}
(-1)^{l_i}\xi_{m_i}^{-l_i}\gamma^{m_i^2\frac{l_i(l_i+1)}{2}}[j_i,e_i-2-l_i]_{m_i}y_{\overline{j+l}}g)\\
&= \frac{1}{m}g^{-1}x^{-a}\prod_{i=1}^{\theta}
[(-1)^{l_i}\xi_{m_i}^{-l_i}\gamma^{m_i^2\frac{l_i(l_i+1)}{2}}
(-y_{m_i}g^{-m_i})^{\overline{j_i+l_i}}S([j_i,e_i-2-l_i]_{m_i})]\\
&=\frac{1}{m}g^{-1}x^{-a}\prod_{i=1}^{\theta}[(-1)^{l_i}\xi_{m_i}^{-l_i}
\gamma^{m_i^2\frac{l_i(l_i+1)}{2}}
(-1)^{\overline{j_i+l_i}}\gamma^{m_i^2\frac{\overline{j_i+l_i}(\overline{j_i+l_i}-1)}{2}}\\
&\quad S([j_i,e_i-2-l_i]_{m_i}) y_{m_i}^{\overline{j_i+l_i}}g^{-m_i\overline{j_i+l_i}}]\\
&=\frac{1}{m}x^{-a}\gamma^{j+l}\prod_{i=1}^{\theta}[(-1)^{l_i}\xi_{m_i}^{-l_i}
\gamma^{m_i^2\frac{l_i(l_i+1)}{2}}
(-1)^{\overline{j_i+l_i}}\gamma^{m_i^2\frac{\overline{j_i+l_i}(\overline{j_i+l_i}-1)}{2}}\\
&\quad S([j_i,e_i-2-l_i]_{m_i})] y_{\overline{j+l}}g^{-\overline{j+l}-1}\\
&=\frac{1}{m}x^{-a}\gamma^{j+l}\prod_{i=1}^{\theta}[(-1)^{l_i}\xi_{m_i}^{-l_i}
\gamma^{m_i^2\frac{l_i(l_i+1)}{2}}
(-1)^{\overline{j_i+l_i}}\gamma^{m_i^2\frac{\overline{j_i+l_i}(\overline{j_i+l_i}-1)}{2}}\\
&\quad (-1)^{e_i-1-\overline{j_i+l_i}}
\gamma^{m_i^2\frac{(e_i-1-\overline{j_i+l_i})(\overline{j_i+l_i}-2j_i-e_i)}{2}}
x^{-(e_i-1-\overline{j_i+l_i})m_id}[l_i,e_i-2-j_i]_{m_i}] y_{\overline{j+l}}g^{-\overline{j+l}-1}\\
&=\frac{1}{m}x^{-a}\gamma^{j+l}\prod_{i=1}^{\theta}[(-1)^{l_i}\xi_{m_i}^{-l_i}
\gamma^{m_i^2\frac{l_i(l_i+1)}{2}}\gamma^{m_i^2(j_i^2+j_il_i-l_i)}\\
&\quad x^{-(e_i-1-\overline{j_i+l_i})m_id}[l_i,e_i-2-j_i]_{m_i}] y_{\overline{j+l}}g^{-\overline{j+l}-1}.
\end{align*}
Here the last equality follows from
\begin{align*} &(-1)^{e_i-1}\gamma^{m_i^2\frac{\overline{j_i+l_i}(\overline{j_i+l_i}-1)}{2}}
\gamma^{m_i^2\frac{(e_i-1-\overline{j_i+l_i})(\overline{j_i+l_i}-2j_i-e_i)}{2}}\\
&=\quad \gamma^{m_i^2(j_i^2+j_il_i-l_i)}.
\end{align*}
Now let's compute the other side.
\begin{align*}
S(u_l)S(u_j)&=g^{m-1}x^b\prod_{i=1}^{\theta}[(-1)^{l_i}\xi_{m_i}^{-l_i}\gamma^{-
m_i^2\frac{l_i(l_i+1)}{2}}x^{l_im_id}g^{-l_im_i}]u_l\\
&\quad g^{m-1}x^b\prod_{i=1}^{\theta}[(-1)^{j_i}\xi_{m_i}^{-j_i}\gamma^{-
m_i^2\frac{j_i(j_i+1)}{2}}x^{j_im_id}g^{-j_im_i}]u_j\\
&= g^{m-1}\prod_{i=1}^{\theta}[(-1)^{l_i+j_i}\xi_{m_i}^{-l_i-j_i}\gamma^{-
m_i^2[\frac{l_i(l_i+1)}{2}+\frac{j_i(j_i+1)}{2}]}x^{(l_i-j_i)m_id}g^{-l_im_i}]\\
&\quad u_lg^{m-1-\sum_{i=1}^{\theta}j_im_i}u_j\\
&= \gamma^{-l-lj}\prod_{i=1}^{\theta}[(-1)^{l_i+j_i}\xi_{m_i}^{-l_i-j_i}\gamma^{-
m_i^2[\frac{l_i(l_i+1)}{2}+\frac{j_i(j_i+1)}{2}]}x^{(l_i+j_i)m_id}g^{-l_im_i-j_im_i}]\\
&\quad g^{-2}x^{2d}u_lu_j\\
&=\gamma^{-l-lj}\prod_{i=1}^{\theta}[(-1)^{l_i+j_i}\xi_{m_i}^{-l_i-j_i}\gamma^{-
m_i^2[\frac{l_i(l_i+1)}{2}+\frac{j_i(j_i+1)}{2}]}x^{(l_i+j_i)m_id}g^{-l_im_i-j_im_i}]\\
&\quad g^{-2}x^{2d}\frac{1}{m}x^{a} \prod_{i=1}^{\theta}
(-1)^{j_i}\xi_{m_i}^{-j_i}\gamma^{m_i^2\frac{j_i(j_i+1)}{2}}[l_i,e_i-2-j_i]_{m_i}y_{\overline{j+l}}g\\
&=\gamma^{-l-lj}\frac{1}{m}\prod_{i=1}^{\theta}[(-1)^{l_i}\xi_{m_i}^{-l_i-2j_i}\gamma^{-
m_i^2[\frac{l_i(l_i+1)}{2}]}[l_i,e_i-2-j_i]_{m_i}
x^{(e_i-1-\overline{(l_i+j_i)})m_id}g^{-\overline{l_i+j_i}m_i}]\\
&\quad g^{-2}\frac{1}{m}x^{-a} y_{\overline{j+l}}g\\
&=\frac{1}{m}\prod_{i=1}^{\theta}[(-1)^{l_i}\xi_{m_i}^{-l_i-2j_i}\gamma^{-
m_i^2(\frac{l_i(l_i+1)}{2})-l_im_i-l_ij_im_i^2+m_i^2(l_i+j_i)^2
+2(l_i+j_i)m_i}\\
&\quad[l_i,e_i-2-j_i]_{m_i}
x^{(e_i-1-\overline{(l_i+j_i)})m_id}]x^{-a} y_{\overline{j+l}}g^{-(\overline{j+l}+1)}\\
&= \frac{1}{m}x^{-a}\gamma^{j+l}\prod_{i=1}^{\theta}[(-1)^{l_i}\xi_{m_i}^{-l_i}
\gamma^{m_i^2\frac{l_i(l_i+1)}{2}}\gamma^{m_i^2(j_i^2+j_il_i-l_i)}\\
&\quad x^{-(e_i-1-\overline{j_i+l_i})m_id}[l_i,e_i-2-j_i]_{m_i}] y_{\overline{j+l}}g^{-\overline{j+l}-1}.
\end{align*}
where the fifth equality follows from $$x^{a+2d}=x^{-\frac{2+\sum_{i=1}^{\theta}(e_i-1)m_i}{2}d+2d}=x^{-a-\sum_{i=1}^{\theta}(e_i-1)m_id}$$ and the last equality is followed by
\begin{align*} &\xi_{m_i}^{-2j_i}\gamma^{-
m_i^2(\frac{l_i(l_i+1)}{2})-l_im_i-l_ij_im_i^2+m_i^2(l_i+j_i)^2
+2(l_i+j_i)m_i}\\
&=\gamma^{
m_i^2(\frac{l_i(l_i+1)}{2})-m_ij_i-m_i^2l_i(l_i+1)-l_im_i-l_ij_im_i^2+m_i^2(l_i+j_i)^2
+2(l_i+j_i)m_i}\\
&= \gamma^{m_i^2\frac{l_i(l_i+1)}{2}+m_i^2(j_i^2+j_il_i-l_i)+j_im_i+l_im_i}.
\end{align*}

The proof is done.\\

\noindent $\bullet$ \emph{Step} 4
($(S*\Id)(u_j)=(\Id*S)(u_j)=\epsilon(u_j)$).

In fact,

\begin{align*}
(S*\Id)(u_0)&=\sum_{j=0}^{m-1}S(\gamma^{-j^2}u_j)x^{-jd}g^ju_{-j} \\
&=\sum_{j=0}^{m-1}\gamma^{-j^2}g^{m-1}x^{b}\prod_{i=1}^{\theta}[(-1)^{j_i}\xi_{m_i}^{-j_i}\gamma^{-
m_i^2\frac{j_i(j_i+1)}{2}}x^{j_im_id}g^{-j_im_i}]u_jx^{-jd}g^ju_{-j} \\
&=\sum_{j=0}^{m-1}g^{m-1}x^{b}\prod_{i=1}^{\theta}[(-1)^{j_i}\xi_{m_i}^{-j_i}\gamma^{-
m_i^2\frac{j_i(j_i+1)}{2}}]u_ju_{-j} \\
           &=\sum_{j=0}^{m-1}g^{m-1}x^{b}\prod_{i=1}^{\theta}[(-1)^{j_i}\xi_{m_i}^{-j_i}\gamma^{-
m_i^2\frac{j_i(j_i+1)}{2}}]\\
&\quad \frac{1}{m}x^{a} \prod_{i=1}^{\theta}
(-1)^{-j_i}\xi_{m_i}^{j_i}\gamma^{m_i^2\frac{-j_i(-j_i+1)}{2}}[j_i,e_i-2-j_i]_{m_i}g\\
           &=\frac{1}{m}x^{a+b}g^m
           \prod_{i=1}^{\theta}[\sum_{j_i=0}^{e_i-1}\gamma_i^{j_i}[j_i,e_i-2-j_i]_{m_i}]\\
           &=\frac{1}{m}x^{-\sum_{i=0}^{\theta}(e_i-1)m_id} \prod_{i=1}^{\theta}[\sum_{j_i=0}^{e_i-1}\gamma_i^{j_i}]j_i-1,j_i-1[_{m_i}] \\
           &=\frac{1}{m}x^{-\sum_{i=0}^{\theta}(e_i-1)m_id}\prod_{i=1}^{\theta}e_ix^{(e_i-1)m_id} \quad\quad (\textrm{Lemma} \;\ref{l3.4}\;(3)) \\
           &=1\\
           &=\epsilon(u_{0}).
\end{align*}
And,
\begin{align*}
(\Id*S)(u_0)&=\sum_{j=0}^{m-1}\gamma^{-j^2}u_j S(x^{-jd}g^ju_{-j}) \\
           &=\sum_{j=0}^{m-1}\gamma^{-j^2}u_j S(u_{-j})S(g^j)x^{jd} \\
           &=\sum_{j=0}^{m-1}\gamma^{-j^2}u_j  g^{m-1}x^{b}\prod_{i=1}^{\theta}[(-1)^{-j_i}\xi_{m_i}^{j_i}\gamma^{-
m_i^2\frac{-j_i(-j_i+1)}{2}}x^{-j_im_id}g^{j_im_i}]u_{-j}g^{-j}x^{jd}\\
           &=\sum_{j=0}^{m-1} x^{(1-m)d+\frac{\sum_{i=1}^{\theta}(e_i-1)m_i}{2}d}\prod_{i=1}^{\theta}[(-1)^{-j_i}\xi_{m_i}^{j_i}\gamma^{-
m_i^2\frac{-j_i(-j_i+1)}{2}}\gamma^{-j_im_i}]g^{m-1}u_ju_{-j}\\
&=\sum_{j=0}^{m-1} x^{(1-m)d+\frac{\sum_{i=1}^{\theta}(e_i-1)m_i}{2}d}\prod_{i=1}^{\theta}[(-1)^{-j_i}\xi_{m_i}^{j_i}\gamma^{-
m_i^2\frac{-j_i(-j_i+1)}{2}}\gamma^{-j_im_i}]g^{m-1}\\
&\quad \frac{1}{m}x^{a} \prod_{i=1}^{\theta}
(-1)^{-j_i}\xi_{m_i}^{j_i}\gamma^{m_i^2\frac{-j_i(-j_i+1)}{2}}[j_i,e_i-2-j_i]_{m_i}g\\
&=\sum_{j=0}^{m-1}\frac{1}{m}\prod_{i=1}^{\theta}
\xi_{m_i}^{2j_i}\gamma^{-j_im_i}]j_i-1,j_i-1[_{m_i}\\
&=\frac{1}{m}\prod_{i=1}^{\theta}\sum_{j_i=0}^{e_i-1}
]j_i-1,j_i-1[_{m_i}\\
&=\frac{1}{m}\prod_{i=1}^{\theta}e_i \quad\quad (\textrm{Lemma}\; \ref{l3.4}\;(1))\\
           &=1 \\
           &=\epsilon(u_{0}).
\end{align*}

For $1\leqslant j \leqslant m-1$,
\begin{align*}
(S*\Id)(u_{j})&=\sum_{k=0}^{m-1}\gamma^{k(j-k)}S(u_k)x^{-kd}g^ku_{j-k}\\
               &=\sum_{j=0}^{m-1}\gamma^{k(j-k)} g^{m-1}x^{b}\prod_{i=1}^{\theta}[(-1)^{k_i}\xi_{m_i}^{-k_i}\gamma^{-
m_i^2\frac{k_i(k_i+1)}{2}}x^{k_im_id}g^{-k_im_i}]u_kx^{-kd}g^ku_{j-k}\\
&=\sum_{k=0}^{m-1}\gamma^{k(j-k)}g^{m-1}x^{b}\prod_{i=1}^{\theta}[(-1)^{k_i}\xi_{m_i}^{-k_i}\gamma^{-
m_i^2\frac{k_i(k_i+1)}{2}}\gamma^{k_i^2m_i^2}]u_ku_{j-k}\\
&=\sum_{k=0}^{m-1}\gamma^{k(j-k)}g^{m-1}x^{b}\prod_{i=1}^{\theta}[(-1)^{k_i}\xi_{m_i}^{-k_i}\gamma^{-
m_i^2\frac{k_i(k_i+1)}{2}}\gamma^{k_i^2m_i^2}]\\
& \quad  \frac{1}{m}x^{a} \prod_{i=1}^{\theta}[
(-1)^{j_i-k_i}\xi_{m_i}^{-j_i+k_i}\gamma^{m_i^2\frac{(j_i-k_i)(j_i-k_i+1)}{2}}
[k_i,e_i-2-j_i+k_i]_{m_i}]y_jg\\
&=\frac{1}{m}x^{a+b}\sum_{k=0}^{m-1}\prod_{i=1}^{\theta}[
(-1)^{j_i}\xi_{m_i}^{-j_i}\gamma^{m_i^2\frac{j_i^2+j_i}{2}+j_im_i-k_im_i^2}
[k_i,e_i-2-j_i+k_i]_{m_i}]y_j\\
&=\frac{1}{m}x^{a+b}\prod_{i=1}^{\theta}[
(-1)^{j_i}\xi_{m_i}^{-j_i}\gamma^{m_i^2\frac{j_i^2+j_i}{2}+j_im_i}]y_j\\
&\quad \prod_{i=1}^{\theta}[\sum_{k_i=0}^{e_i-1}\gamma_i^{k_i}]k_i-1-j_i,k_i-1[_{m_i}]\\
&=0\quad\quad (\textrm{Lemma}\; \ref{l3.4}\;(5))\\
&=\e(u_j)
\end{align*}
\begin{align*}
(\Id*S)(u_{j})&=\sum_{k=0}^{m-1}\gamma^{k(j-k)}u_kS(u_{j-k})g^{-k}x^{kd}\\
               &=\sum_{k=0}^{m-1}\gamma^{k(j-k)}u_k g^{m-1}x^{b}\prod_{i=1}^{\theta}[(-1)^{j_i-k_i}\xi_{m_i}^{k_i-j_i}\gamma^{-
m_i^2\frac{(j_i-k_i)(j_i-k_i+1)}{2}}\\
&\quad \quad x^{(j_i-k_i)m_id}g^{-(j_i-k_i)m_i}]u_{j-k}g^{-k}x^{kd}\\
&=\sum_{k=0}^{m-1}u_k g^{m-1}x^{b}\prod_{i=1}^{\theta}[(-1)^{j_i-k_i}\xi_{m_i}^{k_i-j_i}\gamma^{-
m_i^2\frac{(j_i-k_i)(j_i-k_i+1)}{2}}\\
&\quad \quad x^{j_im_id}g^{-j_im_i}]u_{j-k}\\
&=\sum_{k=0}^{m-1} \gamma^{-k}g^{m-1}x^{(1-m)d+\frac{\sum_{i=1}^{\theta}(e_i-1)m_i}{2}d}
\prod_{i=1}^{\theta}[(-1)^{j_i-k_i}\xi_{m_i}^{k_i-j_i}\gamma^{-
m_i^2\frac{(j_i-k_i)(j_i-k_i+1)}{2}}\\
&\quad \quad x^{j_im_id}\gamma^{-kj_im_i}g^{-j_im_i}]u_ku_{j-k}\\
&=\sum_{k=0}^{m-1} \gamma^{-k}g^{m-1}x^{(1-m)d+\frac{\sum_{i=1}^{\theta}(e_i-1)m_i}{2}d}
\prod_{i=1}^{\theta}[(-1)^{j_i-k_i}\xi_{m_i}^{k_i-j_i}\gamma^{-
m_i^2\frac{(j_i-k_i)(j_i-k_i+1)}{2}}\\
&\quad\quad x^{j_im_id}\gamma^{-kj_im_i}g^{-j_im_i}]\\
&\quad\quad \frac{1}{m}x^{a} \prod_{i=1}^{\theta}[
(-1)^{j_i-k_i}\xi_{m_i}^{-j_i+k_i}\gamma^{m_i^2\frac{(j_i-k_i)(j_i-k_i+1)}{2}}
[k_i,e_i-2-j_i+k_i]_{m_i}]y_jg\\
&=\frac{1}{m}x^{-md}\sum_{k=0}^{m-1} \gamma^{-k}
\prod_{i=1}^{\theta}[\xi_{m_i}^{2(k_i-j_i)}\gamma^{-kj_im_i+j_im_i}
x^{j_im_id}g^{-j_im_i}[k_i,e_i-2-j_i+k_i]_{m_i}]g^{m}y_j\\
&=\frac{1}{m}
\prod_{i=1}^{\theta}[\xi_{m_i}^{-2j_i}\gamma^{j_im_i}
x^{j_im_id}g^{-j_im_i}]\prod_{i=1}^{\theta}
[\sum_{k_i=0}^{e_i-1}\gamma_i^{k_ij_i}]k_i-1-j_i,k_i-1]_{m_i}]y_j\\
&=0\quad\quad (\textrm{Lemma}\; \ref{l3.4}\;(5))\\
&=\e(u_j).
\end{align*}

By steps 1, 2, 3, 4, $D(\underline{m}, d, \gamma)$ is a Hopf algebra.\qed

\begin{proposition}\label{p4.8} Under above notations, the Hopf algebra $D(\underline{m}, d, \gamma)$ has the following properties.
\begin{itemize} \item[(1)] The Hopf algebra $D(\underline{m}, d, \gamma)$ is prime with PI-degree $2m$.
\item[(2)] The Hopf algebra $D(\underline{m}, d, \gamma)$ has a $1$-dimensional representation whose order is $2m$.
\item[(3)] The Hopf algebra $D(\underline{m}, d, \gamma)$ is not pointed and its coradical is not a Hopf subalgebra if $m>1$.
\item[(4)] The Hopf algebra $D(\underline{m}, d, \gamma)$ is pivotal, that is, its representation category is a pivotal tensor category.
\end{itemize}
\end{proposition}
\begin{proof} (1) Recall that the Hopf algebra $D=D(\underline{m}, d, \gamma)=\bigoplus_{i=0}^{2m}D_{i}^{l}$ is strongly $\Z_{2m}$-graded with
 \[{D_i^l=}
\begin{cases}
\k [x^{\pm 1}, y_{m_1},\ldots,y_{m_\theta}] g^{\tfrac{i}{2}}, & i=\textrm{even},\\
\sum_{s=0}^{m-1}\k[ x^{\pm 1}] g^{\tfrac{i-1}{2}}u_s,
& i=\textrm{odd}.
\end{cases}\]
So the algebra $D$ meets the the initial condition of Lemma \ref{l2.10}. Using the notation given in the Lemma \ref{l2.10}, we find that
$$\chi\triangleright y_{m_i}=\xi_{m_i}^{-1}x^{-m_id}y_{m_i}$$
for all $1\leq i\leq \theta.$ This indeed implies the action of $\Z_{2m}$ on $D^{l}_{0}=\k[x^{\pm 1},y_{m_1},\ldots,y_{m_\theta}]$ is faithful. Therefore, by (c) and (d) of Lemma \ref{l2.10}, $D$ is prime with PI-degree $2m$.

(2) This 1-dimensional representation can be given through left homological integrals. In fact, the direct computation shows that the right module structure of left homological integrals is given by:
$$\int_{D}^{l}=D/(x-1,y_{m_1},\ldots,y_{m_\theta},u_1,\ldots,u_{m-1},
u_0-\prod_{i=1}^{\theta}\xi_{m_i}^{(e_i-1)},g-\prod_{i=1}^{\theta}\gamma^{-m_i}).$$
Through the relation that $\xi_{m_i}=\sqrt{\gamma^{m_i}}$ it is not hard to see that the $\io(D)=2m$.

(3) Through direct computations, we find that the subspace $C_{m}(d)$ spanned by $\{(x^{-d}g)^{i}u_j|0\leq i,j\leq m-1\}$ is a simple coalgebra (see Proposition \ref{p7.6} for a detailed proof of this fact) and the coradical of $D$ equals to
$$\bigoplus_{i\in \Z,\;0\leq j\leq m-1}x^{i}g^{j}\oplus(\bigoplus_{i\in \Z,\;0\leq j\leq m-1}x^{i}g^{j}C_{m}(d)).$$
Since $m>1$, it has a simple subcoalgebra $C_{m}(d)$ with dimension $m^2>1$. Therefore, $D$ is not pointed.  Its coradical is not a Hopf subalgebra since it is clear it is not closed under multiplication.

(4) See the proof of (3) of Proposition \ref{p7.14} where we built the result through proving that $D$ being pivotal.
\end{proof}

\begin{remark}\label{r4.9} \begin{itemize}
\item[(1)] \emph{As a special case, through takeing $m=1$ one is not hard to see that the Hopf algebra $D$ constructed above is just the infinite dihedral group algebra $\k \mathbb{D}$. This justifies the choice of the notation ``D".}
\item[(2)]\emph{It is not hard to see the other new examples, i.e., $T(\underline{m},t,\xi),\;B(\underline{m},\omega,\gamma)$, are pivotal since they are pointed and thus the proof of this fact become easier. In fact, keep the notations above, we have $$S^2(h)=(\prod_{i=1}^{\theta} g^{tm_i})h(\prod_{i=1}^{\theta} g^{tm_i})^{-1}$$ for $h\in T(\underline{m},t,\xi)$ and $$S^2(h)=(\prod_{i=1}^{\theta} g^{m_i})h(\prod_{i=1}^{\theta} g^{m_i})^{-1}$$ for $h\in B(\underline{m},\omega,\gamma)$. Through applying Lemma \ref{l2.16}, we get the result.}
\end{itemize}
\end{remark}

Now let $m'\in \N$ and $\{m_1',\ldots,m'_{\theta'}\}$ a fraction of $m'$. As before, we need to compare different fractions of Hopf algebras $D(m,d,\gamma)$. Also, we denote the greatest common divisors of $\{m_1,\ldots,m_\theta\}$ and $\{m_1',\ldots,m'_{\theta'}\}$ by $m_0$ and $m_0'$ respectively. Parallel to case of generalized Liu algebras, we have the following observation.
\begin{proposition} As Hopf algebras,  $D(\underline{m},d,\gamma)\cong D(\underline{m'},d',\gamma')$ if and only if $m=m', \;\theta=\theta',\; d m_0=d' m_0'$ and $\gamma^{m_0^2}=(\gamma')^{(m_0')^2}$.\end{proposition}
\begin{proof} By Proposition \ref{iliu}, it it enough to show that $D(\underline{m},d,\gamma)\cong D(\underline{m'},d',\gamma')$ if and only if their Hopf subalgebras $B(\underline{m},md,\gamma)$ and $B(\underline{m'},m'd',\gamma')$ are isomorphic. It is clear the isomorphism of $D(\underline{m},d,\gamma)$ and $D(\underline{m'},d',\gamma')$ will imply the isomorphism between $B(\underline{m},md,\gamma)$ and $B(\underline{m'},m'd',\gamma')$. Conversely, assume that $B(\underline{m},md,\gamma)\cong B(\underline{m'},m'd',\gamma')$. By Proposition \ref{p6.11}, $D(\underline{m},d,\gamma)$ is determined by $B(\underline{m},md,\gamma)$ entirely. Therefore, $D(\underline{m},d,\gamma)\cong D(\underline{m'},d',\gamma')$ too.
\end{proof}

At last, we point out the examples we constructed until now are not the same.
\begin{proposition} If $m>1$, the Hopf algebras $T(\underline{m'},t,\xi), \;B(\underline{m''},\omega,\gamma'')$ and $D(\underline{m},d,\gamma)$ are not isomorphic to each other.
\end{proposition}
\begin{proof} Since $m>1$, $D(\underline{m},d,\gamma)$ is not pointed by Proposition \ref{p4.8} (3) while $T(\underline{m'},t,\xi)$ and $B(\underline{m''},\omega,\gamma'')$ are pointed. Therefore, $D(\underline{m},d,\gamma)\not\cong T(\underline{m'},t,\xi)$ and $D(\underline{m},d,\gamma)\not\cong B(\underline{m''},\omega,\gamma'').$ Comparing the number of group-likes, we know that $T(\underline{m'},t,\xi)\not\cong B(\underline{m''},\omega,\gamma'')$ either.
\end{proof}

\section{Ideal cases}
In this section, we always assume that $H$ is a prime Hopf algebra of GK-dimension one satisfying (Hyp1) and (Hyp2). So by (Hyp1), $H$ has a $1$-dimensional representation
$$\pi:H\longrightarrow \k$$
whose order equals to PI-deg$(H).$ Recall that in the Subsection \ref{subs2.2}, we already gave the definition of $\pi$-order $\ord (\pi)$ and $\pi$-minor $\mi(\pi)$. The aim of this section is to classify $H$ in the following two ideal cases:
\begin{center} $\mi (\pi)=1$ or $\ord (\pi)=\mi (\pi)$.\end{center}
If moreover assume that $H$ is regular, then the main result of \cite{BZ} is to classify $H$ in ideal cases. Here we apply similar program to classify prime Hopf algebras which may be not regular.

\subsection{Ideal case one: $ \textrm{min}(\pi)=1$.} In this subsection, $H$ is a prime Hopf algebra of GK-dimension one satisfying (Hyp1), (Hyp2) and $ \mi(\pi)=1$.  Let PI.deg$(H)=n>1$ (if $=1$, then it is clear that $H$ is commutative and thus $H$ is the coordinate algebra of connected algebraic group of dimension one). Recall that by the equation \eqref{eq2.3}, $H$ is an $\Z_{n}$-bigraded algebra
$$H=\bigoplus_{i,j=0}^{n-1}H_{ij,\pi}.$$
Here and the following we write $H_{ij,\pi}$ just as $H_{ij}$ for simple.
\begin{lemma} Under above notations, the subalgebra $H_{00}$ is a Hopf subalgebra which is isomorphic to either $\k[x]$ or $\k[x^{\pm 1}].$
\end{lemma}
\begin{proof} Since $\textrm{min}(\pi)=1$, $H_{0}^{l}=H_{0}^{r}=H_{00}$. By (1) and (3) of Lemma \ref{l2.8}, $H_{00}$ is stable under the operations $\D$ and $S$. This implies that $H_{00}$ is a Hopf subalgebra. By Lemma \ref{l2.7} and its proof, we know that $H_{00}$ is a commutative domain of GK-dimension one. So $H_{00}$ is the coordinate algebra of connected algebraic group of dimension one. Thus it is isomorphic to either $\k[x]$ or $\k[x^{\pm 1}].$
\end{proof}

Therefore, we have a dichotomy on the structure of $H$ now.
\begin{definition} \emph{Let  $H$ be a prime Hopf algebra of GK-dimension one satisfying (Hyp1), (Hyp2) and $ \textrm{min}(\pi)=1$.}
\begin{itemize} \item[(a)] \emph{We call }$H$ additive \emph{if $H_{00}$ is the coordinate algebra of the additive group, that is, $H_{00}=\k[x]$.}
\item[(b)] \emph{We call} $H$ multiplicative \emph{if $H_{00}$ is the coordinate algebra of the multiplicative group, that is, $H_{00}=\k[x^{\pm 1}]$.}
\end{itemize}
\end{definition}

\begin{remark} \emph{In both \cite{BZ} and \cite{WLD}, the additive $H$ was called} primitive \emph{while the multiplicative $H$ was called} group-like. \emph{Here we used a slightly different terminology for intuition.}
\end{remark}

If we check the proof of the \cite[Propositions 4.2, 4.3]{BZ} carefully, then one can find that these propositions are still valid even we remove the requirement about regularity. So we state the following result, the same as \cite[Propositions 4.2, 4.3]{BZ}, without proof.

\begin{proposition}\label{p5.4} Let  $H$ be a prime Hopf algebra of GK-dimension one with PI-deg$(H)=n>1$ and satisfies (Hyp1), (Hyp2) and $\textrm{min}(\pi)=1$. Then
\begin{itemize} \item[(a)] If $H$ is additive, then $H\cong T(n,0,\xi)$ of Subsection 2.3.
\item[(b)] If $H$ is multiplicative, then $H\cong \k \mathbb{D}$ of Subsection 2.3.
\end{itemize}
In particular, such $H$ must be regular.
\end{proposition}

\subsection{Ideal case two: $\textrm{ord}(\pi)=\textrm{min}(\pi)$.} In this subsection, $H$ is a prime Hopf algebra of GK-dimension one satisfying (Hyp1), (Hyp2) and $ n:=\ord(\pi)=\textrm{min}(\pi)>1$ (if $=1$, then clearly $H$ commutative by our (Hyp2)).
Recall that we have the following bigrading
$$H=\bigoplus_{i,j=0}^{n-1}H_{ij}.$$
The following is some parts of \cite[Proposition 5.2, Theorem 5.2]{BZ}, which are proved without the hypothesis on regularity and thus they are true in our case.
\begin{lemma}\label{l5.5} Retain the notations above. Then
\begin{itemize}\item[(a)] The center of $H$ equals to $H_{0}:=H_{00}$.
\item[(b)] The center of $H$ is a Hopf subalgebra.
\end{itemize}
\end{lemma}

The statement $(b)$ in this lemma also imply that we are in the same situation as ideal case one now: $H$ is either additive or multiplicative. No matter what kind of $H$ is, $H_{ij}$ is a free $H_0$-module of rank one (see the analysis given in \cite[Page 287]{BZ}), that is
$$H=\bigoplus_{i,j=0}^{n-1}H_{ij}=\bigoplus_{i,j=0}^{n-1}H_{0}u_{ij}
=\bigoplus_{i,j=0}^{n-1}u_{ij}H_{0},$$
and the action of winding automorphism (relative to $\pi$) is given by
$$\Xi_{\pi}^{l}(u_{ij}a)=\xi^{i}u_{ij}a,\quad\quad \textrm{and} \quad\quad
\Xi_{\pi}^{r}(u_{ij}a)=\xi^{j}u_{ij}a$$
for $a\in H_0$ and $\xi$ a primitive $n$th root of unity. Due to \cite[Proposition 6.2]{BZ}, all these elements $u_{ij}\;(0\leq i,j\leq n-1)$ are normal. Moreover, by \cite[Lemma 6.2]{BZ}, they satisfy the following relation:
\begin{equation}\label{eq5.1} u_{ij}u_{i'j'}=\xi^{i'j-ij'}u_{i'j'}u_{ij}.
\end{equation}

By Lemma \ref{l5.5}, $H_{00}$ is a normal Hopf subalgebra of $H$ which implies that there is an exact sequence of Hopf algebras
\begin{equation}\label{eq5.2} \k\longrightarrow H_{00}\longrightarrow H\longrightarrow \overline{H}\longrightarrow \k,
\end{equation}
where $\overline{H}=H/HH_{00}^{+}$ and by definition $H_{00}^{+}=H_{00}\bigcap \ker \e.$ As one of basic observations of this paper, we have the following result.
\begin{lemma}\label{l5.6} As a Hopf algebra, $\overline{H}$ is isomorphic to a fraction version of a Taft algebra $T(n_1,\ldots,n_\theta,\xi)$ for $n_1,\ldots,n_{\theta}$ a fraction of $n$.
\end{lemma}
\begin{proof} Denote the image of $u_{ij}$ in $\overline{H}$ by $v_{ij}$ for $0\leq i,j\leq n-1$.
Due to $H$ is bigraded, $$\overline{H}=\bigoplus_{i,j=0}^{n-1}\overline{H}_{ij}=\bigoplus_{i,j=0}^{n-1}\k v_{ij}.$$
Let $g=v_{11}$. Then by (a), (b) and (e) of \cite[Proposition 6.6]{BZ}, which are still true even $H$ is not regular, these elements $v_{ij}$ can be chosen to satisfy $$g^{n}=1,\;\;\;\;v_{ii}=g^{i},\;\;(0\leq i\leq n-1),\;\;\;\;v_{ij}=g^{i}v_{0(j-i)},\;\;(0\leq i\neq j\leq n-1) $$
and $$v_{ij}^{n}=0,\;\;\;\;(0\leq i\neq j\leq n-1).$$
Moreover, one can use (1), (4) and (5) of Lemma \ref{l2.8} and the axioms for a coproduct to show that $g$ is group-like and
$$\D(v_{ij})=v_{ii}\otimes v_{ij}+v_{ij}\otimes v_{jj}+\sum_{s\neq i,j}c_{ss}^{ij}v_{is}\otimes v_{sj}=g^{i}\otimes v_{ij}+v_{ij}\otimes g^{j}+\sum_{s\neq i,j}c_{ss}^{ij}v_{is}\otimes v_{sj}$$
for some $c_{ss}^{ij}\in \k$ and $0\leq i\neq j\leq n-1$ (see also \cite[Lemma 6.5]{BZ} for a explicit proof). Using this formula for coproduct, it is not hard to see that $\overline{H}$ is a pointed Hopf algebra with $G(\overline{H})=\{g^{i}|0\leq i\leq n-1\}.$

Let $\overline{H}_{i}^{l}:=\bigoplus _{j=0}^{n-1} \overline{H}_{ij}$ and then through inheriting the strongly graded property of $H$, we know that $\overline{H}=\bigoplus_{i=0}^{n-1}\overline{H}_{i}^{l}$ is strongly graded. We want to consider the subalgebra $\overline{H}_{0}^{l}=\bigoplus_{j=0} \k v_{0j}$. For this, we take the following linear map
$$\pi':\overline{H}\longrightarrow \k G(\overline{H}),\;\;v_{ij}\longmapsto \delta_{ij}v_{ij}. $$
At first, we prove that $\pi'$ is an algebraic map. For this, it is enough to show that $$v_{ij}v_{kl}=0 $$
for all $i\neq j$ with $i+k\equiv j+l$ (mod $n$). Assume that this is not true, then $v_{ij}v_{kl}=av_{i+k,j+l}$ for some $0\neq a\in \k$, which is invertible by $v_{ii}=g^{i}$ for all $0\leq i\leq n-1$. But this is impossible since $v_{ij}$ is nilpotent. So, $\pi'$ is an algebraic map. In addition, the formula for the coproduct implies that $\pi'$ is also a coalgebra map. Therefore, $\pi'$ is a Hopf projection.  Using the classical Radford's biproduct (see Subsection \ref{ss2.4}), we have the following decomposition
$$\overline{H}=\overline{H}_{0}^{l}\#\k G(\overline{H}).$$
By \cite[Theorem 2]{Ang}, $\overline{H}_{0}^{l}$ is generated by skew primitive elements, say $x_1,\dots, x_{\theta}$ (we ask that $\theta$ is as small as possible). Moreover, by the proof of \cite[Theorem 2]{Ang} we know that $gx_ig^{-1}\in \k x_i$ for $(1\leq i\leq \theta)$. So, equation \eqref{eq5.1} implies that up to a nonzero scalar $x_i$ equals to a $v_{0j}$ for some $j$. In one word, we prove that the subalgebra $\overline{H}_{0}^{l}$ is generated by $v_{0n_1},\ldots,v_{0n_\theta}$ which are skew primitive elements.

\emph{Claim: $n_1,\ldots,n_\theta$ is a fraction of $n$.}

\emph{Proof of the claim:} Let $e_i$ be the exponent of $n_i$ for $1\leq i\leq \theta.$ We find that $e_i$ is the smallest number such $v_{0n_i}^{e_i}=0$. Indeed, on one hand it is not hard to see that $v_{0n_i}^{e_i}=0$ since by definition $v_{0n_i}^{e_i}\in \overline{H}_{00}=\k$ and $v_{0n_i}$ is nilpotent. On the other hand, assume that there is $l<e_i$ which is smallest such that $v_{0n_i}^{l}=0$. Then
 $$0=\D(v_{0n_i})^{l}=(1\otimes v_{0n_i}+v_{0n_i}\otimes g^{n_i})^{l}=\sum_{k=0}^{l}\left (
\begin{array}{c} l\\k
\end{array}\right)_{\xi^{n_i^2}}v_{0n_i}^{k}\otimes g^{n_i(l-k)}v_{0n_i}^{l-k}$$
 which implies that $\left (
\begin{array}{c} l\\k
\end{array}\right)_{\xi^{n_i^2}}=0$ for all $1\leq k\leq l-1$ and thus $\xi^{n_i^2}$ must be a primitive $l$th root of unity. Now we consider the element $v_{0,ln_i}$ which is not $1$ by the definition of $l$ (explicitly, $n\nmid ln_i$ since $l< e_i$). Thus the elements $ g':=g^{ln_i}, x:=v_{0,ln_i}$ generate a Hopf subalgebra satisfying
   $$g'x=xg',\;\;\;\;\D(x)=1\otimes x+x\otimes g'.$$
 (We need prove these two relations. The relation $g'x=xg'$ is clear. The proof of $\D(x)=1\otimes x+x\otimes g'$ is given as follows: Lifting these $v_{0j}$ to $H$, we get the corresponding elements $u_{0j}$ for $0\leq j\leq n-1$. Due to \cite[Propostion 6.2]{BZ}, they are normal and thus $u_{0n_i}^{l}=f(x)u_{0,ln_i}$ for some $0\neq f(x)\in H_{00}$. By the claim in the proof of the next proposition, that is, Proposition \ref{p5.7}, $u_{0n_i}$ is a skew primitive element. Using the fact that $\xi^{n_i^2}$ is a primitive $l$th root of unity, $u_{0n_i}^{l}$ is still a skew primitive element. This implies that $\D(f(x)u_{0,ln_i})$ and thus $\D(u_{0,ln_i})\in H_{00}\otimes H_{0,ln_i}+H_{0,ln_i}\otimes H_{ln_i,ln_i}$. Therefore, $v_{0,ln_i}$ has to be skew-primitive.)

 It is well known that a Hopf algebra satisfying above relations must be infinite dimensional (in fact, a infinite dimensional Taft algebra) which is a contradiction.   Thus, $e_i$ is the smallest number such $v_{0n_i}^{e_i}=0$.

   Now, we want to show that $(e_i,n_i)=1$. Otherwise, let $d_i=(e_i,n_i)>1$. Therefore,  we consider
$$\D(v_{0n_i})^{\frac{e_i}{d_i}}=(1\otimes v_{0n_i}+v_{0n_i}\otimes g^{n_i} )^{\frac{e_i}{d_i}}.$$
By definition, $e_i/d_i$ is coprime to $n_i$ thus coprime to $n_i^2$. This implies that  $\xi^{n_i^2}$ is a primitive $e_i/d_i$th root of unity. Therefore,
$$\D(v_{0n_i})^{\frac{e_i}{d_i}}=1\otimes v_{0n_i}^{\frac{e_i}{d_i}}+ g^{n_ie_i/d_i}\otimes v_{0n_i}^{\frac{e_i}{d_i}}.$$
Since $e_i$ is the smallest number such $v_{0n_i}^{e_i}=0$, $v_{0n_i}^{\frac{e_i}{d_i}}\neq 0$. This means that we go into the following situation again: Let $g'=g^{n_ie_i/d_i},\;x=v_{0n_i}^{e_i/d_i}$, then the Hopf subalgebra generated by $g',x$ is infinite dimensional. This is impossible.

Next, we want to show that $n|n_in_j$ for all $1\leq i\neq j\leq \theta.$ Through computation,
$$\D(v_{0n_i}v_{0n_j})=1\otimes v_{0n_i}v_{0n_j}+v_{0n_i}\otimes g^{n_i}v_{0n_j}+v_{0n_j}\otimes v_{0n_i}g^{n_j}+v_{0n_i}v_{0n_j}\otimes g^{n_i+n_j}$$ and
$$\D(v_{0n_j}v_{0n_i})=1\otimes v_{0n_j}v_{0n_i}+v_{0n_j}\otimes g^{n_j}v_{0n_i}+v_{0n_i}\otimes v_{0n_j}g^{n_i}+v_{0n_j}v_{0n_i}\otimes g^{n_i+n_j}.$$
By equation \eqref{eq5.1}, one has $v_{0n_i}v_{0n_j}=v_{0n_j}v_{0n_i}$. This implies that $g^{n_j}v_{0n_i}=v_{0n_i}g^{n_j}=\xi^{n_in_j}g^{n_j}v_{0n_i}$. Therefore, $\xi^{n_in_j}=1$ and thus $n|n_in_j$.

At last, we need to prove the conditions (3) and (4) of a fraction (see Definition \ref{d3.1}). Clearly, conditions (3) and (4) is equivalent to say that every $v_{0t}$ can be expressed as a product of $v_{0n_1},\ldots,v_{0,n_{\theta}}$ \emph{uniquely} (up to the order of these $v_{0,n_i}$'s due to the community of them) for all $0\leq t\leq n-1$. Since we already know that $v_{0n_1},\ldots,v_{0n_\theta}$ generate the whole algebra $\overline{H}_{0}^{l}$, it is enough to prove the following two conclusion: 1) $v_{0n_1}^{l_1}\cdots v_{0n_\theta}^{l_\theta}\neq 0$ for all $0\leq l_1\leq e_1-1,\ldots,0\leq l_\theta\leq e_\theta-1;$ 2) the elements in the set $\{v_{0n_1}^{l_1}\cdots v_{0n_\theta}^{l_\theta}|0\leq l_1\leq e_1-1,\ldots,0\leq l_\theta\leq e_\theta-1\}$ are linear independent. Of course, 1) is just a necessary part of 2). However,  we find that they help each other. To show them, we introduce the lexicographical order on $A=\{(l_1,\ldots,l_{\theta})|0\leq l_1\leq e_1-1,\ldots,0\leq l_\theta\leq e_\theta-1\}$ through
$$(l_1,\ldots,l_{\theta})<(l'_1,\ldots,l'_{\theta})\Leftrightarrow \textrm{exsits}\; 1\leq i\leq \theta\; \textrm{s.t.}\; l_j=l_j' \;\textrm{for} j<i\; \textrm{and} \; l_i<l_i'.$$
Now let $S=\{(s_1,\ldots,s_\theta)\in A|v_{0n_1}^{s_1}\cdots v_{0n_\theta}^{s_\theta}\neq 0\}$. Clearly, $S$ is nonempty due to $v_{0n_i}\neq 0$ for all $1\leq i\leq \theta.$ We prove that all elements $\{v_{0n_1}^{s_1}\cdots v_{0n_\theta}^{s_\theta}|(s_1,\ldots,s_\theta)\in S\}$ are linear independent firstly and then show that $S=A$. From this, 1) and 2) are proved clearly. In fact, assume we have a linear dependent relation among the elements in $\{v_{0n_1}^{s_1}\cdots v_{0n_\theta}^{s_\theta}|(s_1,\ldots,s_\theta)\in S\}$. Then there exists a linear combination
 $$a_{l_1,\ldots,l_\theta}v_{0n_1}^{l_1}v_{0n_2}^{l_2}\cdots v_{0n_\theta}^{l_\theta}+\cdots=0$$ with $a_{l_1,\ldots,l_\theta}\neq 0$ and $(l_1,\ldots,l_\theta)$ is as small as possible.  Takeing the coproduct to the above equality and  one can get a smaller item involving in a linear dependent equation. That is a contradiction. Next, let's show that $S=A$. Otherwise, there exists $v_{0n_1}^{l_1}\cdots v_{0n_\theta}^{l_\theta}=0$ for some $(l_1,\ldots,l_{\theta})\in A$. Then take $(l_1,\ldots,l_{\theta})$ as small as possible under above lexicographical order. Without loss generality, we can assume that $l_1>0$. Then take a $k_1$ such $0\leq k_1< l_1$. In the expression of $\D(v_{0n_1}^{l_1}\cdots v_{0n_\theta}^{l_\theta})$ on can find the coefficient of the item $v_{0n_1}^{l_1-k1}\otimes g^{k_1n_1}v_{0n_1}^{k_1}v_{0n_2}^{l_2}\cdots v_{0n_\theta}^{l_\theta}$ is
$$\left (
\begin{array}{c} l_1\\k_1
\end{array}\right)_{\xi^{n_1^2}}$$
which is not zero since we already know that $\xi^{n_1^2}$ is a primitive $e_1$th root of unity. This implies that either $v_{0n_1}^{l_1-k_1}=0$ or $v_{0n_1}^{k_1}v_{0n_2}^{l_2}\cdots v_{0n_\theta}^{l_\theta}=0$ by the linear independent relation we proved. But both of them are not possible. Therefore, $S=A$. So 1) and 2) are proved. The proof of the claim is done.

Let's go back to prove this lemma. Until now, we have proved that the Hopf algebra $\overline{H}$ is generated by $v_{0n_1},\ldots,v_{0n_\theta}$ and $g$ such that $n_1,\ldots,n_\theta$ is a fraction of $n$ and $$g^{n}=1, \;\; v_{0n_i}g=\xi^{n_i}gv_{0n_i},\;\;v_{0n_i}v_{0n_j}=v_{0n_j}v_{0n_i},\;\;v_{0n_i}^{e_i}=0$$
and $g$ is group-like, $v_{0n_i}$ is a $(1,g^{n_i})$-skew primitive element for all $1\leq i,j\leq \theta$. Therefore, we have a Hopf surjection
$$T(n_1,\ldots,n_\theta,\xi)\longrightarrow \overline{H},\;\;y_{n_i}\mapsto v_{0n_i},\;g\mapsto g,\;\;1\leq i\leq \theta.$$
Comparing the dimension of them, we know that this surjection is a bijection.
\end{proof}

With help of this lemma, we are in the position to give the main result of this subsection now.
\begin{proposition}\label{p5.7} Let $H$ be a prime Hopf algebra of GK-dimension one satisfying (Hyp1), (Hyp2) and $ n:=\ord(\pi)=\mi(\pi)>1$. Retain all above notations, then
\begin{itemize} \item[(1)] If $H$ is additive, then it is isomorphic to a fraction version of a infinite dimensional Taft algebra $T(\underline{n},1,\xi)$ of Subsection \ref{ss4.1}.
\item[(2)] If $H$ is multiplicative, then it is isomorphic to a fraction version of a generalized Liu algebra $B(\underline{n},\omega,\gamma)$ of Subsection \ref{ss4.4}.
\end{itemize}
\end{proposition}
\begin{proof} Before we prove (1) and (2), we want to recall some basic facts, which are still valid in our case, on the coproduct from \cite[Proposition 6.7]{BZ}. The first fact is that $g:=u_{11}$ is a group-like element and $u_{ii}$ can defined as $u_{ii}:=u_{11}^{i}$ (see (a) of \cite[Proposition 6.7]{BZ}).  By (1) of Lemma \ref{l2.8}, in general one has
$$\D(u_{ij})=\sum_{s,t}C_{st}^{ij}(u_{is}\otimes u_{tj})$$
for $C_{st}^{ij}\in H_{00}\otimes H_{00}$ and $0\leq i,j,s,t\leq n-1.$
The second fact is $C_{st}^{ij}=0$ when $s\neq t$ (see (6.7.5) in the proof of \cite[Proposition 6.7]{BZ}). Therefore, the coproduct for $u_{ij}$ can be written as
\begin{equation} \D(u_{ij})= C_{ii}^{ij}g^{i}\otimes u_{ij}+C_{jj}^{ij} u_{ij}\otimes g^{j}+\sum_{s\neq i,j} C_{ss}^{ij} u_{is}\otimes u_{sj}
\end{equation}
for all $0\leq i,j\leq n-1.$
Now by Lemma \ref{l5.6} we can assume that $\overline{H}=T(n_1,\ldots,n_{\theta},\xi)$. Then we get the following observation.

\emph{Claim. For all $1\leq i\leq \theta$, the element $u_{0n_i}$ is a $(1,g^{n_i})$-skew primitive element.}

\emph{Proof of the claim.} By direct computation,
\begin{eqnarray*}&&(\id\otimes \D)\D(u_{0n_i})\\
&=&(\id\otimes \D)(C_{00}^{0n_i} 1\otimes u_{0n_i}+C_{n_in_i}^{0n_i}u_{0n_i}\otimes g^{n_i}+\sum_{s\neq 0,n_i}C_{ss}^{0n_i}u_{0s}\otimes u_{sn_i})\\
&=& (\id\otimes \D)(C_{00}^{0n_i})1\otimes (C_{00}^{0n_i} 1\otimes u_{0n_i}+C_{n_in_i}^{0n_i}u_{0n_i}\otimes g^{n_i}+\sum_{s\neq 0,n_i}C_{ss}^{0n_i}u_{0s}\otimes u_{sn_i})\\
&&+ (\id\otimes \D)(C_{n_in_i}^{0n_i})u_{0n_i}\otimes g^{n_i}\times g^{n_i}+\sum_{s\neq 0,n_i}(\id\otimes \D)(C_{ss}^{0n_i})u_{0s}\otimes\\
&& \quad [C_{ss}^{sn_i}g^s\otimes u_{sn_i}+C_{n_in_i}^{sn_i}u_{sn_i}\otimes g^{n_i}+\sum_{t\neq s,n_i}C_{tt}^{sn_i}u_{st}\otimes u_{tn_i}]
\end{eqnarray*}
and
\begin{eqnarray*}
&&(\D\otimes \id)\D(u_{0n_i})\\
&=&(\D\otimes \id)(C_{00}^{0n_i} 1\otimes u_{0n_i}+C_{n_in_i}^{0n_i}u_{0n_i}\otimes g^{n_i}+\sum_{s\neq 0,n_i}C_{ss}^{0n_i}u_{0s}\otimes u_{sn_i})\\
&=& (\D\otimes \id)(C_{00}^{0n_i})1\otimes 1\otimes u_{0n_i}+(\D\otimes \id)(C_{n_in_i}^{0n_i})[
C_{00}^{0n_i}1\otimes u_{0n_i}\\
&&+C_{n_in_i}^{0n_i}u_{0n_i}\otimes g^{n_i}+\sum_{s\neq 0,n_i}u_{0s}\otimes u_{sn_i}]\otimes g^{n_i}\\
&&+ \sum_{s\neq 0,n_i}(\D\otimes \id)(C_{ss}^{0n_i})[C_{00}^{0s}1\otimes u_{0s}+C_{ss}^{0s}u_{0s}\otimes g^s+\sum_{t\neq 0,s}C_{tt}^{0s}u_{0t}\otimes u_{ts}]\otimes u_{sn_i}.
\end{eqnarray*}
By associativity, we get the following identities:
\begin{eqnarray}
\notag&&(\id\otimes \D)(C_{00}^{0n_i})(1\otimes C_{00}^{0n_i})=(\D\otimes \id)(C_{00}^{0n_i})\\
\notag&&(\id\otimes \D)(C_{00}^{0n_i})(1\otimes C_{n_in_i}^{0n_i})=(\D\otimes \id)(C_{n_in_i}^{0n_i})\\
\notag&&(\id\otimes \D)(C_{n_in_i}^{0n_i})=(\D\otimes \id)(C_{n_in_i}^{0n_i})(C_{n_in_i}^{0n_i}\otimes 1)\\
\label{eq5.4}&&(\id\otimes \D)(C_{00}^{0n_i})(1\otimes C_{ss}^{0n_i})=(\D\otimes \id)(C_{ss}^{0n_i})(C_{00}^{0s}\otimes 1)\\
\label{eq5.5}&&(\id\otimes \D)(C_{ss}^{0n_i})(1\otimes C_{ss}^{sn_i})=(\D\otimes \id)(C_{ss}^{0n_i})(C_{ss}^{0s}\otimes 1)
\end{eqnarray}
for $s\neq 0,n_i.$ From the first three identities, we find that $C_{00}^{0n_i}=C_{n_in_i}^{0n_i}=1$ by using the same method given in \cite[Page 297]{BZ}. This indeed implies that  $$C_{00}^{0t}=C_{tt}^{0t}=1$$
for all $0\leq t\leq n-1$ since we have the same first three identities just through replacing $n_i$ by $t$.

Recall again the dichotomy of $H_{00}$: either $H_{00}=\k[x]$ or $H_{00}=\k[x^{\pm 1}]$. From this we know that $C_{ss}^{0n_i}=\sum_{k,l} a^{s,0,n_i}_{kl}x^k\otimes x^{l}$ for $s\neq 0,n_i$ and $a^{s,0,n_i}_{kl}\in \k.$ We just prove our claim in the case $H_{00}=\k[x]$ since the other case can be proved similarly. By the image of $u_{0n_i}$ in $\overline{H}$ is a skew primitive element, $$a^{s,0,n_i}_{00}=0.$$
Since $C_{00}^{0t}=C_{tt}^{0t}=1$ for all $0\leq t\leq n-1$, the equation \eqref{eq5.4} is simplified into
$$(1\otimes C_{ss}^{0n_i})=(\D\otimes \id)(C_{ss}^{0n_i})$$
which implies that  $a^{s,0,n_i}_{kl}=0$ if $k\neq 0$. Similarly, the equation \eqref{eq5.5} implies that $a^{s,0,n_i}_{0l}=0$ if $l\neq 0$. Thus, $C_{ss}^{0n_i}=0$ for $s\neq 0,n_i$ and $u_{0n_i}$ is a $(1,g^{n_i})$-skew primitive element for $1\leq i\leq \theta.$ Moreover, we point out that through the same way given in \cite[Theorem 6.7]{BZ} one can show that as an algebra the Hopf algebra $H$ is generated by $H_{00}, g=u_{11} $ and $u_{0n_i}$ for $1\leq i\leq \theta.$

(1) Now $H$ is additive with $H_{00}=\k[x]$. We already know that $g=u_{11}$ is group-like and thus $g^{n}$ is a group-like in $H_{00}$ by the bigrading property. But the only group-like in $H_{00}$ is $1$ and thus $$g^{n}=1.$$  Consider the element $u_{0n_i}$ for $1\leq i\leq \theta.$ Through the quantum binomial theorem, $u_{0n_i}^{e_i}$ is a primitive element now. This means  there exists $c_i\in \k$ such that $u_{0n_i}^{e_i}=c_ix$. Since $H$ is prime, $c_i\neq 0$. Therefore, through multiplying $u_{0n_i}$ by a suitable scalar one can assume that
$$u_{0n_i}^{e_i}=x$$
for all $1\leq i\leq \theta.$ By equation \eqref{eq5.1}, $u_{0n_i}u_{0n_j}=u_{0n_j}u_{0n_i}$ for all $1\leq i,j\leq \theta.$  Therefore, we have a Hopf surjection
$$\phi:\;T(\underline{n},1,\xi)\longrightarrow H,\;\;x\mapsto x,\;\;y_{n_i}\mapsto u_{0n_i},\;\;g\mapsto g,$$
where $\underline{n}=\{n_1,\dots,n_\theta\}.$ Since both of them are prime of GK-dimension one, $\phi$ is an isomorphism.

(2) Now $H$ is multiplicative with $H_{00}=\k[x^{\pm 1}]$. We already know that $g=u_{11}$ is group-like and thus $g^{n}$ is a group-like element in $H_{00}$ by the bigrading property. Since $\{x^i|i\in \Z\}$ are all the group-likes in $H_{00}$,  $$g^{n}=x^{\omega}$$
  for some $\omega\geq 0$ (noting that we can replace $x$ by $x^{-1}$ if $\omega$ is negative). We claim that $\omega\neq 0$. If not, then as the proof of (1) we know that $u_{0n_i}^{e_i}$ is primitive in $H_{00}$. Hence $u_{0n_i}^{e_i}=0$ which is impossible since $H$ is prime.

   Consider the element $u_{0n_i}$ for $1\leq i\leq \theta.$ Through the quantum binomial theorem, $u_{0n_i}^{e_i}$ is a $(1,g^{e_in_i})=(1,x^{\omega\frac{e_in_i}{n}})$-skew primitive element in $H_{00}$. Therefore, after dividing if necessary by  non-zero scalar,
    $$u_{0n_i}^{e_i}=1-x^{\omega\frac{e_in_i}{n}}$$
    for all $1\leq i\leq \theta.$ Also by equation \eqref{eq5.1}, $u_{0n_i}u_{0n_j}=u_{0n_j}u_{0n_i}$ for all $1\leq i,j\leq \theta.$  Therefore, we have a Hopf surjection
$$\phi:\;B(\underline{n},\omega,\xi)\longrightarrow H,\;\;x\mapsto x,\;\;y_{n_i}\mapsto u_{0n_i},\;\;g\mapsto g,$$
where $\underline{n}=\{n_1,\dots,n_\theta\}.$ Since both of them are prime of GK-dimension one, $\phi$ is an isomorphism.

\end{proof}

\section{Remaining case}
In the previous section, we already dealt with the ideal cases: the case $\mi (\pi)=1$ and the case $\ord(\pi)=\mi(\pi)>1.$ In this section, we want to deal with the remaining case: $\ord{(\pi)}>\mi(\pi)>1.$ The main aim of this section is to classify  prime Hopf algebras of GK-dimension one $H$ in this remaining case. To realize this aim, we apply the similar idea used in \cite{WLD}, that is, we first construct a special Hopf subalgebra $\widetilde{H}$, which can be classified by previous results, and then we show that $\widetilde{H}$ determines the structure of $H$ entirely.

In this section, $H$  is a prime Hopf algebra of GK-dimension one satisfying (Hyp1), (Hyp2) and $n:=\ord{(\pi)}>m:=\mi(\pi)>1$ unless stated otherwise. And as before, the $1$-dimensional representation in (Hyp1) is denoted by $\pi.$ Recall that $$H=\bigoplus_{i,j\in \Z_{n}}H_{ij}$$
is $\Z_n$-bigraded by \eqref{eq2.3}.

\subsection{The Hopf subalgebra $\widetilde{H}$.} By definition, we know that $m|n$ and thus let $t:=\frac{n}{m}.$ We define the following subalgebra
$$\widetilde{H}:=\bigoplus_{0\leq i,j\leq m-1} H_{it,jt}.$$
The following result is a collection of \cite[Proposition 5.4, Lemma 5.5]{WLD}, which were proved in \cite{WLD} without  using the condition of regularity.
\begin{lemma}\label{l6.1} Retain above notations.
\begin{itemize}\item[(1)] For every $i, j$ with $1\leqslant i, j\leqslant n-1$, $H_{ij}\neq 0$ if and only if $i-j\equiv 0$ \emph{(mod} $t$\emph{)} for all $0\leqslant i, j\leqslant n-1 $ .
\item[(2)] The algebra $\widetilde{H}$ is a Hopf subalgebra of $H$.
\end{itemize}
\end{lemma}

The key observation of \cite{WLD} and here is that  Hopf subalgebra $\widetilde{H}$ lives in an ideal case.

\begin{proposition}\label{p6.2} For the  Hopf algebra $\widetilde{H}$, we have the following results.
\begin{itemize}\item[(1)] It is prime of GK-dimension one.
\item[(2)] It satisfies (Hyp1) and (Hyp2) through the restriction $\pi|_{\widetilde{H}}$ of $\pi$ to $\widetilde{H}$.
\item[(3)] $\ord(\pi|_{\widetilde{H}})=\mi(\pi|_{\widetilde{H}})=m.$
\end{itemize}
\end{proposition}
\begin{proof} (1) For each $0\leq i\leq m-1$, let $\widetilde{H}^{l}_{it}:=\bigoplus_{0\leq j\leq m-1}H_{it,jt}.$ By Lemma \ref{l6.1}, we know that $\widetilde{H}_{it}^{l}=H_{it}^{l}$. Therefore $\widetilde{H}=\bigoplus_{0\leq i\leq m-1}\widetilde{H}^{l}_{it}$ is strongly graded and $\widetilde{H}^{l}_{0}$ is a commutative domain. Thus the Lemma \ref{l2.10} is applied. As consequences, $\widetilde{H}$ is prime with PI-degree $m$. Since $\widetilde{H}^{l}_{0}=H^{l}_{0}$ is of GK-dimension one and $\widetilde{H}$ is $\Z_{m}$-strongly graded, $\widetilde{H}$ is of GK-dimension one.

(2) Denote the restriction
of the actions of $\Xi_\pi^l$ and $\Xi_\pi^r$ to $\widetilde{H}$ by
$\Gamma^l$ and $\Gamma^r$, respectively. Since
$\widetilde{H}=\bigoplus_{0\leqslant i\leqslant m-1 } H_{it}^l$, we
can see that for each $0\leqslant i\leqslant m-1$ and any $0\neq x
\in H_{it}^l$,
$$(\Gamma^l)^m(x)=\xi^{itm}x=x$$
for $\xi$ a primitive $n$th root of unity.
This implies that the group $\langle\Gamma^l\rangle$ has order $m$ and thus $\pi|_{\widetilde{H}}$ is of order $m$. We already know that PI-deg$(\widetilde{H})=m$ and the invariant component $\widetilde{H}^{l}_{0}=H^{l}_{0}$ is a domain. So $\widetilde{H}$ satisfies (Hyp1) and (Hyp2).

(3) Similarly, $|\langle\Gamma^r\rangle|=m$. We claim that
$$\langle\Gamma^l\rangle\cap \langle\Gamma^r\rangle={1}.$$
In fact, if $(\Gamma^l)^i=(\Gamma^r)^j$ for some $0\leqslant i,
j\leqslant m-1$. Choose $0\neq x\in H_{tt}$, we find
$$\xi^{ti}x=(\Gamma^l)^i(x)=(\Gamma^r)^j(x)=\xi^{tj}x$$ which implies
$i=j$. Let $0\neq y\in H_{0,t}$, then
$$y=(\Gamma^l)^i(y)=(\Gamma^r)^j(y)=\xi^{tj}y$$ forces
$j=0$. Thus we get $i=j=0$, i.e., $\langle\Gamma^l\rangle\cap
\langle\Gamma^r\rangle={1}$. This implies that $\mi(\pi|_{\widetilde{H}})=m$.
\end{proof}

\begin{corollary}\label{c6.3} As a Hopf algebra $\widetilde{H}$ is isomorphic to either a faction version of infinite dimensional Taft algebra $ T(\underline{m},1,\xi)$ or a fraction version of generalized Liu algebra  $B(\underline{m},\omega,\gamma)$.
\end{corollary}
\begin{proof} This is a direct consequence of Propositions \ref{p5.7} and \ref{p6.2}.
\end{proof}

This corollary implies that either $H_{00}=\k[x]$ (i.e. $H\cong T(\underline{m},1,\xi)$) or $H_{00}=\k[x^{\pm1}]$ (i.e. $H\cong B(\underline{m},\omega,\gamma)$) again. That is, we go back to a familiar situation that we have a dichotomy on $H$ now.
\begin{definition} \emph{We call $H$ is} additive \emph{(resp.} multiplicative\emph{) if $H_{00}=\k[x]$} \emph{(resp.} $H_{00}=\k[x^{\pm 1}])$.
\end{definition}

We realize that the \cite[Proposition 6.6]{WLD} is also true in our case and we recall it as follows.

\begin{lemma}\label{l6.5} Every homogeneous component $H_{i,i+jt}$ of $H$ is a free $H_{00}$-module of rank one on both sides for all $0\leq i\leq n-1$ and $0\leq j\leq m-1.$
\end{lemma}

From this lemma, there is a generating set $\{u_{i,i+jt}|0\leq i\leq n-1,\;0\leq j\leq m-1\}$ satisfying
$$u_{00}=1\;\;\;\;\textrm{and}\;\;\;\;H_{i, i+jt}=u_{i, i+jt}H_{00}=H_{00}u_{i, i+jt}.$$
So, $H$ can be written as
\begin{equation}\label{eq6.1} H=\underset{0\leqslant j\leqslant m-1}{\bigoplus_{0\leqslant i\leqslant n-1}}H_{00}u_{i,
i+jt}=\underset{0\leqslant j\leqslant m-1}{\bigoplus_{0\leqslant
i\leqslant n-1}}u_{i, i+jt}H_{00}.\end{equation}

\subsection{Additive case.}

If $H$ is additive,  $\widetilde{H}=T(\underline{m}, 1, \xi)$. Recall that $n$ is the $\pi$-order and $n=mt$. We will prove $H$ is isomorphic as a Hopf algebra to $T(\underline{m}, t, \zeta)$, for $\zeta$ some primitive $n$th root of 1. Recall that
\begin{eqnarray*}\widetilde{H}=T(\underline{m}, 1, \xi)&=&\k\langle g, y_{m_1},\ldots,y_{m_\theta}|g^m=1,y_{m_i}g=\xi^{m_i} gy_{m_i},
 y_{m_i}y_{m_j}=y_{m_j}y_{m_i},\\ && \quad y_{m_i}^{e_i}=y_{m_j}^{e_j},\,1\leq i,j\leq \theta\rangle,\end{eqnarray*}
here by Proposition \ref{p4.3} we assume that $(m_1,\ldots,m_{\theta})=1$ without loss of generality.  Note that
$H=\bigoplus_{0\leqslant i\leqslant n-1, 0\leqslant j\leqslant m-1}
H_{i, i+jt}$, $\widetilde{H}=\bigoplus_{0\leqslant i, j\leqslant
m-1} H_{it, jt}$ and $H_{it, jt}=\k[y_{m_1}^{e_1}]y_{j-i}g^{i}$ (the index $j-i$
is interpreted mod $m$). In particular, $H_{00}=\k[y_{m_1}^{e_1}]$, $H_{0,
jt}=\k[y_{m_1}^{e_1}]y_{j}$ and $H_{tt}=\k[y_{m_1}^{e_1}]g$.

By Lemma \ref{l2.8} (5), $\epsilon(u_{11})\neq 0$. Multiplied with a suitable scalar, we can assume that $\epsilon(u_{11})=1$ throughout  this subsection. The following results are parallel to \cite[Lemma 7.1, Propositions 7.2, 7.3]{WLD}. Since the situation is changed, we write the details out.

\begin{lemma}\label{l6.6} Let $u:=u_{11}$. Then $H_1^l=H_0^lu$, $H=\bigoplus_{0\leqslant k\leqslant t-1}\widetilde{H}u^k$ and $u$ is invertible.
\end{lemma}
\begin{proof} By the bigraded structure of $H$, we have
$$H_{0,m_it}H_{11}\subseteq H_{1, 1+m_it},\;\;\;\; H_{0, (e_i-1)m_it}H_{1, 1+m_it}\subseteq
H_{11},$$ which imply $$H_{0,m_it}H_{0, (e_i-1)m_it}H_{1, 1+m_it}\subseteq
H_{0,m_it}H_{11}\subseteq H_{1, 1+m_it},$$
for all $1\leq i\leq \theta.$

Since $H_{0,m_it}H_{0,(e_i-1)m_it}=y_{m_i}^{e_i}H_{00}$ is a maximal ideal of $H_{00}=\k[y_{m_1}^{e_1}]=\k[y_{m_i}^{e_i}]$ and $H_{1,1+m_it}$ is a free $H_{00}$-module of rank one (by Lemma \ref{l6.5}),  $H_{0,m_it}H_{0,(e_i-1)m_it}H_{1, 1+m_it}$ is a maximal $H_{00}$-submodule of $H_{1, 1+m_it}$.
Thus $$H_{0,m_it}H_{11}=H_{0,m_it}H_{0,
(e_i-1)m_it}H_{1, 1+m_it}=y_{m_i}^{e_i}H_{1, 1+m_it}\;\; \textrm{or}\;\; H_{0,m_it}H_{11}=H_{1, 1+m_it}.$$

If $H_{0,m_it}H_{11}=y_{m_i}^{e_i}H_{1, 1+m_it}$, then $y_{m_i}u_{11}=y_{m_i}^{e_i}\alpha(y_{m_i}^{e_i})u_{1,
1+m_it}$ for some polynomial $\alpha(y_{m_i}^{e_i})\in \k[y_{m_i}^{e_i}]$. So
$$y_{m_i}(u_{11}-y_{m_i}^{e_i-1}\alpha(y_{m_i}^{e_i})u_{1, 1+m_it})=0.$$
Therefore, $y_{m_i}^{e_i}(u_{11}-y_{m_i}^{e_i-1}\alpha(y_{m_i}^{e_i})u_{1, 1+m_it})=0$.
Note that each homogenous $H_{i, i+jt}$ is a torsion-free
$H_{00}$-module, so $$u_{11}=y_{m_i}^{e_i-1}\alpha(y_{m_i}^{e_i})u_{1, 1+m_it}.$$
By  assumption, $\epsilon(u_{11})=1$.
But, by definition,   $\epsilon(y_{m_i})= 0$. This is impossible. So $H_{0,m_it}H_{11}=H_{1, 1+m_it}$ which implies that $H_{0,m_it}u_{11}=H_{1, 1+m_it}$.

Since above $i$ is arbitrary, that is $1\leq i\leq \theta$, we can show that $H_{0,jt}u_{11}=H_{1, 1+jt}$
for $0\leqslant j\leqslant m-1$.
 Thus $H_1^l=H_0^lu_{11}$. Since $H=\bigoplus_{0\leqslant j\leqslant n-1} H_j^l$ is strongly graded, $u_{11}$ is invertible and $H_j^l=H_0^lu_{11}^j$ for all $0\leqslant
j\leqslant n-1$. Let $u:=u_{11}$, then we have
$$H=\bigoplus_{0\leqslant k\leqslant t-1}\widetilde{H}u^k.$$
\end{proof}

We are in a position to determine the structure of $H$ now.

 \begin{lemma}\label{l6.7} With above notations, we have
 $$u^t=g, \;\;\;\;y_{m_i}u=\zeta^{m_i} uy_{m_i}\;\;\;\;(1\leq i\leq \theta),$$
where $\zeta$ is a primitive $n$th root of $1$.
 \end{lemma}
\begin{proof}
 For all $1\leq i\leq \theta$, by $H_{0,m_it}u=uH_{0,m_it}$, there exists a polynomial $\beta_i(y_{m_i}^{e_i})\in \k[y_{m_i}^{e_i}]$ such that
$$y_{m_i}u=uy_{m_i}\beta_i(y_{m_i}^{e_i}).$$  Then
$$y_{m_i}u^t=u^ty_{m_i}{\beta}_i'(y_{m_i}^{e_i})$$ for some polynomial ${\beta}_i'(y_{m_i}^{e_i})\in
\k[y_{m_i}^{e_i}]$ induced by $\beta(y_{m_i}^{e_i})$. Since $u^t$ is invertible and
$u^t\in H_{t,t}=\k[y_{m_i}^{e_i}]g$,  $u^t=ag$ for $0\neq a\in k$. By
assumption, $\epsilon(u)=1$ and thus $a=1$. Therefore,
$$u^t=g.$$
Since $y_{m_i}g=\xi^{m_i} gy_{m_i}$,  we have ${\beta}_{i}'(y_{m_i}^{e_i})=\xi^{m_i}$. Then it is easy to see that $\beta_i(y_{m_i}^{e_i})=\zeta_i\in \k$ with $\zeta_i^t=\xi^{m_i}$. By assumption, $(m_1,\ldots,m_{\theta})=1$ and thus there exists $\zeta\in \k$ such that $\zeta^{t}=\xi$ and $\zeta^{m_i}=\zeta_i$ for all $1\leq i\leq \theta.$ Of course,
$\zeta^n=1$.

The last job is to show that $\zeta$ is a primitive $n$th root of
$1$. Indeed, assume $\zeta$ is a primitive $n'$th root of 1. By definition, $m|n'|n$ and $n'\neq n$. Therefore, it is not hard to see that $$u':=u^{n'}\in C(H)$$
the center of $H$. Since $g^{m}=u^{n}=(u')^{\frac{n}{n'}}=1$, we have orthogonal central idempotents $1_{l}:=\sum_{j=0}^{\frac{n}{n'}-1} \varsigma^{-lj}(u')^{j}$ for $0\leq l\leq \frac{n}{n'}-1$ and $\varsigma$ a primitive $\frac{n}{n'}$th root of unity. This contradicts to the fact that $H$ is prime.
\end{proof}

\begin{lemma}\label{l6.8} The element $u$ is a group-like element of $H$.
\end{lemma}
\begin{proof} First of all $H_0^r\cong \k[x]\cong H_0^l$. Then $H_0^r\otimes H_0^l\cong \k[x, y]$ and the only invertible elements in $H_0^r\otimes H_0^l$ are nonzero scalars in $\k$. Since $\Delta(u)$ and $u\otimes u$ are invertible, $\Delta(u)(u\otimes u)^{-1}$ is invertible (and hence a scalar). Thus $u$ must be group-like by noting that $\epsilon(u)=1$.
\end{proof}

The next proposition follows from above lemmas directly.

\begin{proposition}\label{p6.9}  Let $H$ be a prime regular Hopf algebra of GK-dimension one satisfying (Hyp1), (Hyp2) and $\ord(\pi)=n>\mi(\pi)=m>1$. If $H$ is additive, then $H$ is isomorphic as a Hopf algebra to a fraction version of infinite dimensional Taft algebra.
\end{proposition}

\subsection{Multiplicative case.}
If $H$ is multiplicative, then $\widetilde{H}=B(\underline{m}, \omega,
\gamma)$ for $\underline{m}=\{m_1,\ldots,m_\theta\}$ a fraction of $m$, $\gamma$ a primitive $m$th root of $1$ and $\omega$ a positive integer. As usual, the generators of $B(\underline{m}, \omega, \gamma)$ are denoted by $x^{\pm 1},y_{m_1},\ldots,y_{m_\theta}$ and $g$. By equation \eqref{eq4.5} and \cite[Remark 6.3]{WLD}, we can assume that
 $\widetilde{H}=\bigoplus_{0\leqslant i, j\leqslant m-1}H_{it,
jt}$ with $$H_{it, jt}=\k[x^{\pm 1}]y_{j-i}g^i$$ (the index $j-i$ is
interpreted mod $m$). In particular, $H_{00}=\k[x^{\pm 1}]$, $H_{0,
jt}=\k[x^{\pm 1}]y_{j}$ and $H_{t, t}=\k[x^{\pm 1}]g$.

Set $u_j:=u_{1,1+jt} (0\leqslant j\leqslant m-1)$ for convenience.
By the structure of the bigrading of $H$, we have
\begin{equation}\label{eq6.2} y_{m_i}u_j=\phi_{m_i,j}u_{m_i+j}\end{equation}
and
\begin{equation}\label{eq6.3}u_jy_{m_i}=\varphi_{m_i,j}u_{m_i+j}\end{equation}
for some polynomials $\phi_{m_i,j}, \varphi_{m_i,j}\in \k[x^{\pm 1}]$ and $1\leqslant i\leqslant
\theta,\;0\leq j\leq m-1.$  With these notions and the equality $y_{m_i}^{e_i}=1-x^{\omega \frac{e_im_i}{m}}$, we find that
\begin{equation}\label{eq6.4}(1-x^{\omega \frac{e_im_i}{m}})u_j=y_{m_i}^{e_i}u_j=\phi_{m_i,j}\phi_{m_i,m_i+j}\cdots
\phi_{m_i,(e_i-1)m_i+j}u_j\end{equation} and
\begin{equation}\label{eq6.5} u_j(1-x^{\omega \frac{e_im_i}{m}})=u_jy_{m_i}^{e_i}=\varphi_{m_i,j}\varphi_{m_i,m_i+j}\cdots
\varphi_{m_i,(e_i-1)m_i+j}u_j,\end{equation}
for $1\leqslant i\leqslant
\theta$ and $0\leq j\leq m-1.$

\begin{lemma}\label{l6.10} There is no such $H$ satisfying $\ord(\pi)=n>\mi(\pi)=m>1$ and $n/m>2$.
\end{lemma}
\begin{proof} Since $u_jH_{00}=H_{00}u_j$, we have $$u_jx=\alpha_j(x^{\pm
1})u_j \ \ \ \ \text{and}\ \ \ \ u_jx^{-1}=\beta_j(x^{\pm 1})u_j$$
for some $\alpha_j(x^{\pm 1}), \beta_j(x^{\pm 1})\in \k[x^{\pm 1}]$ for $0\leq j\leq m-1$.
From
$$u_j=u_jxx^{-1}=\alpha_j(x^{\pm
1})u_jx^{-1}=\alpha_j(x^{\pm 1})\beta_j(x^{\pm 1})u_j,$$ we get
$\alpha_j(x^{\pm 1})\beta_j(x^{\pm 1})=1$ and thus $\alpha_j(x^{\pm
1})=\lambda_jx^{a_j}$ for some $0\neq \lambda_j\in \k, 0\neq a_j\in
\mathbb{Z}$. Note that $u_j^t\in H_{t,(1+jt)t}=\k[x^{\pm 1}]y_{\bar{jt}}g$, where
$\bar{jt}\equiv jt\; (\text{mod}\; m)$. So we have
$u_j^t=\gamma_j(x^{\pm 1})y_{\bar{jt}}g$ for some $\gamma_j(x^{\pm
1})\in \k[x^{\pm 1}]$. Hence $u_j^t$ commutes with $x$. Applying
$u_jx=\lambda_jx^{a_j}u_j$ to $u_j^tx=xu_j^t$, we get
$\lambda_j^{\sum_{s=0}^{t-1}a_j^s}=1$ and $x^{a_j^t}=x$. If $t$ is odd, $a_j=1$ and if $t$ is even, then $a_j$ is either $1$ or $-1$.

Now we consider the special case $j=0$.  By
$\epsilon(xu_0)=\epsilon(u_0x)\neq 0$, we find that $\lambda_0=1$.

If $a_0=1$, that is $u_0x=xu_0$, then we will see $u_jx=xu_j$ for
all $0\leqslant j\leqslant m-1$. In fact, for this, it is enough to show that $u_{m_{i}}x=xu_{m_{i}}$ for all $1\leq i\leq \theta.$ Since
$$\phi_{m_i,0}xu_{m_i}=x\phi_{m_i,0}u_{m_i}=xy_{m_i}u_0=yxu_0=y_{m_i}u_0x=\phi_{m_i,0}u_{m_i}x,$$ we have $u_{m_i}x=xu_{m_i}$ since
$H_{1,1+m_i t}$ is a torsion-free $H_{00}$-module. Then by the strongly graded
structure $u_{i,i+jt}\in H_i^l=(H_1^l)^i$ and $x$ is commutative with $H_1^l$, it is not hard to see that $u_{i,i+jt}x=xu_{i,i+jt}$ for
all $0\leqslant i\leqslant n-1, 0\leqslant j\leqslant m-1$.
Therefore the center $C(H)\supseteq H_{00}=\k[x^{\pm 1}]$. By \cite[Lemma 5.2]{BZ}, $C(H)\subseteq H_0$ and thus $C(H)=H_0=\k[x^{\pm 1}]$. This implies that
$$\text{rank}_{C(H)}H=\text{rank}_{H_{00}}H= nm < n^2,$$ which contradicts  the fact:
the PI-degree of $H$ is $n$ and equals the square root of the rank
of $H$ over $C(H)$.

If $a_0=-1$, that is $u_0x=x^{-1}u_0$, we can deduce that
$u_{i,i+jt}x=x^{-1}u_{i,i+jt}$ for all $0\leqslant i\leqslant n-1,\ 0\leqslant j\leqslant m-1$ by using  the parallel proof of the case $a_0=1$. For $s\in \mathbbm{N}$, let $z_s:=x^s+x^{-s}$.
Define $\k[z_s|s\geq 0]$ to be the subalgebra of $\k[x^{\pm 1}]$ generated by all $z_{s}$. Note that $\k[x^{\pm 1}]$ has rank 2 over $\k[z_s| s\geqslant 1]$.
Thus $C(H)\supseteq \k[z_s|s\geq 0]$. Using  \cite[Lemma 5.2]{BZ} again, we have $C(H)= \k[z_s|s\geq 0]$. Hence
$$\text{rank}_{C(H)}H = 2nm\neq n^2$$
since $n/m >2$ by assumption. This contradicts the fact that the PI-deg$H=n$ again.

Combining these two cases, we get the desired result.
\end{proof}

We turn now to consider the case: $\ord(\pi)=2\mi(\pi)=2m$. In this case, $t=2$. As
discussed in the proof of Lemma \ref{l6.10}, if such $H$ exists then the following relations
\begin{equation}\label{rl} u_jx=x^{-1}u_j\ \ (0\leqslant i\leqslant m-1)\end{equation}
 hold in $H$. Using these relations and \eqref{eq6.5}, we have
 \begin{equation}\label{r2} \varphi_{m_i,j}\varphi_{m_i,m_i+j}\cdots \varphi_{m_i,(e_i-1)m_i+j}=1-x^{-\omega\frac{e_im_i}{m}}.
 \end{equation}
 for all $1\leq i\leq \theta$ and $0\leq j\leq m-1$. To determine the structure of $H$, we need to give some harmless assumptions on the choice of $u_j$ ($0\leqslant j\leqslant m-1$) and $\phi_{m_i,j}$:
 \begin{itemize}\item[(1)] We assume $\epsilon(u_0)=1$;
 \item[(2)] For each $1\leq i\leq \theta$, let $\xi_{i,s}:=e^{\frac{2s\pi i}{\omega\frac{e_im_i}{m}}}$ and thus $1-x^{\omega\frac{e_im_i}{m}}=\Pi_{s\in S_i}(1-\xi_{i,s}x)$, where $S_i:=\{0,1,
\cdots, \omega\frac{e_im_i}{m}-1\}$.
Since \begin{equation*}\phi_{m_i,j}\cdots \phi_{m_i,(e_i-1)m_i+j}=y_{m_i}^{e_i}=1-x^{\omega\frac{e_im_i}{m}},\end{equation*} there is no harm to assume that
\begin{equation*} \phi_{m_i, tm_i+j}=\Pi_{s\in S_{i,j,t}}(1-\xi_{i,s}x),\end{equation*}
where $S_{i,j,t}$ is a subset of $S_{i}$.
 \item[(3)]
By the strongly graded structure of $H$, the equality $H_2^l=H^{l}_0g$ and
the fact that $g$ is invertible in $H$, we can take $u_{k, k+2j}$ such that
\begin{equation*}
u_{k, k+2j}=
\begin{cases}
 g^{\frac{k-1}{2}}u_j  & \text{if}\ \ k \ \ \text{is odd},\\
 y^jg^{\frac{k}{2}}    & \text{if}\ \ k \ \ \text{is even},
\end{cases}
\end{equation*}
 for all
$2\leqslant k \leqslant 2m-1$. \end{itemize}

In the rest of this section, we always make these assumptions.

 We still need two notations, which appeared in the proof of Proposition \ref{p4.7}. For
 a polynomial
$f=\sum a_ix^{b_i} \in \k[x^{\pm 1}]$, we denote by $\bar{f}$ the polynomial $\sum
a_ix^{-b_i}$. Then by \eqref{rl}, we have $fu_i=u_i\bar{f}$ and $u_if=\bar{f}u_i$ for all $0\leqslant i\leqslant m-1$.
For any $h\in H\otimes H$, we use
 $$h_{(s_1,t_1)\otimes (s_2,t_2)}$$
 to denote the homogeneous part of $h$ in $H_{s_{1},t_{1}}\otimes H_{s_{2},t_{2}} $.
Both these notations will be used frequently in the proof of the next proposition.

\begin{proposition}\label{p6.11} Keep the notations above. Let $H$ be a prime Hopf algebra of GK-dimension one satisfying (Hyp1) and (Hyp2). Assume that $\widetilde{H}=B(\underline{m},\omega,\gamma)$ and $\ord(\pi)=2\mi(\pi)>2$, then we have
 \begin{itemize}
 \item[(1)]$m|\omega,\;\; 2|\sum_{i=1}^{\theta}(m_i-1)(e_i-1), \;\;2|\sum_{i=1}^{\theta}(e_i-1)m_i\frac{\omega}{m}$.
 \item[(2)] As a Hopf algebra, $$H\cong D(\underline{m},d,\gamma)$$ constructed as in Subsection \ref{ss4.4} where $d=\frac{\omega}{m}$.
\end{itemize}

\end{proposition}
\noindent\emph{Proof.} We divide the proof into several steps.

\emph{Claim 1. We have $m|\omega$ and for $1\leq i\leq \theta, 0
\leq j\leq m-1$,  $y_{m_i}u_j=\phi_{m_i,j} u_{m_i+j}=\xi_{m_i} x^{dm_i}u_jy_{m_i}$ for $d=\frac{\omega}{m}$ and some $\xi_{m_i}\in \k$ satisfying $\xi_{m_i}^{e_i}=-1$.}

\emph{Proof of Claim 1:} By associativity of the multiplication, we have many equalities:
\begin{align*}
y_{m_i}u_jy_{m_i}^{e_i-1}&=\phi_{m_i,j}\varphi_{m_i,m_i+j}\varphi_{m_i,2m_i+j}\cdots \varphi_{m_i,(e_i-1)m_i+j}u_0\\
           &=\varphi_{m_i,j}\phi_{m_i,m_i+j}\varphi_{m_i,2m_i+j}\cdots \varphi_{m_i,(e_i-1)m_i+j}u_0\\
           &\cdots\\
           &=\varphi_{m_i,j}\varphi_{m_i,m_i+j}\varphi_{m_i,2m_i+j}\cdots \phi_{m_i,(e_i-1)m_i+j}u_0,
\end{align*}
which imply that \begin{equation}\label{eq6.8}\phi_{m_i,sm_i+j}\varphi_{m_i,tm_i+j}
=\varphi_{m_i,sm_i+j}\phi_{m_i,tm_i+j}\end{equation}
for all $0\leq s,
t\leq e_i-1$. Using associativity again, we have
\begin{align*}
y_{m_i}^{e_i}u_jy_{m_i}^{e_i(e_i-1)}&=(1-x^{\omega\frac{e_im_i}{m}})
u_j(1-x^{\omega\frac{e_im_i}{m}})^{e_i-1}\\
&=-x^{\omega\frac{e_im_i}{m}}
(1-x^{-\omega\frac{e_im_i}{m}})^{e_i}u_j\\
&=-x^{\omega\frac{e_im_i}{m}}(\varphi_{m_i,j}\varphi_{m_i,m_i+j}\varphi_{m_i,2m_i+j}\cdots \varphi_{m_i,(e_i-1)m_i+j})^{e_i}u_j\\
&=(\phi_{m_i,j}\varphi_{m_i,m_i+j}\varphi_{m_i,2m_i+j}\cdots \varphi_{m_i,(e_i-1)m_i+j})^{e_i}u_j\\
&=(\varphi_{m_i,j}\phi_{m_i,m_i+j}\varphi_{m_i,2m_i+j}\cdots \varphi_{m_i,(e_i-1)m_i+j})^{e_i}u_j\\
&\cdots\\
&=(\varphi_{m_i,j}\varphi_{m_i,m_i+j}\varphi_{m_i,2m_i+j}\cdots \phi_{m_i,(e_i-1)m_i+j})^{e_i}u_j,
\end{align*}
where the fourth ``$=$", for example, is gotten in the following way: We multiply $u_j$ by one $y_{m_i}$ from left side at first, then multiply it with $y_{m_i}^{e_i-1}$ from right side, then continue the procedures above. From these equalities, we have
 $$\phi_{m_i,sm_i+j}^{e_i}=-x^{\omega\frac{e_im_i}{m}}\varphi_{m_i,sm_i+j}^{e_i}$$ for all $0\leq s\leq e_i-1$.  This implies that
 $$e_i|\omega\frac{e_im_i}{m}.$$
 So, $m|\omega m_i$ for all $1\leq i\leq \theta$. Since $m$ is coprime to $(m_1,\ldots,m_{\theta})$, we have
 $$m|\omega.$$
 So $\phi_{m_i,sm_i+j}=\xi_{m_i,sm_i+j}x^{dm_i}\varphi_{m_i,sm_i+j}$ where $d=\frac{\omega}{m}$ and $\xi_{m_i,sm_i+j}\in \k$ satisfying $\xi_{m_i,sm_i+j}^{e_i}=-1$. We next want to prove that $\xi_{m_i,sm_i+j}$ does not depend on the number $sm_i+j.$ In fact, by equation \eqref{eq6.8}, we can see $\xi_{m_i,sm_i+j}=\xi_{m_i,tm_i+j}$ for all
$0\leq s, t\leq e_i-1$, and so we can write it through $\xi_{m_i,j}$. Now consider for any $1\leq i'\leq \theta$, by definition we have
$\phi_{m_{i'},0}u_{m_{i'}}=y_{m_{i'}}u_{0}$.  Therefore
\begin{align*} y_{m_i}y_{m_{i'}}u_{0}&= \phi_{m_{i'},0}y_{m_i}u_{m_{i'}}\\
&=\phi_{m_{i'},0}\xi_{m_i,m_{i'}}x^{m_id}u_{m_{i'}}y_{m_i},
\end{align*}
and
\begin{align*} y_{m_i}y_{m_{i'}}u_{0}&= y_{m_{i'}}y_{m_i}u_{0}\\
&=\xi_{m_i,0}x^{m_id}y_{m_{i'}}u_{0}y_{m_i}\\
&=\phi_{m_{i'},0}\xi_{m_i,0}x^{m_id}u_{m_{i'}}y_{m_i}.
\end{align*}
So, $\xi_{m_i,0}=\xi_{m_i,m_{i'}}$ which indeed tells us that $\xi_{m_i,j}$ does depend on $j$ (due to $j$ is generated by these $m_{i'}$'s) and so we write it as $\xi_{m_i}.$\qed

 In the following of the proof, $d$ is fixed to be the number $\omega/m$.

\noindent\emph{Claim 2. We have $u_jg=\lambda_j x^{-2d}gu_j$ for $\lambda_j=\pm \gamma^j$ and $0\leq j\leq m-1$.}

\emph{Proof of Claim 2:}
Since $g$ is invertible in $H$, $u_jg=\psi_jgu_j$ for some
invertible $\psi_j\in \k[x^{\pm 1}]$. Then $u_jg^m=\psi_j^mg^mu_j$ yields $\psi_j^m=x^{-2\omega}$. So $\psi_j=\lambda_jx^{-2d}$ for $\lambda_j\in \k$ with
$\lambda_j^m=1$. Our last task is to show that $\lambda_j=\pm \gamma^j$. To show this, we need a preparation, that is, we need to show that $u_ju_l\neq 0$ for all $j,l$. Otherwise, assume that there exist $j_0,l_0\in \{0,\ldots,m-1\}$ such that
$u_{j_0}u_{l_0}=0$. Using Claim 1, we can find that $u_{j_0}u_l\equiv 0$ and $u_ju_{l_0}\equiv 0$ for all $j,l$. Let $(u_{j_0})$ and $(u_{l_0})$ be the ideals generated by $u_{j_0}$ and $u_{l_0}$ respectively. Then it is not hard to find that
$(u_{j_0})(u_{l_0})=0$ which contradicts  $H$ being prime. So we always have
\begin{equation}\label{r3} u_ju_l\neq 0
\end{equation} for all $0\leq j,l\leqslant m-1$.

 Applying this observation, we have $0\neq
u_j^2\in H_{2,2+4j}=\k[x^{\pm 1}]y_{2j}g$,
$u_j^2g=\psi_j\overline{\psi_j}gu_j^2=\gamma^{2j}gu_j^2$. Thus
$\psi_j=\pm \gamma^jx^{-2d}$ which implies that $\lambda_j=\pm \gamma^j$. The proof is ended.\qed

We can say more about $\lambda_j$ at this stage.  By $0\neq
u_ju_lg=\gamma^{j+l}gu_ju_l$, we know that $\psi_j=\gamma^jx^{-2d}$ for all $j$ or
$\psi_j=-\gamma^jx^{-2d}$ for all $j$. So
 \begin{equation}\label{r5} \lambda_{j}=\gamma^{j}\ \ \textrm{or}\ \ \lambda_{j}=-\gamma^{j}
 \end{equation} for all $0\leq j\leq m-1$. In fact, we will show that  $\psi_j=\gamma^jx^{-2d}$ for all $j$ later.

\noindent\emph{Claim 3. For each $0\leqslant j\leqslant m-1$, there are $f_{jl},h_{jl}\in \k[x^{\pm 1}]$ with $h_{jl}$ monic such that \begin{equation}\label{r4}\D(u_j)=\sum_{k=0}^{m-1}f_{jk}u_k\otimes
h_{jk}g^ku_{j-k},\end{equation}
where the following $j-k$ is interpreted \emph{mod} $m$.}

\emph{Proof of Claim 3:} Since $u_j\in H_{1,1+2j}$, $\D(u_j)\in H_1^l\otimes H_{1+2j}^r$ by Lemma \ref{l2.8}. Noting that $H_1^l=\bigoplus_{k=0}^{m-1}H_{00}u_k$ and $H_{1+2j}^r=\bigoplus_{s=0}^{m-1}H_{00}g^su_{j-s}$, we can write $$\D(u_j)=\sum_{0\leq k,s\leq m-1}F^j_{ks}(u_k\otimes g^su_{j-s}),$$ where $F^j_{ks}\in H_{00}\otimes H_{00}$.
Then we divide the proof into two steps.

\noindent $\bullet$ \emph{Step} 1 ($\D(u_j)=\sum_{0\leq k\leqslant m-1}F^j_{kk}(u_k\otimes g^ku_{j-k})$).

Recall that $u_jg=\lambda_jx^{-2d}gu_j$, where $\lambda_j$ is either $\gamma^j$ for all $j$ or $-\gamma^j$ for all $j$. The equations
\begin{align*}
\D(u_jg)&=\D(u_j)\D(g)=\sum_{0\leq k,s\leq m-1}F^j_{ks}(u_k\otimes g^su_{j-s})(g\otimes g)\\
           &=\sum_{0\leq k,s\leq m-1}F^j_{ks}(\lambda_k x^{-2d}gu_k\otimes \lambda_{j-s} x^{-2d}g^{s+1}u_{j-s})\\
           &=\sum_{0\leq k,s\leq m-1}\lambda_k\lambda_{j-s}( x^{-2d}g\otimes x^{-2d}g)F^j_{ks}(u_k\otimes g^s u_{k-s})
\end{align*}
and
\begin{align*}
\D(\lambda_j x^{-2d}gu_j)&=\lambda_j ( x^{-2d}g\otimes x^{-2d}g) \sum_{0\leq k,s\leq m-1}F^j_{ks}(u_k\otimes g^su_{j-s})\\
           &=\sum_{0\leq k,s\leq m-1}\lambda_j ( x^{-2d}g\otimes x^{-2d}g)F^j_{ks}(u_k\otimes g^s u_{j-s})
\end{align*}
imply that $\lambda_j=\lambda_k\lambda_{j-s}$ for all $k, s$. If $\lambda_j=-\gamma^j$ for all $j$, then we have $-\gamma^j=\lambda_j=\lambda_k\lambda_{j-s}=\gamma^{k+j-s}.$ This implies $k=s\pm m/2$. Applying $(\epsilon\otimes \id)$ to $\D(u_j)$, $$(\epsilon\otimes \id)\D(u_j)=(\epsilon\otimes \id)(F^j_{0,\; m/2})g^{m/2}u_{j-m/2}\neq u_j,$$ which is absurd. If $\lambda_j=\gamma^j$ for all $j$, then $\gamma^j=\lambda_j=\lambda_k\lambda_{j-s}=\gamma^{k+j-s}$. This implies $k=s$ (which is compatible with the equality $(\epsilon\otimes \id)\D(u_j)=u_j$). So we get $F^j_{ks}\neq 0$ only if $k=s$ and $\lambda_j=\gamma^j$ for all $j$. Thus we have $\D(u_j)=\sum_{0\leq k\leq m-1}F^j_{kk}(u_k\otimes g^ku_{j-k})$ for all $j$.

\noindent $\bullet$ \emph{Step} 2 (There exist $f_{jk}, h_{jk} \in H_{00}$ with $h_{jk}$ monic such that $F^j_{kk}=f_{jk}\otimes h_{jk}$ for $0\leq j,k\leq m-1$).

We replace $F^j_{kk}$ by $F^j_{k}$ for convenience. Since
\begin{align*}
(\D\otimes \id)\D(u_j)&=(\D\otimes \id)(\sum_{0\leq k\leq m-1}F^j_{k}(u_k\otimes g^ku_{j-k}))\\
           &=\sum_{0\leq k\leq m-1}(\D\otimes \id)(F^j_{k})(\sum_{0\leq s\leq m-1}F_s^k(u_s\otimes g^su_{k-s})\otimes g^ku_{j-k})\\
           &=\sum_{0\leq k, s\leqslant m-1}(\D\otimes \id)(F^j_{k})(F_s^k\otimes 1)(u_s\otimes g^su_{k-s}\otimes g^ku_{j-k})
\end{align*}
and
\begin{align*}
(\id\otimes\D)\D(u_j)&=(\id\otimes\D)(\sum_{0\leq k\leq m-1}F^j_{k}(u_k\otimes g^ku_{j-k}))\\
           &=\sum_{0\leq k\leq m-1}(\id\otimes\D)(F^j_{k})(u_k\otimes(\sum_{0\leq s\leq m-1}F_s^{j-k}(g^ku_s\otimes g^{k+s}u_{j-k-s}))\\
           &=\sum_{0\leq k, s\leqslant m-1}(\id\otimes\D)(F^j_{s})(1\otimes F_{k-s}^{j-s})(u_s\otimes g^su_{k-s}\otimes g^ku_{j-k}),
\end{align*}
we have \begin{equation}\label{eqclaim3}(\D\otimes \id)(F^j_{k})(F_s^k\otimes 1)=(\id\otimes\D)(F^j_{s})(1\otimes F_{k-s}^{j-s})\end{equation} for all $0\leq j,k, s\leq m-1$.

Begin with the case $j=k=s=0$. Let $F_0^0=\sum_{p,q}k_{pq}x^p\otimes x^q$. Comparing equation
\begin{align*}
(\D\otimes \id)(F^0_{0})(F_0^0\otimes 1)&=(\sum_{p,q}k_{pq}x^p\otimes x^p\otimes x^q)(\sum_{p',q'}k_{p'q'}x^{p'}\otimes x^{q'}\otimes 1)\\
           &=(\sum_{p,q,p',q'}k_{pq}k_{p'q'}x^{p+p'}\otimes x^{p+q'}\otimes x^q)
\end{align*}
and equation
\begin{align*}
(\id\otimes\D)(F^0_{0})(1\otimes F_0^0)&=(\sum_{p,q}k_{pq}x^p\otimes x^q\otimes x^q)(\sum_{p',q'}k_{p'q'}1 \otimes x^{p'}\otimes x^{q'} )\\
           &=(\sum_{p,q,p',q'}k_{pq}k_{p'q'}x^{p}\otimes x^{q+p'}\otimes x^{q+q'}),
\end{align*}
one can see that $p=q=0$ by comparing the degrees of $x$ in these two expressions. Then $F_0^0= 1\otimes 1$ by applying $(\epsilon\otimes \id)\D$ to $u_0$.
Next, consider the case $k=s=0$. Write $F_0^j=\sum_{p,q}k_{pq}x^p\otimes x^q$. Similarly, we have $F_0^j= x^{a_j}\otimes 1$ for some $a_j\in \mathbb{Z}$ by the equation $$(\D\otimes \id)(F^j_{0})(F_0^0\otimes 1)=(\id\otimes\D)(F^j_{0})(1\otimes F_0^j).$$  Finally, write $F_k^j=\sum_{p,q}k_{pq}x^p\otimes x^q$ and consider the case $s=0$.
Let $F_0^j= x^{a_j}\otimes 1$ and $F_0^k=x^{a_k}\otimes 1$. The equation
\begin{align*}
&(\sum_{p,q}k_{pq}x^{p+a_k}\otimes x^p\otimes x^q)
=(\D\otimes \id)(F^j_{k})(F_0^k\otimes 1)\\
&\quad \quad=(\id\otimes\D)(F^j_{0})(1\otimes F_k^j)
=(\sum_{p,q}k_{pq}x^{a_j}\otimes x^p\otimes x^q)
\end{align*}
shows that $p=a_j-a_k$, that is, $F_k^j=x^{c_{jk}}\otimes \beta_{jk}$ some $c_{jk}\in \mathbb{Z}, \beta_{jk}\in H_{00}$.

By steps 1 and 2, $F^j_k$ can be written as $f_{jk}\otimes h_{jk}$ with $h_{jk}$ monic after multiplying suitable scalar, where $f_{jk}, h_{jk}\in \k[x^{\pm 1}]$. That is, $$\D(u_j)=\sum_{k=0}^{m-1}f_{jk}u_k\otimes
h_{jk}g^ku_{j-k},$$ where $f_{jk},h_{jk}\in \k[x^{\pm 1}]$ with $h_{jk}$ monic. \qed

Since $\lambda_j=\gamma^j$ for all $j$ has been shown above, we can improve Claim 2 as

\noindent \emph{Claim 2'. We have $u_jg=\gamma^{j} x^{-2d}gu_j$ for $0\leq j\leq m-1$.}

By Claim 2', we have a unified formula in $H$: For all $s\in \mathbbm{Z}$,
\begin{equation}\label{r7} u_jg^s=\gamma^{js}x^{-2sd}g^su_j.\end{equation}

\noindent\emph{Claim 4. We have $\phi_{m_i,j}=1-\gamma^{-m_i(m_i+j)}x^{m_id}=1-\gamma^{-m_i^2(1+j_i)}x^{m_id}$ for $1\leq i\leq \theta$ and $0\leq j\leq m-1$.}

\emph{Proof of Claim 4:} By Claim 3, there are polynomials $f_{0j},h_{0j},$ such that
$$\D(u_0)=u_0\otimes u_0+ f_{01}u_1\otimes
h_{01}gu_{m-1}+ \cdots  + f_{0, m-1}u_{m-1}\otimes h_{0,
m-1}g^{m-1}u_1.$$

Firstly, we will show $\phi_{m_i,0}=1-\gamma^{-m_i^2}x^{m_id}$ by considering the equations
$$\D(y_{m_i}u_0)_{11\otimes (1,1+2m_i)}=\D(\xi_{m_i} x^{m_id} u_0y_{m_i})_{11\otimes (1,1+2m_i)}=\D(\phi_{m_i,0}u_{m_i})_{11\otimes (1,1+2m_i)}.$$
Direct computations show that
\begin{align*}
&\D(y_{m_i}u_0)_{11\otimes (1,1+2m_i)}\\
&=u_0\otimes y_{m_i}u_0+y_{m_i} f_{0, (e_i-1)m_i} u_{(e_i-1)m_i}\otimes g^{m_i} h_{0, (e_i-1)m_i} g^{(e_i-1)m_i}u_{-(e_i-1)m_i}\\
&=u_0\otimes \phi_{m_i,0}u_{m_i}+ f_{0, (e_i-1)m_i}\phi_{m_i,(e_i-1)m_i} u_{0}\otimes x^{e_im_id} h_{0, (e_i-1)m_i} u_{-(e_i-1)m_i},\\
&\D(\xi_{m_i} x^{m_id} u_0y_{m_i})_{11\otimes (1,1+2m_i)}=\xi_{m_i} x^{m_id}u_0\otimes x^{m_id}u_0y_{m_i}\\
&\quad\quad
+ \xi_{m_i} x^{m_id} f_{0, (e_i-1)m_i} u_{(e_i-1)m_i}y_{m_i}\otimes x^{m_id} h_{0, (e_i-1)m_i} g^{(e_i-1)m_i}u_{-(e_i-1)m_i} g^{m_i}\\
&= x^{m_id}u_0\otimes \phi_{m_i,0}u_{m_i}+ f_{0, (e_i-1)m_i}\phi_{m_i,(e_i-1)m_i} u_{0}\otimes \gamma^{m_i^2}x^{(e_i-1)m_id} h_{0, (e_i-1)m_i}u_{-(e_i-1)m_i}.
\end{align*}
Owing to $\D(y_{m_i}u_0)_{11\otimes (1,1+2m_i)}=\D(\xi_{m_i} x^{m_id} u_0y_{m_i})_{11\otimes (1,1+2m_i)}$,
\begin{align*}&(1-x^{m_id})u_0\otimes \phi_{m_i,0}u_{m_i}\\
& \quad + f_{0, (e_i-1)m_i}\phi_{m_i,(e_i-1)m_i}u_0\otimes (x^{m_id}-\gamma^{m_i^2}) x^{(e_i-1)m_id} h_{0, (e_i-1)m_i} u_{-(e_i-1)m_i}\\
&=0.\end{align*}

Thus we can assume $\phi_{m_i,0}=c_0 (x^{m_id}-\gamma^{m_i^2}) x^{(e_i-1)m_id} h_{0,
(e_i-1)m_i}$ for some $0\neq c_0\in \k$. Then $1-x^{m_id}=-c_0^{-1}f_{0,
(e_i-1)m_i}\phi_{m_i,(e_i-1)m_i}$. Therefore,
\begin{align*}
&\D(y_{m_i}u_0)_{11\otimes (1,1+2m_i)}\\
&=u_0\otimes \phi_{m_i,0}u_{m_i}-c_0(1-x^{m_id})u_0\otimes \frac{1}{c_0}\frac{x^{m_id}\phi_{m_i,0}}{x^{m_id}-\gamma^{m_i^2}}  u_{-(e_i-1)m_i}\\
                       &=u_0\otimes (1-\frac{x^{m_id}}{x^{m_id}-\gamma^{m_i^2}})\phi_{m_i,0}u_{-(e_i-1)m_i}
                       +x^{m_id}u_{0}\otimes \frac{x^{m_id}\phi_{m_i,0}}{x^{m_id}-\gamma^{m_i^2}}  u_{-(e_i-1)m_i}\\
                       &=u_0\otimes -\frac{\gamma^{m_i^2}}{x^{m_id}-\gamma^{m_i^2}}\phi_{m_i,0}u_{-(e_i-1)m_i}
                       +x^{m_id}u_{0}\otimes \frac{x^{m_id}\phi_{m_i,0}}{x^{m_id}-\gamma^{m_i^2}}  u_{-(e_i-1)m_i},
\end{align*} where $\frac{\phi_{m_i,0}}{x^{m_id}-\gamma^{m_i^2}}$ is understood as $c_0 x^{(e_i-1)m_i}h_{0,(e_i-1)m_i}$. Note that $\D(y_{m_i}u_0)_{11\otimes (1,1+2m_i)}=\D(\phi_{m_i,0}u_{m_i})_{11\otimes (1,1+2m_i)}=\D(\phi_{m_i,0})(f_{m_i,0}u_0\otimes u_{m_i})$. From which, we get $\phi_{m_i,0}=1+c x^{m_id}$ for
some $c\in \k$. Then it is not hard to see that $f_{m_i, 0}=1,
 h_{0, (e_i-1)m_i}=x^{-(e_i-1)m_id}$ and
$c=-\gamma^{-m_i^2}$. So $\phi_{m_i,0}=1-\gamma^{-m_i^2}x^{m_id}$.

Secondly, we want to determine $\phi_{m_i,j}$ for $0\leq j\leq m-1$. We note that we always have $h_{j0}=f_{jj}=1$ due to $(\e\otimes \id)\D(u_j)=u_j$. To determine $\phi_{m_i,j}$, we  will prove the fact
\begin{equation}\label{r6} f_{j0}=1 \end{equation}
 for all $0\leqslant j\leqslant m-1$ at the same time. We proceed by induction. We already know that $f_{00}=h_{00}=f_{m_i 0}=1$. Assume that $f_{j, 0}=1$ now. We consider the case $j+m_i$. Similarly, direct computations show that
 \begin{align*}
&\D(y_{m_i}u_j)_{11\otimes (1, 1+2j+2m_i)}\\
&=u_0\otimes y_{m_i}u_j+y_{m_i} f_{j, (e_i-1)m_i} u_{(e_i-1)m_i}\otimes g^{m_i} h_{j, (e_i-1)m_i} g^{(e_i-1)m_i}u_{m_i+j}\\
&=u_0\otimes \phi_{m_i,j}u_{m_i+j}+f_{j, (e_i-1)m_i}\phi_{m_i,(e_i-1)m_i} u_{0}\otimes x^{e_im_id} h_{j, (e_i-1)m_i}u_{m_i+j},\\
&\D(\xi_{m_i} x^{m_id} u_jy_{m_i})_{11\otimes (1, 1+2j+2m_i)}\\
&=\xi_{m_i} x^{m_id}u_0\otimes x^{m_id}u_jy_{m_i}+ \xi_{m_i} x^{m_id} f_{j, (e_i-1)m_i} u_{(e_i-1)m_i} y_{m_i}\otimes x^{m_id} h_{j, (e_i-1)m_i} g^{(e_i-1)m_i}u_{j+m_i} g^{m_i}\\
&=x^{m_id}u_0\otimes \phi_{m_i,j}u_{m_i+j}+  f_{j, (e_i-1)m_i}\phi_{m_i,(e_i-1)m_i} u_{0} \otimes \gamma^{m_i(j+m_i)}x^{(e_i-1)m_id} h_{j, (e_i-1)m_i} u_{j+m_i}.
\end{align*}
By $\D(y_{m_i}u_j)_{11\otimes (1, 1+2j+2m_i)}=\D(\xi_{m_i} x^{m_id} u_jy_{m_i})_{11\otimes (1, 1+2j+2m_i)}$,
\begin{align*}&(1-x^{m_id})u_0\otimes \phi_{m_i,j}u_{m_i+j}\\
&\quad + f_{j, (e_i-1)m_i}\phi_{m_i,(e_i-1)m_i}u_0\otimes
 (x^{m_id}-\gamma^{m_i(m_i+j)}) x^{(e_i-1)m_id} h_{j, (e_i-1)m_i} u_{j+m_i}\\
 &=0.\end{align*}
Thus we can assume $\phi_{m_i,j}=c_j (x^{m_id}-\gamma^{m_i(m_i+j)}) x^{(e_i-1)m_id} h_{j, (e_i-1)m_i}$ for some $0\neq c_j\in \k$. Then $1-x^{m_id}=-c_j^{-1}f_{j,
(e_i-1)m_i}\phi_{m_i,(e_i-1)m_i}$. Therefore
\begin{align*}
&\D(y_{m_i}u_j)_{11\otimes (1, 1+2j+2m_i)}\\
&=u_0\otimes \phi_{m_i,j}u_{m_i+j}-c_j(1-x^{m_id})u_0\otimes \frac{1}{c_j}\frac{x^{m_id}}{x^{m_id}-\gamma^{m_i(m_i+j)}} \phi_{m_i,j} u_{m_i+j}\\
&=u_0\otimes \frac{-\gamma^{m_i(m_i+j)}}{x^{m_id}-\gamma^{m_i(m_i+j)}}\phi_{m_i,j}u_{m_i+j}
+x^{m_id}u_0\otimes \frac{x^{m_id}}{x^{m_id}-\gamma^{m_i(m_i+j)}} \phi_{m_i,j} u_{m_i+j}.
\end{align*}
Note that $\D(y_{m_i}u_j)_{11\otimes (1, 1+2j+2m_i)}=\D(\phi_{m_i,j}u_{m_i+j})_{11\otimes (1, 1+2j+2m_i)}=\D(\phi_{m_i,j})(f_{m_i+j,
0}u_0\otimes h_{m_i+j, 0}u_{m_i+j})$. Comparing the first components of
$$\D(y_{m_i}u_j)_{11\otimes (1, 1+2j+2m_i)}\;\; \textrm{and}\;\; \D(\phi_{m_i,j}u_{m_i+j})_{11\otimes (1, 1+2j+2m_i)},$$ we get $\phi_{m_i,j}=1-\gamma^{-m_i(m_i+j)}x^{m_id}$ similarly. And it is
not hard to see that $f_{m_i+j, 0}=1$. Since here $i$ is arbitrary and $m_1,\ldots,m_\theta$ generate $0,1,\ldots,m-1$,  we prove that $f_{j, 0}=h_{j, 0}=1$ at the same time for all $0\leq j\leq m-1$.\qed

\noindent \emph{Claim 5. The coproduct of $H$ is given by
$$\D(u_j)=\sum_{k=0}^{m-1}\gamma^{k(j-k)}u_k\otimes
x^{-kd}g^ku_{j-k}$$for $0\leq j\leqslant m-1$.}

\emph{Proof of Claim 5:} By Claim 3, $\D(u_j)=\sum_{k=0}^{m-1}f_{jk}u_k\otimes
h_{jk}g^ku_{j-k}$. So, to show this claim, it is enough to determine the explicit form of every $f_{jk}$ and $h_{jk}$. By \eqref{r6} and the sentence before it, $f_{j,0}=h_{j,0}=1$ for all $0\leq j\leq m-1$. We will prove that $f_{jk}=\gamma^{k(j-k)}$ and $h_{jk}=x^{-kd}$ for all $0\leqslant j,
k \leqslant m-1$ by induction. So it is enough to show that
$f_{j,k+m_i}=\gamma^{(k+m_i)(j-k-m_i)}$ and $h_{j, k+m_i}=x^{-(k+m_i)d}$
 for all $1\leq i\leq \theta$ under the hypothesis of $f_{jk}=\gamma^{k(j-k)}$ and
 $h_{jk}=x^{-kd}$. In fact, for $1\leq i\leq \theta$,
\begin{align*}
&\D(y_{m_i}u_j)_{(1, 1+2k+2m_i)\otimes (1+2k+2m_i, 1+2j+2m_i)}\\
&=y_{m_i}f_{jk}u_k\otimes g^{m_i}h_{jk}g^ku_{j-k}+ f_{j, k+m_i} u_{k+m_i}\otimes y_{m_i} h_{j, k+m_i} g^{k+m_i}u_{j-k-m_i}\\
&=f_{jk}y_{m_i}u_k\otimes h_{jk}g^{k+m_i}u_{j-k}+ f_{j, k+m_i} u_{k+m_i}\otimes \gamma^{(k+m_i)m_i} h_{j, k+m_i} g^{k+m_i}y_{m_i}u_{j-k-m_i},\\
&\D(\xi_{m_i} x^{m_id} u_jy_{m_i})_{(1, 1+2k+2m_i)\otimes (1+2k+2m_i, 1+2j+2m_i)}\\
&=\xi_{m_i} x^{m_id}f_{jk}u_ky_{m_i}\otimes x^{m_id}h_{jk}g^ku_{j-k}g^{m_i}\\
&\quad+ \xi_{m_i} x^{m_id}f_{j, k+m_i} u_{k+m_i}\otimes x^{m_id} h_{j, k+m_i} g^{k+m_i}u_{j-k-m_i} y_{m_i}\\
&=f_{jk}y_{m_i}u_k\otimes \gamma^{(j-k)m_i}x^{-m_id}h_{jk}g^{m_i+k}u_{j-k}\\
&\quad+  x^{m_id}f_{j, k+m_i} u_{k+m_i}\otimes  h_{j, k+m_i} g^{k+m_i}y_{m_i}u_{j-k-m_i}.
\end{align*}
Since they are equal,
\begin{align*}&f_{jk}y_{m_i}u_k\otimes (1-\gamma^{(j-k)m_i}x^{-m_id})h_{jk}g^{m_i+k}u_{j-k}\\
 &= (x^{m_id}-\gamma^{(k+m_i)m_i})f_{j, k+m_i} u_{k+m_i}\otimes h_{j, k+m_i} g^{k+m_i}y_{m_i} u_{j-k-m_i}.\end{align*}
Using induction and the expression of $\phi_{m_i,k}$, we have
\begin{align*}&\gamma^{k(j-k)}(1-\gamma^{-m_i(m_i+k)}x^{m_id})u_{k+m_i}\otimes
(1-\gamma^{(j-k)m_i}x^{-m_id})x^{-kd}g^{m_i+k}u_{j-k}\\& =
\gamma^{k(j-k)}(1-\gamma^{-m_i(m_i+k)}x^{m_id})u_{k+m_i}\otimes
(x^{m_id}-\gamma^{(j-k)m_i})x^{-(k+m_i)d}g^{m_i+k}u_{j-k}\\
&= (x^{m_id}-\gamma^{(k+m_i)m_i})f_{j, k+m_i} u_{k+m_i}\otimes (1-\gamma^{-(j-k)m_i}x^{m_id})h_{j, k+m_i} g^{k+m_i}u_{j-k}.\end{align*}
This implies that
$h_{j, k+m_i}=x^{-(k+m_i)d}$ and
$$f_{j,k+m_i}=\gamma^{k(j-k)-m_i^2-m_ik+m_ij-m_ik}=\gamma^{(k+m_i)(j-k-m_i)}.$$\qed

\noindent \emph{Claim 6. For $0\leqslant j,l\leqslant m-1$, the multiplication between $u_j$ and $u_l$ satisfies that $$u_ju_l=\frac{1}{m}x^{a} \prod_{i=1}^{\theta}
(-1)^{l_i}\xi_{m_i}^{-l_i}\gamma^{m_i^2\frac{l_i(l_i+1)}{2}}[j_i,e_i-2-l_i]_{m_i}y_{j+l}g$$
for some $a\in \mathbbm{Z}$ and where $[-,-]_{m_i}$ is defined as \eqref{eqpre} and $j+l$ is interpreted \emph{mod} $m$.}

\emph{Proof of Claim 6:} We need to consider the relation between $u_0^2$ and $u_ju_{m-j}$
for all $1\leqslant j \leqslant m-1$ at first. We remark that as before for any $k\in \Z$ we write $u_{k}:=u_{\overline{k}}$ where $\overline{k}$ is the remainder of $k$ dividing by $m$.  Thus  $u_j=u_{j_1m_1+\ldots+j_{\theta}m_{\theta}}$ and $u_{m-j}=u_{(e_1-j_1)m_1+\ldots+(e_\theta-j_\theta)m_{\theta}}.$

By definition, $x^{m_id}\overline{\phi_{m_i,sm_i}}=-\gamma^{-m_i^2(s+1)}\phi_{m_i,(e_i-s-2)m_i}$ for all $s$.  Then
\begin{align*}
&y_{m_1}^{e_1}y_{m_2}^{e_2}\cdots y_{m_\theta}^{e_\theta} u_0^2\\
&=\xi_{m_1}^{e_1-j_1}\xi_{m_2}^{e_2-j_2}\cdots \xi_{m_\theta}^{e_\theta-j_\theta} x^{(e_1-j_1)m_1d+\ldots +(e_{\theta}-j_\theta)m_\theta d} y_{j}u_0y_{m-j}u_0\\
&=\prod_{i=1}^{\theta}[\xi_{m_i}^{e_i-j_i}x^{(e_i-j_i)m_id}\phi_{m_{i},0}\cdots \phi_{m_i,(j_i-1)m_i}]u_j \prod_{i=1}^{\theta}[\phi_{m_{i},0}\cdots \phi_{m_i,(e_i-j_i-1)m_i}]u_{m-j}\\
&=\prod_{i=1}^{\theta}[\xi_{m_i}^{e_i-j_i}x^{(e_i-j_i)m_id}\phi_{m_{i},0}\cdots \phi_{m_i,(j_i-1)m_i}\overline{\phi_{m_{i},0}}\cdots \overline{\phi_{m_i,(e_i-j_i-1)m_i}}]u_j u_{m-j}\\
&=\prod_{i=1}^{\theta}[(-1)^{e_i-j_i}\xi_{m_i}^{e_i-j_i}
\gamma^{-m_i^2\frac{(e_i-j_i)(e_i-j_i+1)}{2}}\phi_{m_{i},0}\cdots \phi_{m_i,(e_i-2)m_i}{\phi_{m_{i},(j_i-1)m_i}}]u_j u_{m-j}.
\end{align*}
By $\phi_{m_{i},0}\cdots \phi_{m_i,(e_i-2)m_i}\phi_{m_i,(e_i-1)m_i}=1-x^{e_im_id}$ (see Lemma \ref{l3.4} (2)), we have
\begin{align*}
\label{r8}&\phi_{m_1,(e_1-1)m_1}\cdots \phi_{m_\theta,(e_\theta-1)m_\theta} y_{m_1}^{e_1}\cdots y_{\theta}^{e_\theta} u_0^2\\
&=\prod_{i=1}^{\theta}[(-1)^{e_i-j_i}\xi_{m_i}^{e_i-j_i}
\gamma^{-m_i^2\frac{(e_i-j_i)(e_i-j_i+1)}{2}}(1-x^{e_im_ed}){\phi_{m_{i},(j_i-1)m_i}}]u_ju_{m-j} .\end{align*}
Due to $y_{m_i}^{e_i}=1-x^{e_im_id}$, we get a desired formula
\begin{equation}\label{r8} \prod_{i=1}^{\theta}[\phi_{m_i,(e_i-1)m_i}]u_0^2=\prod_{i=1}^{\theta}[(-1)^{e_i-j_i}\xi_{m_i}^{e_i-j_i}
\gamma^{-m_i^2\frac{(e_i-j_i)(e_i-j_i+1)}{2}}{\phi_{m_{i},(j_i-1)m_i}}]u_ju_{m-j} .
\end{equation}
Since $u_0^2, u_ju_{m-j}\in H_{22}=\k[x^{\pm 1}]g$, we may assume
$u_0^2=\alpha_0 g, u_ju_{m-j}=\alpha_j g$ for some $\alpha_0,
\alpha_j \in \k[x^{\pm 1}]$ for all $1\leqslant j \leqslant m-1$.

Then Equation \eqref{r8} implies $\alpha_0=\alpha \prod_{i=1}^{\theta}[\phi_{m_{i},0}\cdots \phi_{m_i,(e_i-2)m_i}]$ for some $\alpha \in \k[x^{\pm 1}]$. We claim $\alpha$ is
invertible. Indeed, by
$$\prod_{i=1}^{\theta}[\phi_{m_i,(e_i-1)m_i}]\alpha_0=\prod_{i=1}^{\theta}[(-1)^{e_i-j_i}\xi_{m_i}^{e_i-j_i}
\gamma^{-m_i^2\frac{(e_i-j_i)(e_i-j_i+1)}{2}}{\phi_{m_{i},(j_i-1)m_i}}]
\alpha_j,$$ we have
$$\alpha_j=\prod_{i=1}^{\theta}[(-1)^{j_i-e_i}\xi_{m_i}^{j_i-e_i}
\gamma^{m_i^2\frac{(e_i-j_i)(e_i-j_i+1)}{2}}]j_i-1,j_i-1[_{m_i}]\alpha.$$ Then
$$H_{11}\cdot H_{11} + \sum_{j=1}^{m-1}H_{1, 1+2j}\cdot H_{1,
1+2(m-j)}\subseteq \alpha H_{22}.$$
By the strong grading of $H$,
$$H_{22}=H_{11}\cdot H_{11} + \sum_{j=1}^{m-1}H_{1, 1+2j}\cdot H_{1,
1+2(m-j)},$$  which shows that $\alpha$ must be invertible. Since
$\epsilon(\alpha_{0})=1, \epsilon(\phi_{m_i,0}\cdots \phi_{m_i,(e_i-2)m_i})=e_i$ and $m=e_1\cdots e_\theta$, we
may assume $\alpha_0=\frac{1}{m}x^a \prod_{i=1}^{\theta}[\phi_{m_i,0}\cdots \phi_{m_i,(e_i-2)m_i}]$ for
some integer $a$. Thus
\begin{align*}
u_ju_{m-j}&=\frac{1}{m}x^{a}\prod_{i=1}^{\theta}[(-1)^{j_i-e_i}\xi_{m_i}^{j_i-e_i}
\gamma^{m_i^2\frac{(e_i-j_i)(e_i-j_i+1)}{2}}]j_i-1,j_i-1[_{m_i}]\;g.
\end{align*}

Now \begin{align*} &y_{j}y_{l}u_0^2\\
&=\prod_{i=1}^{\theta}\xi_{m_i}^{l_i}x^{l_im_id} y_{j}u_0y_{l}u_{0}\\
&=\prod_{i=1}^{\theta}[\xi_{m_i}^{l_i}x^{l_im_id}\phi_{m_i,0}\phi_{m_i,m_i}\cdots \phi_{m_i,(j_i-1)m_i}]u_j \prod_{i=1}^{\theta}[\phi_{m_i,0}\phi_{m_i,m_i}\cdots \phi_{m_i,(l_i-1)m_i}]u_l\\
&=\prod_{i=1}^{\theta}[\xi_{m_i}^{l_i}x^{l_im_id}\phi_{m_i,0}\cdots \phi_{m_i,(j_i-1)m_i}\overline{\phi_{m_i,0}}\cdots \overline{\phi_{m_i,(l_i-1)m_i}}]u_ju_{l}\\
&=\prod_{i=1}^{\theta}[(-1)^{l_i}\xi_{m_i}^{l_i}\gamma^{-m_i^2\frac{l_i(l_i+1)}{2}}
\phi_{m_i,0}\cdots \phi_{m_i,(j_i-1)m_i}{\phi_{m_i,(e_i-2)m_i}}\cdots {\phi_{m_i,(e_i-1-l_i)m_i}}]u_ju_{l}
\end{align*}
For each $1\leq i\leq \theta$, we find that
\begin{align*}&\phi_{m_i,0}\cdots \phi_{m_i,(j_i-1)m_i}{\phi_{m_i,(e_i-2)m_i}}\cdots {\phi_{m_i,(e_i-1-l_i)m_i}}\\
&=\begin{cases}
\phi_{m_i,0}\cdots \phi_{m_i,(j_i-1)m_i}\phi_{m_i,(e_i-1-l_i)m_i}\cdots\phi_{m_i,(e_i-2)m_i}, & \textrm{if}\; j_i+l_i\leq e_i-2
\\\phi_{m_i,0}\cdots \phi_{m_i,(e_i-2)m_i}, & \textrm{if}\; j_i+l_i=e_i-1 \\
\phi_{m_i,0}\cdots \phi_{m_i,(j_i-1)m_i}\phi_{m_i,(e_i-1-l_i)m_i}\cdots\phi_{m_i,(e_i-1)m_i}, & \textrm{if}\;
j_i+l_i\geq e_i.
\end{cases}
\end{align*}

Using the same method to compute $u_ju_{m-j}$ given above and the notations introduced in equations \eqref{eqomit} and \eqref{eqpre}, we have a unified
expression:
\begin{align*}
u_ju_l&=\frac{1}{m}x^{a} \prod_{i=1}^{\theta}
(-1)^{l_i}\xi_{m_i}^{-l_i}\gamma^{m_i^2\frac{l_i(l_i+1)}{2}}[j_i,e_i-2-l_i]_{m_i}y_{j+l}g\\
      &=\frac{1}{m}x^{a} \prod_{i=1}^{\theta}
(-1)^{l_i}\xi_{m_i}^{-l_i}\gamma^{m_i^2\frac{l_i(l_i+1)}{2}}]-1-l_i,j_i-1[_{m_i}y_{j+l}g
\end{align*}
for all $0\leq j, l\leq m-1$. \qed

\noindent \emph{Claim 7. We have $\xi_{m_i}^2=\gamma^{m_i}, \ a=-\frac{2+\sum_{i=1}^{\theta}(e_i-1)m_i}{2}d$ and
$$S(u_j)=x^{b}g^{m-1}\prod_{i=1}^{\theta}[(-1)^{j_i}\xi_{m_i}^{-j_i}\gamma^{-
m_i^2\frac{j_i(j_i+1)}{2}}
x^{j_im_id}
g^{-j_im_i}]u_j$$ for $0\leqslant j\leqslant m-1$ and $b=(1-m)d-\frac{\sum_{i=1}^{\theta}(e_i-1)m_i}{2}d.$}

\emph{Proof of Claim 7:}
By Lemma \ref{l2.8} (3), $S(H_{ij})=H_{-j, -i}$ and thus  $S(u_0)=hg^{m-1}u_0$ for some
$h\in \k[x^{\pm 1}]$. Combining
\begin{align*}
S(y_{m_i}u_0)&=S(u_0)S(y_{m_i})=hg^{m-1}u_0  (-y_{m_i}g^{-m_i})=-\xi_{m_i}^{-1} x^{-m_id}hg^{m-1}y_{m_i}u_0g^{-m_i}\\
       &=-\xi_{m_i}^{-1}\gamma^{-m_i^2} x^{m_id}hg^{m-1-m_i}y_{m_i}u_0=-\xi_{m_i}^{-1}\gamma^{-m_i^2} x^{m_id}hg^{m-1-m_i}\phi_{m_i,0}u_{m_i}
\end{align*} with
$$S(y_{m_i}u_0)=S(\phi_{m_i,0} u_{m_i})=S(u_{m_i})S(\phi_{m_i,0})=\phi_{m_i,0} S(u_{m_i}),$$
we get $S(u_{m_i})=-\xi_{m_i}^{-1} \gamma^{-m_i^2} x^{m_id} h g^{m-1-m_i} u_{m_i}$. The computation above tells us that we can prove that
$$S(u_j)=hg^{m-1}\prod_{i=1}^{\theta}[(-1)^{j_i}\xi_{m_i}^{-j_i}\gamma^{-
m_i^2\frac{j_i(j_i+1)}{2}}
x^{j_im_id}
g^{-j_im_i}]u_j$$ by induction. In fact, in order to prove above formula for the antipode it is enough to show that it is still valid for $j+m_i$ for all $1\leq i\leq \theta$ under assumption that it is true for $j$. By combining
\begin{align*}
&S(y_{m_i}u_j)\\
&=S(u_j)S(y_{m_i})\\
&=hg^{m-1}\prod_{s=1}^{\theta}[(-1)^{j_s}\xi_{m_s}^{-j_s}\gamma^{-
m_s^2\frac{j_s(j_s+1)}{2}}
x^{j_sm_sd}
g^{-j_sm_s}]u_j (-y_{m_i}g^{-m_i})\\
       &=-\xi_{m_i}^{-1}x^{-m_id}hg^{m-1}\prod_{s=1}^{\theta}[(-1)^{j_s}\xi_{m_s}^{-j_s}\gamma^{-
m_s^2\frac{j_s(j_s+1)}{2}}
x^{j_sm_sd}
g^{-j_sm_s}]y_{m_i}u_jg^{-m_i}\\
&= -\xi_{m_i}^{-1}x^{-m_id}hg^{m-1}\prod_{s=1}^{\theta}[(-1)^{j_s}\xi_{m_s}^{-j_s}\gamma^{-
m_s^2\frac{j_s(j_s+1)}{2}}
x^{j_sm_sd}
g^{-j_sm_s}]y_{m_i}\gamma^{-m_i^2j_i}x^{2m_id}g^{-m_i}u_j\\
&= -\xi_{m_i}^{-1}x^{m_id}hg^{m-1}\prod_{s=1}^{\theta}[(-1)^{j_s}\xi_{m_s}^{-j_s}\gamma^{-
m_s^2\frac{j_s(j_s+1)}{2}}
x^{j_sm_sd}
g^{-j_sm_s}]\gamma^{-m_i^2(j_i+1)}g^{-m_i}y_{m_i}u_j\\
&= -\xi_{m_i}^{-1}x^{m_id}hg^{m-1}\prod_{s=1}^{\theta}[(-1)^{j_s}\xi_{m_s}^{-j_s}\gamma^{-
m_s^2\frac{j_s(j_s+1)}{2}}
x^{j_sm_sd}
g^{-j_sm_s}]\gamma^{-m_i^2(j_i+1)}g^{-m_i}\phi_{m_i,j}u_{j+m_i}
\end{align*} with $$S(y_{m_i}u_j)=S(\phi_{m_i,j} u_{j+m_i})=S(u_{j+m_i})S(\phi_{m_i,j})=\phi_{m_i,j} S(u_{j+m_i}),$$
we find that $$S(u_{j+m_i})=hg^{m-1}\prod_{s=1}^{\theta}[(-1)^{(j+m_i)_s}\xi_{m_s}^{-(j+m_i)_s}\gamma^{-
m_s^2\frac{(j+m_i)_s((j+m_i)_s+1)}{2}}
x^{(j+m_i)_sm_sd}
g^{-(j+m_i)_sm_s}]u_j.$$

In order to determine the relationship between $\xi$ and $\gamma$,
we consider the equality $(\Id*S)(u_{m_i})=0$. By computation,
\begin{align*}
&(\Id*S)(u_{m_i})\\
&=\sum_{j=0}^{m-1}\gamma^{j(m_i-j)}u_jS(x^{-jd}g^ju_{m_i-j}) \\
&=\sum_{j=0}^{m-1}\gamma^{j(m_i-j)}u_jhg^{m-1}
\prod_{s=1}^{\theta}
[(-1)^{(m_i-j)_s}\xi_{m_s}^{-(m_i-j)_s}\gamma^{-m_s^2\frac{(m_i-j)_s((m_i-j)_s+1)}{2}}
\\
&\quad\quad\quad x^{(m_i-j)_sm_sd}g^{-(m_i-j)_sm_s}]u_{m_i-j} g^{-j}x^{jd}\\
&=\sum_{j=0}^{m-1}\gamma^{-j(m_i-j)}\overline{h}u_jg^{m-1}
\prod_{s=1}^{\theta}
[(-1)^{(m_i-j)_s}\xi_{m_s}^{-(m_i-j)_s}\gamma^{-m_s^2\frac{(m_i-j)_s((m_i-j)_s+1)}{2}}
\\
&\quad\quad\quad x^{(m_i-j)_sm_sd}g^{-(m_i-j)_sm_s}]\gamma^{(m_i-j)j}x^{jd}g^{-j}u_{m_i-j} \\
&=\sum_{j=0}^{m-1}\overline{h}u_jg^{m-1}
\prod_{s=1}^{\theta}
[(-1)^{(m_i-j)_s}\xi_{m_s}^{-(m_i-j)_s}\gamma^{-m_s^2\frac{(m_i-j)_s((m_i-j)_s+1)}{2}}]
\\
&\quad\quad\quad x^{m_id}g^{-m_i}u_{m_i-j}\\
&=\sum_{j=0}^{m-1}
\prod_{s=1}^{\theta}
[(-1)^{(m_i-j)_s}\xi_{m_s}^{-(m_i-j)_s}\gamma^{-m_s^2\frac{(m_i-j)_s((m_i-j)_s+1)}{2}}]
\\
&\quad\quad\quad \overline{h}x^{-m_id}\gamma^{j(-1-m_i)}x^{-2(m-1-m_i)d}g^{m-1-m_i}u_ju_{m_i-j}\\
&=\sum_{j=0}^{m-1}
\prod_{s=1}^{\theta}
[(-1)^{(m_i-j)_s}\xi_{m_s}^{-(m_i-j)_s}\gamma^{-m_s^2\frac{(m_i-j)_s((m_i-j)_s+1)}{2}}]
\\
&\quad\quad\quad \gamma^{j(-1-m_i)}x^{(-2m+2+m_i)d}\overline{h}g^{m-1-m_i}u_ju_{m_i-j}\\
&=\sum_{j=0}^{m-1}
\prod_{s=1}^{\theta}
[(-1)^{(m_i-j)_s}\xi_{m_s}^{-(m_i-j)_s}\gamma^{-m_s^2\frac{(m_i-j)_s((m_i-j)_s+1)}{2}}]
\gamma^{j(-1-m_i)}x^{(-2m+2+m_i)d}\overline{h}g^{m-1-m_i}\\
&\quad\quad\quad \frac{1}{m}x^{a} \prod_{s=1}^{\theta}
(-1)^{(m_i-j)_i}\xi_{m_s}^{-(m_i-j)_s}\gamma^{m_s^2\frac{(m_i-j)_s((m_i-j)_s+1)}{2}}
[j_s,e_s-2-(m_i-j)_s]_{m_s}y_{m_i}g\\
&=\frac{1}{m}\gamma^{m_i}x^{(-2m+2+m_i)d+a}\overline{h}g^{m-m_i}y_{m_i}\sum_{j=0}^{m-1}
\prod_{s=1}^{\theta}
[\xi_{m_s}^{-2(m_i-j)_s}\gamma^{j(-1-m_i)}[j_s,e_s-2-(m_i-j)_s]_{m_s}]\\
&=\frac{1}{m}\gamma^{m_i}\xi_{m_i}^{-2}x^{(-2m+2+m_i)d+a}\overline{h}g^{m-m_i}y_{m_i}\\
&\prod_{s=1,s\neq i}^{\theta}[ \sum_{j_s=0}^{e_s-1}
\xi_{m_s}^{2j_s}\gamma^{-j_sm_s}]j_s-1,j_s-1[_{m_s}]\sum_{j_i=0}^{e_i-1}
\xi_{m_i}^{2j_i}\gamma^{-j_im_i(1+m_i)}]j_i-2,j_i-1[_{m_i}
\end{align*}

where Equation \eqref{r7} is used. By Lemma \ref{l3.5} (1), each $\sum_{j_s=0}^{e_s-1}
\xi_{m_s}^{2j_s}\gamma^{-j_sm_s}]j_s-1,j_s-1[_{m_s}\neq 0$.
Thus $$(\Id*S)(u_i)=0 \Leftrightarrow
\sum_{j_i=0}^{e_i-1}
\xi_{m_i}^{2j_i}\gamma^{-j_im_i(1+m_i)}]j_i-2,j_i-1[_{m_i}=0.$$
This forces $\xi_{m_i}^2=\gamma^{m_i}$ by Lemma \ref{l3.5} (2).

Next, we will determine the expression of $h$ and $a$ through considering the equations
$$(S*\Id)(u_0)=(\Id*S)(u_0)=1.$$ Indeed,
\begin{align*}
&(S*\Id)(u_0)\\
&=\sum_{j=0}^{m-1}S(\gamma^{-j^2}u_j)x^{-jd}g^ju_{-j} \\
           &=\sum_{j=0}^{m-1}\gamma^{-j^2} hg^{m-1}
\prod_{i=1}^{\theta}
[(-1)^{j_i}\xi_{m_i}^{-j_i}\gamma^{-m_i^2\frac{j_i(j_i+1)}{2}}
x^{j_im_id}g^{-j_im_i}]u_jx^{-jd}g^ju_{-j}\\
&=hg^{m-1}\sum_{j=0}^{m-1}\prod_{i=1}^{\theta}
[(-1)^{j_i}\xi_{m_i}^{-j_i}\gamma^{-m_i^2\frac{j_i(j_i+1)}{2}}]u_ju_{-j}\\
 &=hg^{m-1}\sum_{j=0}^{m-1}\prod_{i=1}^{\theta}
[(-1)^{j_i}\xi_{m_i}^{-j_i}\gamma^{-m_i^2\frac{j_i(j_i+1)}{2}}] \frac{1}{m}x^{a} \\ &\quad\prod_{i=1}^{\theta}
[(-1)^{(-j)_i}\xi_{m_i}^{-(-j)_i}\gamma^{m_i^2\frac{(-j)_i((-j)_i+1)}{2}}]j_i-l,j_i-1[_{m_i}]g\\
&=\frac{1}{m}x^{a}hg^{m}\sum_{j=0}^{m-1}\prod_{i=1}^{\theta}
[(-1)^{e_i}\xi_{m_i}^{-e_i}\gamma^{-m_i^2(\frac{e_i(e_i+1)}{2}-j_i)}]j_i-l,j_i-1[_{m_i}]  \\
&=\frac{1}{m}x^{a}hg^{m}(-1)^{\sum_{i=1}^{\theta}(m_i-1)(e_i+1)}
\prod_{i=1}^{\theta}\sum_{j_i=0}^{e_i-1}
\gamma^{-m_i^2j_i}]j_i-l,j_i-1[_{m_i}\\
&=\frac{1}{m}x^{a}hg^{m}(-1)^{\sum_{i=1}^{\theta}(m_i-1)(e_i+1)}
\prod_{i=1}^{\theta}e_ix^{(e_i-1)m_id} \quad (\,\textrm{by  Lemma}\ \ref{l3.4}\ (3))\\
&=(-1)^{\sum_{i=1}^{\theta}(m_i-1)(e_i+1)}x^{a+\sum_{i=1}^{\theta}(e_i-1)m_id+md}h,
\end{align*}
\begin{align*} &(\Id*S)(u_0)\\
&=\sum_{j=0}^{m-1}\gamma^{-j^2}u_j S(x^{-jd}g^ju_{-j}) \\
&=\sum_{j=0}^{m-1}\gamma^{-j^2}u_j S(u_{-j})g^{-j}x^{jd} \\
&=\sum_{j=0}^{m-1}\gamma^{-j^2}u_j hg^{m-1}
\prod_{i=1}^{\theta}
[(-1)^{(-j)_i}\xi_{m_i}^{-(-j)_i}\gamma^{-m_i^2\frac{(-j)_i((-j)_i+1)}{2}}
x^{(-j)_im_id}g^{-(-j)_im_i}]u_{-j}g^{-j}x^{jd}\\
&=\sum_{j=0}^{m-1}u_j
\prod_{i=1}^{\theta}
[(-1)^{(-j)_i}\xi_{m_i}^{-(-j)_i}\gamma^{-m_i^2\frac{(-j)_i((-j)_i+1)}{2}}
]hg^{m-1}u_{-j} \\
&= x^{(2-2m)d}\overline{h}g^{m-1}\sum_{j=0}^{m-1}
\gamma^{-j}\prod_{i=1}^{\theta}
[(-1)^{(-j)_i}\xi_{m_i}^{-(-j)_i}\gamma^{-m_i^2\frac{(-j)_i((-j)_i+1)}{2}}
]u_{j}u_{-j}\\
&=\sum_{j=0}^{m-1}u_j
\prod_{i=1}^{\theta}
[(-1)^{(-j)_i}\xi_{m_i}^{-(-j)_i}\gamma^{-m_i^2\frac{(-j)_i((-j)_i+1)}{2}}
]hg^{m-1}u_{-j} \\
&= x^{(2-2m)d}\overline{h}g^{m-1}\sum_{j=0}^{m-1}
\gamma^{-j}\prod_{i=1}^{\theta}
[(-1)^{(-j)_i}\xi_{m_i}^{-(-j)_i}\gamma^{-m_i^2\frac{(-j)_i((-j)_i+1)}{2}}]\\
&\quad \frac{1}{m}x^{a} \prod_{i=1}^{\theta}
[(-1)^{(-j)_i}\xi_{m_i}^{-(-j)_i}\gamma^{m_i^2\frac{(-j)_i((-j)_i+1)}{2}}]j_i-l,j_i-1[_{m_i}]g\\
&=\frac{1}{m}x^{(2-m)d+a}\overline{h}\sum_{j=0}^{m-1}
\gamma^{-j}\prod_{i=1}^{\theta}
\xi_{m_i}^{-2(-j)_i}]j_i-l,j_i-1[_{m_i}]\\
&=x^{(2-m)d+a}\overline{h} \quad (\,\textrm{by  Lemma}\ \ref{l3.4}\ (1)).
\end{align*}
So, $(S*\Id)(u_0)=(\Id*S)(u_0)=1$ implies
$h=x^{-a-\sum_{i=1}^{\theta}(e_i-1)m_id- md}(-1)^{\sum_{i=1}^{\theta}(m_i-1)(e_i-1)}=x^{(2-m)d+a}$. Thus
$$a=-d-\frac{\sum_{i=1}^{\theta}(e_i-1)m_id}{2}\quad \textrm{and}\quad 2|\sum_{i=1}^{\theta}(m_i-1)(e_i-1),\;\; 2|\sum_{i=1}^{\theta}(e_i-1)m_id.$$
And $h=x^{(1-m)d-\frac{\sum_{i=1}^{\theta}(e_i-1)m_id}{2}}.$
Therefore, for $0\leqslant j\leqslant m-1$,$$S(u_j)=x^{b}g^{m-1}\prod_{i=1}^{\theta}[(-1)^{j_i}\xi_{m_i}^{-j_i}\gamma^{-
m_i^2\frac{j_i(j_i+1)}{2}}
x^{j_im_id}
g^{-j_im_i}]u_j$$ for $b=(1-m)d-\frac{\sum_{i=1}^{\theta}(e_i-1)m_i}{2}d.$  \qed

From Claim 7, we know that $a=-d-\frac{\sum_{i=1}^{\theta}(e_i-1)m_id}{2}$ and we can improve Claim 6 as the following form:

\noindent \emph{Claim 6'.  For $0\leqslant j,l\leqslant m-1$, the multiplication between $u_j$ and $u_l$ satisfies that $$u_ju_l=\frac{1}{m}x^{-d-\frac{\sum_{i=1}^{\theta}(e_i-1)m_id}{2}} \prod_{i=1}^{\theta}
(-1)^{l_i}\xi_{m_i}^{-l_i}\gamma^{m_i^2\frac{l_i(l_i+1)}{2}}[j_i,e_i-2-l_i]_{m_i}y_{j+l}g$$
where $j+l$ is interpreted \emph{mod} $m$.}

We can prove Proposition \ref{p6.11} now. The statement (1) is gotten from Claim 1 and the proof of Claim 7. For (2), by Claims 1,2',3,4,5,6' and 7, we have a natural surjective Hopf homomorphism
$$f:\; D(\underline{m},d,\gamma)\to H,\;x\mapsto x,\ y_{m_i}\mapsto y_{m_i},\ g\mapsto g,\ u_{j}\mapsto u_{j}$$
for $1\leq i\leq \theta$ and $0\leq j\leq m-1$. It is not hard to see that $f|_{D_{st}}:\; D_{st}\to H_{st}$ is an isomorphism of
$\k[x^{\pm 1}]$-modules for $0\leqslant s,t\leqslant 2m-1$. So $f$ is an isomorphism. \qed

\section{Main result and consequences}
We conclude this paper by giving the classification of  prime Hopf algebras of GK-dimension one satisfying (Hyp1), (Hyp2) and some consequences.
\subsection{Main result.} The main result of this paper can be stated as follows.
\begin{theorem}\label{t7.1} Let $H$ be a prime Hopf algebra of GK-dimension one which satisfies (Hyp1) and (Hyp2). Then $H$ is isomorphic to one of Hopf algebras constructed in Section 4.
\end{theorem}
\begin{proof} Let $\pi:\;H\to \k$ be the canonical 1-dimensional representation of $H$ which exits by (Hyp1). If PI-deg$(H)=1$, then it is  easy to see that $H$ is commutative and thus $H\cong \k[x]$ or $\k[x^{\pm 1}]$. So, we assume that $n:=$ PI-deg$(H)>1$ in the following analysis. If $\mi (\pi)=1$, then $H$ is isomorphic to either a $T(n,0,\xi)$ or $\k\mathbb{D}$ by Proposition \ref{p5.4}. If $\ord(\pi)=\mi(\pi)$, then $H$ is isomorphic to either a $T(\underline{n},1,\xi)$ or a $B(\underline{n},\omega,\gamma)$ by Proposition \ref{p5.7}. The last case is $n=\ord(\pi)>m:=\mi(\pi)>1$. In such case, using Corollary \ref{c6.3}, $H$ is either additive or multiplicative. If, moreover, $H$ is additive then $H$ is isomorphic to  a $T(\underline{m},t,\xi)$ by Proposition \ref{p6.9} for $t=\frac{n}{m}$ and if $H$ is multiplicative then it is isomorphic to a $D(\underline{m},d,\gamma)$ by Proposition \ref{p6.11}.
\end{proof}

\begin{remark} \emph{(1) All prime Hopf algebras of GK-dimension one which are regular are special cases of their fraction versions. For example, the infinite dimensional Taft algebra $T(n,t,\xi)$ is isomorphic to $T(\underline{n},t,\xi)$ where $\underline{n}=\{1\}$ is a fraction of $n$ of length $1$ (that is, $\theta=1$ by previous notation).}

\emph{(2) By Proposition \ref{p4.8}, we know that $D(\underline{m},d,\gamma)$ is not a pointed Hopf algebra if $m\neq 1$. Thus we get more examples of non-pointed Hopf algebras of GK-dimension one.}

\emph{(3) In \cite[Question 7.3C.]{BZ}, the authors asked that what other Hopf algebras can be included if the regularity hypothesis is dropped. So our result gives many this kind of Hopf algebras. }
\end{remark}
\subsection{Question \eqref{1.1}.}
As an application, we can give the answer to question \eqref{1.1} now. We give the following definition at first.
\begin{definition} \emph{We call an irreducible algebraic curve $C$} a fraction line \emph{if there is a natural number $m$ and a fraction $m_1,\ldots, m_\theta$ of $m$ such that it's coordinate algebra $\k[C]$ is isomorphic to $\k[y_{m_1},\ldots,y_{m_{\theta}}]/(y_{m_i}^{e_i}-y_{m_j}^{e_j},\;1\leq i\neq j\leq \theta).$}
\end{definition}

The answer to question \eqref{1.1} is given as follows.

\begin{proposition} Assume $C$ is an irreducible algebraic curve over $\k$ which can be realized as a Hopf algebra in ${^{\Z_{n}}_{\Z_n}\mathcal{YD}}$ where $n$ is as small as possible. Then $C$ is either an algebraic group or a fraction line.
\end{proposition}
\begin{proof} If $n=1$, then $\k[C]$ is a Hopf algebra and thus $C$ is an algebraic group of dimension one. Now assume $n>1$. By assumption, $\Z_n$ acts on $\k[C]$ faithfully. Using Lemma \ref{l2.10} and the argument developed in the proof of Corollary \ref{c2.9},  the Hopf algebra $\k[C]\# \k\Z_n$ (the Radford's biproduct) is a prime Hopf algebra of GK-dimension one with PI-degree $n$. It is known that $\k\Z_n$ has a 1-dimensional representation of order $n$:
$$\k\Z_n=\k\langle g|g^n=1\rangle \To \k,\;\;\;\;g\mapsto \xi$$
for a primitive $n$th root of unity $\xi$. Through the canonical projection $\k[C]\# \k\Z_n\to \k\Z_n$ we get a 1-dimensional representation $\pi$ of $H:=\k[C]\# \k\Z_n$ of order $n=$PI-deg$(H)$. Therefore, $H$ satisfies (Hyp1). Also, by the definition of the Radford's biproduct we know that  the right invariant component $H_{0}^{r}$ of $\pi$ is exactly the domain $\k[C]$. Therefore, $H$ satisfies (Hyp2) too. The classification result, that is Theorem \ref{t7.1}, can be applied now. One can check the proposition case by case.
\end{proof}
\subsection{Finite-dimensional quotients.} We realize that from the Hopf algebra $D(\underline{m},d,\gamma)$ we can get many new finite-dimensional Hopf algebras through quotient method. Among of them, two kinds of Hopf algebras are particularly interesting for us: one series are semisimple and another series are nonsemisimple. As a byproduct of these new examples, we can give an answer to a professor Siu-Hung Ng's question at least. We will give and analysis the structures and representation theory of these two kinds of finite-dimensional Hopf algebras.

$\bullet$ \emph{The series of semisimple Hopf algebras.} Keep the notations used in Section 4 and let $D=D(\underline{m},d,\gamma)$ where $\underline{m}=\{m_1,\ldots,m_\theta\}$ a fraction of $m$. For simple, we assume that $(m_1,m_2,\ldots,m_\theta)=1$. Consider the quotient Hopf algebra
$$\overline{D}:=D/(y_{m_1},\ldots,y_{m_\theta}).$$
We want to give the generators, relations and operations for $\overline{D}$ at first. For notational convenience, the images of $x,g,u_{j}$ in $\overline{D}$ are still
written as $x,g$ and $u_j$ respectively. By the definition of $D$, we see that:
 As an algebra, $\overline{D}=\overline{D}(\underline{m},d,\gamma)$ is generated by $x^{\pm 1}, g^{\pm 1}, u_0, u_1, \cdots,
u_{m-1}$, subject to the following relations
\begin{eqnarray}
\notag&&xx^{-1}=x^{-1}x=1,\quad gg^{-1}=g^{-1}g=1,\quad xg=gx,\\
&& 0=1-x^{e_im_id},\quad g^{m}=x^{md},\label{eq7.1}\\
\notag&&xu_j=u_jx^{-1},\quad 0=\phi_{m_i,j}u_{j+m_i},\quad u_j g=\gamma^j x^{-2d}gu_j,\\
\notag &&u_ju_l=\left \{
\begin{array}{ll} \frac{1}{m}x^{a}\prod_{i=1}^{\theta}(-1)^{l_i}
\gamma^{\frac{l_i(l_i+1)}{2}}\xi_{m_i}^{-l_i}[j_i,e_i-2-l_i]_{m_i}g, &  \ j+l\equiv 0\ (\textrm{mod} \ m),\\
0, & \text{otherwise,}
\end{array}\right.
\end{eqnarray}
for $1\leq i\leq \theta,\;0\leq j,l\leq m-1$ and   $a=-\frac{2+\sum_{i=1}^{\theta}(e_i-1)m_i}{2}d$.

The coproduct $\D$, the counit $\epsilon$ and the antipode $S$ of $\overline{D}(\underline{m},d,\gamma)$ are given by
\begin{eqnarray*}
&&\D(x)=x\otimes x,\;\; \D(g)=g\otimes g,\\
&&\D(u_j)=\sum_{k=0}^{m-1}\gamma^{k(j-k)}u_k\otimes x^{-kd}g^ku_{j-k};\\
&&\epsilon(x)=\epsilon(g)=\epsilon(u_0)=1,\;\;\epsilon(u_s)=0;\\
&&S(x)=x^{-1},\;\; S(g)=g^{-1}, \\
&& S(u_j)=x^{b}g^{m-1}\prod_{i=1}^{\theta}(-1)^{j_i}\xi_{m_i}^{-j_i}\gamma^{-
m_i^2\frac{j_i(j_i+1)}{2}}
x^{j_im_id}
g^{-j_im_i}u_j,
\end{eqnarray*}
for $1\leq s\leq m-1\;,0\leq j\leq m-1$ and $b=(1-m)d-\frac{\sum_{i=1}^{\theta}(e_i-1)m_i}{2}d$.
As an observation, we find that

\begin{lemma}\label{l7.5} The Hopf algebra $\overline{D}$ is a semisimple Hopf algebra of dimension $2m^2d.$
\end{lemma}
\begin{proof} Before the proof, we want to simplify a relation given in the definition of $\overline{D}.$ That is, a relation formulated in \eqref{eq7.1}: $x^{e_im_id}=1$ for all $1\leq i\leq \theta.$ We claim that it it equivalent to the following relation
\begin{equation}\label{eq7.2} x^{md}=1.
\end{equation}
Clearly, it is enough to show that \eqref{eq7.1} implies \eqref{eq7.2} since by definition $m|e_im_i$ for all $1\leq i\leq \theta.$ Indeed, by (3) of the definition of a fraction \ref{d3.1}, $e_i|m$ and thus we know that $(\frac{e_1m_1}{m},\frac{e_2m_2}{m},\ldots,\frac{e_\theta m_\theta}{m})=1$ since we already assume that $(m_1,m_2,\ldots,m_\theta)=1.$ Therefore, there exist $s_i\in\Z$ such that $\sum_{i=1}^{\theta}s_i\frac{e_im_i}{m}=1$ and thus
$$x^{md}=x^{md\sum_{i=1}^{\theta}s_i\frac{e_im_i}{m}}=x^{\sum_{i=1}^{\theta}s_i e_im_id}=1.$$
By \eqref{eq7.2}, we further get $g^{m}=1$ since $g^{m}=x^{md}.$

We use the classical Maschke Theorem to show that $\overline{D}$ is semisimple. To do that, we construct the left integral of $\overline{D}$ as follows:
$$\int_{\overline{D}}^l:=\sum_{i=0}^{md-1}\sum_{j=0}^{m-1}x^ig^j+\sum_{i=0}^{md-1}\sum_{j=0}^{m-1}x^ig^ju_0.$$
Let's show that it is really a left integral. Indeed, it is not hard to see that $x\int_{\overline{D}}^l=g\int_{\overline{D}}^l=\int_{\overline{D}}^l$ and
\begin{align*}u_0\cdot \int_{\overline{D}}^l&=
\sum_{i=0}^{md-1}\sum_{j=0}^{m-1}x^ig^ju_0+\sum_{i=0}^{md-1}\sum_{j=0}^{m-1}x^ig^ju_0^{2}\\
&=\sum_{i=0}^{md-1}\sum_{j=0}^{m-1}x^ig^ju_0+\sum_{i=0}^{md-1}\sum_{j=0}^{m-1}x^ig^j\\
&=\epsilon(u_0)\cdot \int_{\overline{D}}^l.
\end{align*}
Now by the relation $0=\phi_{m_i,j}u_{j+m_i}$, \begin{equation}\label{eq7.3}
u_{j+m_i}=\gamma_i^{1+j_i}x^{m_id}u_{j+m_i}\end{equation}
for $0\leq j\leq m-1$ and $\gamma_i=\gamma^{-m_i^2}.$ So for any $1\leq s\leq m-1$, there must exit an $1\leq i\leq \theta$ such that $s_i\neq 0.$ From \eqref{eq7.3}, we have $u_s=\gamma_{i}^{s_i}x^{m_id}u_s$ and thus
\begin{align*}u_s\cdot \int_{\overline{D}}^l&=
\sum_{i=0}^{md-1}\sum_{j=0}^{m-1}\gamma^{sj}x^ig^ju_s\\
&=\sum_{i=0}^{md-1}\sum_{j=0}^{m-1}\gamma^{sj}x^ig^j\gamma_{i}^{s_i}x^{d}u_s\\
&=\gamma_{i}^{s_i}\sum_{i=0}^{md-1}\sum_{j=0}^{m-1}\gamma^{sj}x^ig^ju_s,
\end{align*}
which implies that $\sum\limits_{i=0}^{md-1}\sum\limits_{j=0}^{m-1}\gamma^{sj}x^ig^ju_s=0$, and so
$u_s\cdot \int_{\overline{D}}^l=0=\epsilon(u_{s})\int_{\overline{D}}^l$
for all $1 \leqslant s \leqslant m-1$. Combining above equations together, $\int_{\overline{D}}^l$ is a left integral of $\overline{D}.$
Clearly,  $\epsilon(\int_{\overline{D}}^l)=2m^{2}d\neq
0$. So $\overline{D}$ is semisimple.

At last, we want to determine the dimension of this semisimple Hopf algebra. The main idea is to apply the bigrading \eqref{eqD} and \eqref{eq4.24} of $D$ to $\overline{D}.$ To apply them, we need determine the dimension of space spanned by $\{x^{t}u_j|0\leq t\leq md-1\}$ for any $0\leq j\leq m-1$. To do that, we want give an equivalent form of the \eqref{eq7.3}. By \eqref{eq7.3},
$x^{m_id}u_{k}=\gamma^{m_ik}u_k$ for any $0\leq k\leq m-1$. Note that $(m_1,\ldots,m_{\theta})=1$ and thus we have $s_i\in \Z$ such that $\sum_{i=1}^{\theta}s_im_i=1$. Therefore
 \begin{equation}\label{e7.4} x^du_{k}=x^{\sum_{i=0}^{\theta}s_im_id}u_{k}=\gamma^{\sum_{i=0}^{\theta}s_im_ik}u_{k}=\gamma^{k}u_{k}.
 \end{equation}By this formula, the space spanned by $\{x^{t}u_j|0\leq t\leq md-1\}$ is the same as the space spanned by $\{x^{t}u_j|0\leq t\leq d-1\}$ and its dimension is $d.$ Now applying the bigrading \eqref{eqD} and \eqref{eq4.24}, we see that the set
$$\{x^ig^j, x^{t}g^{j}u_s|0\leq i\leq md-1, 0\leq j\leq m-1, 0\leq t\leq d-1, 0\leq s\leq m-1\}$$ is a basis of $\overline{D}$ and thus
$$\dim_{\k} \overline{D}=2m^2d.$$
\end{proof}

Next we want to analysis the coalgebra and algebra structure of this semisimple Hopf algebra. Its coalgebra structure can be determined easily.

\begin{proposition}\label{p7.6} Keep above notations.
\begin{itemize}\item[(1)] Let $C$ be the subspace spanned by $\{g^iu_j|0\leq i,j\leq m-1\}$. Then $C$ is simple coalgebra.
\item[(2)] The following is the decomposition of $\overline{D}$ into simple coalgbras
$$\overline{D}=\bigoplus_{i=0}^{md-1}\bigoplus_{j=0}^{m-1}\k x^ig^j\oplus \bigoplus_{i=0}^{d-1}x^iC.$$
\item[(3)] Up to isomorphisms of comodules, $\overline{D}$ has $m^2d$-number of 1-dimensional comodules and $d$-number of $m$-dimensional simple comodules.
\end{itemize}
\end{proposition}
\begin{proof} (1) One can apply similar method used in \cite{Wu} to prove this statement. For completeness, we write the details out. Clearly, to show the result, it is sufficient to show that the $\k$-linear dual $C^*:=\Hom_\k(C, \k)$ is a simple algebra. In fact, we will see that $C^*$ is the matrix algebra of order $m$. We change the basis of $C$ for the convenience. Using relation \eqref{e7.4}, $C$ is also spanned by $\{(x^{-d}g)^iu_j|0\leq i,j\leq m-1\}$. Denote by $f_{ij}:=((x^{-d}g)^iu_j)^*$, that is, $\{f_{ij}| 0\leqslant i, j\leqslant m-1\}$ is the dual basis of the basis $\{(x^{-d}g)^iu_j|0\leq i,j\leq m-1\}$ of $C$. We prove this fact by two steps: firstly, we study the multiplication of the dual basis; secondly, we construct an algebraic isomorphism from $C^*$ to the matrix algebra of order $m$.

\noindent \emph{Step 1.}\quad Since
\begin{align*}
&(f_{i_1, j_1}*f_{i_2, j_2})((x^{-d}g)^iu_j)\\
&=m(f_{i_1, j_1}\otimes f_{i_2, j_2})(\D((x^{-d}g)^iu_j))\\
   \notag  &=m(f_{i_1, j_1}\otimes f_{i_2, j_2})(\sum_{s=0}^{m-1}\gamma^{s(j-s)}(x^{-d}g)^iu_s\otimes (x^{-d}g)^{i+s}u_{j-s} )\\
    &=\sum_{s=0}^{m-1}\gamma^{s(j-s)}f_{i_1, j_1}((x^{-d}g)^iu_s)f_{i_2, j_2}((x^{-d}g)^{i+s}u_{j-s})\\
\end{align*}
one can see that $(f_{i_1, j_1}*f_{i_2, j_2})((x^{-d}g)^iu_j)\neq 0$ if and only if $i_1=i, j_1=s, i_2=i+s$ and $j_2=j-s$ for some $0\leqslant s\leqslant m-1$. This forces $i_1+j_1=i_2, i=i_1$ and $j=j_1+j_2$. So we have
\begin{equation}\label{eq*}f_{i_1, j_1}*f_{i_2, j_2}=\begin{cases}
\gamma^{j_1j_2}f_{i_1, j_1+j_2}, & \textrm{if}\; i_1+j_1=i_2,\\
0, & \textrm{otherwise}.
\end{cases}\end{equation}

\noindent \emph{Step 2.}\quad Set $M=M_m(\k)$ and let $E_{ij}$ be the matrix units (that is, the matrix with 1 is in the (i, j)
entry and 0 elsewhere) for $0\leqslant i, j\leqslant m-1$. Now we claim that $$\varphi: C^*\to M, f_{ij}\mapsto \gamma^{ij}E_{i, i+j}$$ is an algebraic isomorphism (the index $i+j$ in $E_{i+j}$ is interpreted mod $m$).
It is sufficient to verify that $\varphi$ is an algebraic map. In fact,
\begin{align*}\varphi(f_{i_1, j_1})\varphi(f_{i_2, j_2})&=\gamma^{i_1j_1}E_{i_1, i_1+j_1}\gamma^{i_2j_2}E_{i_2, i_2+j_2}\\
&=\begin{cases}
\gamma^{i_1j_1+i_2j_2}E_{i_1, i_2+j_2}, & \textrm{if}\; i_1+j_1=i_2,\\
0, & \textrm{otherwise},
\end{cases}\\
&=\begin{cases}
\gamma^{i_1j_1+i_2j_2-i_1(j_1+j_2)}\varphi(f_{i_1, j_1+j_2}), & \textrm{if}\; i_1+j_1=i_2,\\
0, & \textrm{otherwise},
\end{cases}\\
&=\begin{cases}
\varphi(f_{i_1, j_1+j_2}), & \textrm{if}\; i_1+j_1=i_2,\\
0, & \textrm{otherwise},
\end{cases}\\
&=\varphi(f_{i_1, j_1}*f_{i_2, j_2}).
\end{align*}
So $\varphi$ is an algebraic map and the proof is completed.

(2) Comparing the dimensions of left side and right side, we have the statement.

(3) This is a direct consequence of (2).
\end{proof}

Next, we want to determine the algebraic structure of $\overline{D}$. As in the proof of Lemma \ref{l7.5}, $\{x^ig^j, x^tg^ju_s|0\leq i\leq md-1, 0\leq j\leq m-1, 0\leq t\leq d-1, 0\leq s\leq m-1\}$ is a basis of $\overline{D}$. Denote by $G$ the group of all group-likes of $\overline{D}$. Then clearly every element in $\overline{D}$ can be written uniquely in the following way:
$$f+\sum_{i=0}^{m-1}f_iu_i$$
for $f, f_{i}\in \k G$ and $0\leq i\leq m-1$.

We use $C(\overline{D})$ to denote the center of $\overline{D}$. Next result helps us to determine the center of $\overline{D}$.

\begin{lemma}\label{l7.7} The element $e=f+\sum_{i=0}^{m-1}f_iu_i\in C(\overline{D})$ if and only if $f, f_0u_0\in C(\overline{D})$ and $f_1=\ldots=f_{m-1}=0$.
\end{lemma}
\begin{proof} The sufficiency is obvious. We just prove the necessity. At first, we show that
$f_1=\ldots=f_{m-1}=0$. Otherwise, assume that, say, $f_1\neq 0$. By assumption, $ge=eg$ which implies that $gf_1u_1=f_1u_1g$. By the definition of $\overline{D}$, $f_1u_1g=\gamma^{-1}gf_1u_1$. So we have $\gamma^{-1}gf_1u_1=gf_1u_1$ which is absurd. Similarly, we have $f_2=\ldots=f_{m-1}=0$. Secondly, let's show that $f\in C(\overline{D})$. Also, by $eu_i=u_ie$ we know that $fu_i=u_if$ for $0\leq i\leq m-1$. By definition, $f$ commutes with all elements in $G$. Therefore, $f\in C(\overline{D})$. Since $e=f+f_0u_0$ and $e\in C(\overline{D})$, $f_0u_0\in C(\overline{D})$ too.
\end{proof}

Let $\zeta$ be an $md$th root of unity satisfying $\zeta^{d}=\gamma$. Define
$$1_i^{x}:=\frac{1}{md}\sum_{j=0}^{md-1}\zeta^{-ij}x^j,\ \ 1_k^{g}:=\frac{1}{m}\sum_{j=0}^{m-1}\gamma^{-kj}g^j$$
for $0\leq i\leq md-1$ and $0\leq k\leq m-1$. It is well-known that $\{1_i^{x}1_k^{g}|0\leq i\leq md-1, 0\leq k\leq m-1\}$ is also a basis of $\k G$. Therefore, one can assume that
$$f=\sum_{i=0}^{md-1}\sum_{j=0}^{m-1}a_{ij}1_i^{x}1_j^{g}=\sum_{i,j}a_{ij}1_i^{x}1_j^{g}.$$
For any natural number $i$, we use $i'$ to denote the remainder of $i$ divided by $m$ in the following of this subsection.

\begin{lemma}\label{l7.8} Let $f=\sum_{i,j}a_{ij}1_i^{x}1_j^{g}$ be an element in $\k G$. Then $f\in C(\overline{D})$ if and only if $a_{ij}=a_{md-i,j-i'}$ for all $0\leq i\leq md-1, 0\leq j\leq m-1$.
\end{lemma}
\begin{proof} Define $$\unit_{i}^x:=\frac{1}{d}(1+\zeta^{-i} x+\zeta^{-2i} x^2+\ldots+\zeta^{-(d-1)i} x^{d-1})$$
for $0\leq i\leq md-1$. For any $0\leq k\leq m-1$, it is not hard to see that the elements in $\{\unit_{i}^{x}|i\equiv k\ (\textrm{mod}\ m)\}$ are linear independent. Using equation \eqref{e7.4} and a direct computation, one can show that \begin{equation}\label{eq7.5}1_i^xu_k=\left \{
\begin{array}{ll} \unit_i^{x}u_k, & \;\;\text{if}\ \ i\equiv k\ (\textrm{mod}\ m),\\
0, & \;\;\text{otherwise,}
\end{array}\right.\end{equation}
\begin{equation}u_k1_i^x=\left \{
\begin{array}{ll} u_k\unit_i^{x}, & \;\;\text{if}\ \ i+k\equiv 0\ (\textrm{mod}\ m),\\
0, & \;\;\text{otherwise.}
\end{array}\right.\end{equation}
and
$$1_{md-i}^{x}=1_{i}^{x^{-1}}.$$
Therefore, we have
$$
fu_k=\sum_{i,j}a_{ij}1_{i}^{x}1_{j}^{g}u_k=\sum_{i,j}a_{ij}1_{i}^{x}u_k1_{j-k}^{g}=
\sum_{i\equiv k\;(\textrm{mod}\;m),j}a_{ij}\unit_{i}^{x}1_{j}^{g}u_k,
$$
\begin{eqnarray*}u_kf&&=\sum_{i,j}a_{ij}u_k1_{i}^{x}1_{j}^{g}=\sum_{i+k\equiv 0\;(\textrm{mod}\;m),j}a_{ij}u_{k}1_{i}^{x}1_{j}^{g}=\sum_{i+k\equiv 0\;(\textrm{mod}\;m),j}a_{ij}1_{i}^{x^{-1}}u_k1_{j}^{g}\\
&&=\sum_{i+k\equiv 0\;(\textrm{mod}\;m),j}a_{ij}1_{md-i}^{x}u_k1_{j}^{g}=\sum_{i+k\equiv 0\;(\textrm{mod}\;m),j}a_{ij}\unit_{md-i}^{x}1_{j+k}^{g}u_k.
\end{eqnarray*}
This means that $fu_k=u_kf$ if and only if $a_{k+lm,j}=a_{m(d-l)-k,j-k}$ for some $0\leq k\leq m-1$. From this, the proof is done.
\end{proof}

Assume that $f_0=\sum_{i,j}b_{ij}1_i^x1_j^g$. Using \eqref{eq7.5}, we know that
$$f_0u_0=\sum_{i\equiv 0\;(\textrm{mod}\;m),j}b_{ij}\unit_i^x1_j^gu_0.$$
So we can assume that $f_0=\sum_{i\equiv 0\;(\textrm{mod}\;m),j}b_{ij}1_i^x1_j^g$ directly. With this assumption, we have the following result.
\begin{lemma} The element $f_0u_0$ belongs to the center of $\overline{D}$ if and only if
\begin{equation}\label{eq7.7} f_0=\left \{
\begin{array}{ll} \sum_{j}b_{0j}1_0^{x}1_j^g, & \;\;\emph{if}\ \ d\ \emph{is\ odd}, \\
\sum_{j}b_{0j}1_0^{x}1_j^g+\sum_{j}b_{\frac{d}{2}m,j}1_{\frac{d}{2}m}^{x}1_j^g, & \;\;\emph{if}\ \ d\ \emph{is\ even}.\end{array}\right.
\end{equation}
\end{lemma}
\begin{proof} From $xf_0u_0=f_0u_0x$, we have $xf_0=x^{-1}f_0$, which implies exactly the equation \eqref{eq7.7}. The converse is straightforward.
\end{proof}

Next, we want determine when a central element is idempotent.
\begin{lemma}\label{l7.9} Let $e=f+f_0u_0$ be an element living in the center $C(\overline{D})$. Then $e^2=e$ if and only if $f=f^2+f_0^2u_0^2$ and $f_0=2ff_0$.\end{lemma}
\begin{proof} By Lemma \ref{l7.7}, $f$ commutes with $f_0u_0$ and $(f_0u_0)^2=f_0^2u_0^2$. From this, the lemma becomes clear.
\end{proof}

\begin{lemma}\label{l7.11} Let $f=\sum_{i,j}a_{ij}1_i^x1_j^g$ and $f_0=\sum_{i,j}b_{ij}1_i^x1_j^g$ satisfying $e=f+f_0u_0$ is a central element. Then $e$ is an idempotent if and only if
\begin{eqnarray} a_{sm,j}&=&a_{sm,j}^2+b_{sm,j}^2\zeta^{asm}\gamma^j\;\;\;\;(0\leq s\leq d-1, \ 0\leq j\leq m-1)\label{eq7.8}\\
a_{ij}^2&=&a_{ij}\;\;\;\;(i\not\equiv 0\ (\emph{mod} \ m), \ 0\leq j\leq m-1)\\
b_{ij}&=&2a_{ij}b_{ij}\;\;\;\;(0\leq i\leq md-1,\ 0\leq j\leq m-1).\label{eq7.10}
\end{eqnarray}
where $a=-\frac{2+\sum_{i=1}^{\theta}(e_i-1)m_i}{2}d.$
\end{lemma}
\begin{proof} We just translate the equivalent conditions in Lemma \ref{l7.9} into the equalities about coefficients.
\end{proof}
By equation \eqref{eq7.10}, we know that $a_{ij}=\frac{1}{2}$ if $b_{ij}\neq 0$. By equation \eqref{eq7.8}, we have that $b_{sm,j}=\pm \frac{1}{2}\sqrt{\gamma^{-j}\zeta^{-asm}}$ if $b_{sm,j}\neq 0$. We use $[.]$ to denote the floor function, i.e. for any rational number $t$, $[t]$ is the biggest integer which is not bigger than $t$. Now we can give the algebraic structure of $\overline{D}$.

\begin{proposition}\label{p7.12} Keep above notations.
\begin{itemize}\item[(1)] If $d$ is even, then the following is a complete set of primitive central idempotents of $\overline{D}$:
\begin{eqnarray*}
&&\frac{1}{2}1_0^x1_j^g+ \frac{1}{2}\sqrt{\gamma^{-j}}1_0^x1_j^gu_0,\ \
  \frac{1}{2}1_0^x1_j^g- \frac{1}{2}\sqrt{\gamma^{-j}}1_0^x1_j^gu_0,\\
&& \frac{1}{2}1_{\frac{d}{2}m}^x1_j^g+ \frac{1}{2}\sqrt{\gamma^{-j}(-1)^{-a}}1_{\frac{d}{2}m}^x1_j^gu_0,\ \
  \frac{1}{2}1_{\frac{d}{2}m}^x1_j^g- \frac{1}{2}\sqrt{\gamma^{-j}(-1)^{-a}}1_{\frac{d}{2}m}^x1_j^gu_0,\\
&& 1_{sm}^x1_j^g+1_{(d-s)m}^x1_j^g,\ \ \ \ (0< s\leq d-1, s\neq \frac{d}{2},\ 0\leq j\leq m-1)\\
&& 1_{lm+i}^x1_j^g+1_{(d-l-1)m+(m-i)}^x1_{j-i}^g,\ \ \ \ (0\leq l\leq d-1, 0< i\leq [\frac{m}{2}],\ 0\leq j\leq m-1).
\end{eqnarray*}
If $d$ is odd, then the following is a complete set of primitive central idempotents of $\overline{D}$:
\begin{eqnarray*}
&&\frac{1}{2}1_0^x1_j^g+ \frac{1}{2}\sqrt{\gamma^{-j}}1_0^x1_j^gu_0,\ \
  \frac{1}{2}1_0^x1_j^g- \frac{1}{2}\sqrt{\gamma^{-j}}1_0^x1_j^gu_0,\\
&& 1_{sm}^x1_j^g+1_{(d-s)m}^x1_j^g,\ \ \ \ (0< s\leq d-1, \ 0\leq j\leq m-1)\\
&& 1_{lm+i}^x1_j^g+1_{(d-l-1)m+(m-i)}^x1_{j-i}^g,\ \ \ \ (0\leq l\leq d-1, 0< i\leq [\frac{m}{2}],\ 0\leq j\leq m-1).
\end{eqnarray*}
\item[(2)] If $d$ is even, then as an algebra $\overline{D}$ has the following decomposition:
$$\overline{D}= \k^{(4m)}\oplus M_{2}(\k)^{(\frac{m^2d-2m}{2})}.$$
If $d$ is odd, then as an algebra $\overline{D}$ has the following decomposition:
$$\overline{D}= \k^{(2m)}\oplus M_{2}(\k)^{(\frac{m^2d-m}{2})}.$$
\end{itemize}
\end{proposition}
\begin{proof} (1) According to Lemmas \ref{l7.8}-\ref{l7.11}, we know all above elements are central idempotents. It is easy to find that the sum of these elements is just $1$. So to show the result, it is enough to show that they are all primitive central idempotents. We just prove this fact for the case $d$ even since the other case can be proved in the same way. In fact, by definition we can find the elements in the last two lines presented in this proposition can be decomposed into a sum of two idempotents which are not central, and so the simple modules corresponding to these central idempotents have dimension $\geq 2$. There are $\frac{(d-2)m}{2}+\frac{(m-1)dm}{2}$ cental idempotents in the last two lines and $4m$ ones in the first two lines. Therefore, all of these idempotents create an ideal with dimension $\geq 4m+4(\frac{(d-2)m}{2}+\frac{(m-1)dm}{2})=2m^2d=\dim_{\k} \overline{D}$. This implies they are all primitive.

(2) This is just a direct consequence of the statement (1).
\end{proof}

Due to our recent great interest on finite tensor categories \cite{EGNO}, in particular fusion categories \cite{ENO}, it seems better to present all simple modules of $\overline{D}$ and their tensor product decomposition law here.

As the proof of above proposition, we only deal with the case $d$ being even (actually, the case of $d$ being odd is quite similar and in fact easier). According to the central idempotents stated in Proposition \ref{p7.12} (1), we construct the following six kinds of simple modules of $\overline{D}$:
\begin{itemize}
\item[(1)] $V_{0,j}^{+}\;(0\leq j\leq m-1)$: The dimension of $V_{0,j}^{+}$ is $1$ and the action of $\overline{D}$ is given by
    \begin{align*}& x\mapsto 1,\quad\quad\quad\quad g\mapsto \gamma^{j}\\
    & u_{0}\mapsto \sqrt{\gamma^j}, \quad\quad u_{i}\mapsto 0 \;(1\leq i\leq m-1).
    \end{align*}
   A basis of this module can be chosen as $\frac{1}{2}1_0^x1_j^g+ \frac{1}{2}\sqrt{\gamma^{-j}}1_0^x1_j^gu_0.$\\
\item[(2)] $V_{0,j}^{-}\;(0\leq j\leq m-1)$: The dimension of $V_{0,j}^{-}$ is $1$ and the action of $\overline{D}$ is given by
    \begin{align*}& x\mapsto 1,\quad\quad\quad\quad\quad g\mapsto \gamma^{j}\\
    & u_{0}\mapsto -\sqrt{\gamma^j}, \quad\quad u_{i}\mapsto 0 \;(1\leq i\leq m-1).
    \end{align*}
    A basis of this module can be chosen as $\frac{1}{2}1_0^x1_j^g- \frac{1}{2}\sqrt{\gamma^{-j}}1_0^x1_j^gu_0.$\\
\item[(3)] $V_{\frac{d}{2}m,j}^{+}\;(0\leq j\leq m-1)$: The dimension of $V_{\frac{d}{2}m,j}^{+}$ is $1$ and the action of $\overline{D}$ is given by
    \begin{align*}& x\mapsto -1,\quad\quad\quad\quad\quad\quad\quad g\mapsto \gamma^{j}\\
    & u_{0}\mapsto \sqrt{\gamma^j(-1)^{-a}}, \quad\quad u_{i}\mapsto 0 \;(1\leq i\leq m-1).
    \end{align*}
   A basis of this module can be chosen as $\frac{1}{2}1_{\frac{d}{2}m}^x1_j^g+ \frac{1}{2}\sqrt{\gamma^{-j}(-1)^{-a}}1_{\frac{d}{2}m}^x1_j^gu_0.$ Recall that by definition $a=-\frac{2+\sum_{i=1}^{\theta}(e_i-1)m_i}{2}d.$\\
\item[(4)] $V_{\frac{d}{2}m,j}^{-}\;(0\leq j\leq m-1)$: The dimension of $V_{\frac{d}{2}m,j}^{-}$ is $1$ and the action of $\overline{D}$ is given by
    \begin{align*}& x\mapsto -1,\quad\quad\quad\quad\quad\quad\quad g\mapsto \gamma^{j}\\
    & u_{0}\mapsto -\sqrt{\gamma^j(-1)^{-a}}, \quad\quad u_{i}\mapsto 0 \;(1\leq i\leq m-1).
    \end{align*}
    A basis of this module can be chosen as $\frac{1}{2}1_{\frac{d}{2}m}^x1_j^g- \frac{1}{2}\sqrt{\gamma^{-j}(-1)^{-a}}1_{\frac{d}{2}m}^x1_j^gu_0.$\\
\item[(5)] $V_{sm,j}\;(0<s\leq d-1,\;s\neq \frac{d}{2},\; 0\leq j\leq m-1)$: The dimension of $V_{sm,j}$ is $2$ with basis $\{1_{sm}^x1_{j}^g, 1_{(d-s)m}^x1_{j}^gu_0\}$ and the action of $\overline{D}$ is given by
    \begin{align*}& x\mapsto \left (
\begin{array}{cc} \zeta^{sm}&0\\0&\zeta^{(d-s)m}
\end{array}\right),\quad\quad\quad\quad g\mapsto \left (
\begin{array}{cc} \gamma^{j}&0\\0&\gamma^{j}
\end{array}\right)\\
    & u_{0}\mapsto \left (
\begin{array}{cc}0&\zeta^{asm}\gamma^{j}\\1&0
\end{array}\right), \quad\quad u_{i}\mapsto 0 \;(1\leq i\leq m-1).
    \end{align*}
    Note that we have $$V_{sm,j}\cong V_{(d-s)m,j}.$$\\
\item[(6)] $V_{lm+i,j}\;(0\leq l\leq d-1,\;0<i<m,\; 0\leq j\leq m-1)$: The dimension of $V_{lm+i,j}$ is $2$ with basis $\{1_{lm+i}^x1_{j}^g, 1_{(d-l-1)m+(m-i)}^x1_{j-i}^gu_{m-i}\}$ and the action of $\overline{D}$ is given by
    \begin{align*}& x\mapsto \left (
\begin{array}{cc} \zeta^{lm+i}&0\\0&\zeta^{(d-l-1)m+(m-i)}
\end{array}\right),\quad\quad\quad\quad g\mapsto \left (
\begin{array}{cc} \gamma^{j}&0\\0&\gamma^{j-i}
\end{array}\right)\\
    & u_{m-i}\mapsto \left (
\begin{array}{cc}0&0\\1&0
\end{array}\right), \quad\quad\quad\quad\quad\quad\quad\quad\quad\quad\quad u_{i}\mapsto \left (
\begin{array}{cc}0&c_i\\0&0
\end{array}\right),\\
& u_{t}\mapsto 0\;\;\;\quad\quad\quad\quad\quad\quad\quad\quad(0\leq t\leq m-1, t\neq m-i,i),
    \end{align*}
where $c_i=\frac{1}{m}\zeta^{(lm+i)a}\gamma^{j}\prod_{s=1}^{\theta}(-1)^{-i_s}\xi_{m_s}^{i_s}
\gamma^{m_s^2\frac{-i_s(-i_s+1)}{2}}[i_s,e_s-2-(m-i)_s]_{m_s}.$ Note that we have
$$V_{lm+i,j}\cong V_{(d-l-1)m+(m-i),j-i}.$$
\end{itemize}

The following table give us the tensor product decomposition law for these simple modules. We omit the proof since it is routine.

\renewcommand{\arraystretch}{2.4}
\setlength{\tabcolsep}{8pt}
\begin{center}
\begin{tabular}{c}
\hline
$\quad\quad\quad\quad \quad\quad\quad\quad \quad\quad\quad\quad  \textrm{The \;fusion\; rule \;I}\quad\quad\quad\quad \quad\quad\quad\quad\quad\quad\quad \quad\quad\quad\quad $\\
\end{tabular}
\begin{tabular}{|ccccc|ccccc|}
\hline
%1
$V^+_{0,j}$&$\otimes$&$ V^+_{0,j'}$&$=$&$V^+_{0,j+j'}$
&
$V^{-}_{0,j}$&$\otimes$&$V^+_{0,k}$&$=$&$V^-_{0,j+k}$\\
%2
$V^+_{0,j}$&$\otimes$&$ V^-_{0,k}$&$=$&$V^-_{0,j+k}$
&
$V^-_{0,j}$&$\otimes$&$V^-_{0,k}$&$=$&$V^+_{0,j+k}$\\
%3
$V^+_{0,j}$&$\otimes$&$ V^+_{\frac{d}{2}m,k}$&$=$&$V^+_{\frac{d}{2}m,j+k}$
&
$V^-_{0,j}$&$\otimes$&$V^+_{\frac{d}{2}m,k}$&$=$&$V^-_{\frac{d}{2}m,j+k}$\\
%4
$V^+_{0,j}$&$\otimes$&$ V^-_{\frac{d}{2}m,k}$&$=$&$V^-_{\frac{d}{2}m,j+k}$
&
$V^-_{0,j}$&$\otimes$&$V^-_{\frac{d}{2}m,k}$&$=$&$V^+_{\frac{d}{2}m,j+k}$\\
%5
$V^+_{0,j}$&$\otimes$&$ V_{sm,k}$&$=$&$V_{sm,j+k}$
&
$V^-_{0,j}$&$\otimes$&$V_{sm,k}$&$=$&$V_{sm,j+k}$\\
%6
$V^+_{0,j}$&$\otimes$&$ V_{lm+i,k}$&$=$&$V_{lm+i,j+k}$
&
$V^-_{0,j}$&$\otimes$&$V_{lm+i,k}$&$=$&$V_{lm+i,j+k}$\\
\hline
%Lower Blocks
%1
$V^+_{\frac{d}{2}m,j}$&$\otimes$&$ V^+_{0,k}$&$=$&$V^+_{\frac{d}{2}m,j+k}$
&
$V^-_{\frac{d}{2}m,j}$&$\otimes$&$V^+_{0,k}$&$=$&$V^-_{\frac{d}{2}m,j+k}$\\
%2
$V^+_{\frac{d}{2}m,j}$&$\otimes$&$ V^-_{0,k}$&$=$&$V^-_{\frac{d}{2}m,j+k}$
&
$V^-_{\frac{d}{2}m,j}$&$\otimes$&$V^-_{0,k}$&$=$&$V^+_{\frac{d}{2}m,j+k}$\\
%3
$V^+_{\frac{d}{2}m,j}$&$\otimes$&$ V^+_{\frac{d}{2}m,k}$&$=$&$V^+_{0,j+k}$
&
$V^-_{\frac{d}{2}m,j}$&$\otimes$&$V^+_{\frac{d}{2}m,k}$&$=$&$V^-_{0,j+k}$\\
%4
$V^+_{\frac{d}{2}m,j}$&$\otimes$&$ V^-_{\frac{d}{2}m,k}$&$=$&$V^-_{0,j+k}$
&
$V^-_{\frac{d}{2}m,j}$&$\otimes$&$V^-_{\frac{d}{2}m,k}$&$=$&$V^+_{0,j+k}$\\
%5
$V^+_{\frac{d}{2}m,j}$&$\otimes$&$ V_{sm,k}$&$=$&$V_{(s+\frac{d}{2})m,j+k}$
&
$V^-_{\frac{d}{2}m,j}$&$\otimes$&$V_{sm,k}$&$=$&$V_{(s+\frac{d}{2})m,j+k}$\\
%6
$V^+_{\frac{d}{2}m,j}$&$\otimes$&$ V_{lm+i,k}$&$=$&$V_{(l+\frac{d}{2})m,j+k}$
&
$V^-_{\frac{d}{2}m,j}$&$\otimes$&$V_{lm+i,k}$&$=$&$V_{(l+\frac{d}{2})m,j+k}$\\
\hline
\end{tabular}
\end{center}

\renewcommand{\arraystretch}{3.4}
\setlength{\tabcolsep}{8pt}
\begin{center}
\begin{tabular}{c}
\hline
$\quad\quad\quad\quad \quad\quad\quad\quad \quad\quad\quad\quad  \textrm{The \;fusion\; rule \;II}\quad\quad\quad\quad \quad\quad\quad\quad\quad\quad\quad \quad\quad\quad\quad $\\
\end{tabular}
\begin{tabular}{|ccccc|}
\hline
%1
$V_{sm,j}$&$\otimes$&$ V^{+}_{0,k}$&$=$&$V_{sm,j+k}$\\
$V_{sm,j}$&$\otimes$&$ V^{-}_{0,k}$&$=$&$V_{sm,j+k}$\\
$V_{sm,j}$&$\otimes$&$ V^{+}_{\frac{d}{2}m,k}$&$=$&$V_{(s+\frac{d}{2})m,j+k}$\\
$V_{sm,j}$&$\otimes$&$ V^{-}_{\frac{d}{2}m,k}$&$=$&$V_{(s+\frac{d}{2})m,j+k}$\\
$V_{sm,j}$&$\otimes$&$ V_{lm,k}$&$=$&$V_{(s+l)m,j+k}\oplus V_{(s-l)m,j+k}\;\;\;\;\;\;(\ast)$\\
$V_{sm,j}$&$\otimes$&$ V_{lm+i,k}$&$=$&$V_{(s+l)m+i,j+k}\oplus V_{(l-s)m+i,j+k}$\\
\hline
$V_{lm+i,j}$&$\otimes$&$ V^{+}_{0,k}$&$=$&$V_{lm+i,j+k}$\\
$V_{lm+i,j}$&$\otimes$&$ V^{-}_{0,k}$&$=$&$V_{lm+i,j+k}$\\
$V_{lm+i,j}$&$\otimes$&$ V^{+}_{\frac{d}{2}m,k}$&$=$&$V_{(l+\frac{d}{2})m+i,j+k}$\\
$V_{lm+i,j}$&$\otimes$&$ V^{-}_{\frac{d}{2}m,k}$&$=$&$V_{(l+\frac{d}{2})m+i,j+k}$\\
$V_{lm+i,j}$&$\otimes$&$ V_{sm,k}$&$=$&$V_{(s+l)m+i,j+k}\oplus V_{(l-s)m+i,j+k}$\\
$V_{lm+i,j}$&$\otimes$&$ V_{sm+t,k}$&$=$&$V_{(s+l)m+(i+t),j+k}\oplus V_{(l-s)m+(i-t),j+k-t}\;\;\;\;\;\;(\ast)$\\
\hline
\end{tabular}
\end{center}
where the mark $(\ast)$, say for the case $V_{sm,j}\otimes V_{lm,k}$, has the following meaning:
\begin{itemize}\item[(1)] If $(s+l)m\not\equiv 0,\frac{d}{2}m$ (mod $dm$) and $(s-l)m\not\equiv 0,\frac{d}{2}m$ (mod $dm$), then \begin{equation}\label{eq7.12} V_{sm,j}\otimes V_{lm,k}=V_{(s+l)m,j+k}\oplus V_{(s-l)m,j+k}.\end{equation}
\item[(2)] If $(s+l)m\equiv 0$ (mod $dm$), then in the formula \eqref{eq7.12} $V_{(s+l)m,j+k}$ is decomposed further and  represents $$V^{+}_{0,j+k}\oplus V^{-}_{0,j+k}.$$
\item[(3)] If $(s+l)m\equiv \frac{d}{2}m$ (mod $dm$), then in the formula \eqref{eq7.12} $V_{(s+l)m,j+k}$ is decomposed further and  represents $$V^{+}_{\frac{d}{2}m,j+k}\oplus V^{-}_{\frac{d}{2}m,j+k}.$$
\item[(4)] If $(s-l)m\equiv 0$ (mod $dm$), then in the formula \eqref{eq7.12} $V_{(s-l)m,j+k}$ is decomposed further and  represents $$V^{+}_{0,j+k}\oplus V^{-}_{0,j+k}.$$
\item[(5)] If $(s-l)m\equiv \frac{d}{2}m$ (mod $dm$), then in the formula \eqref{eq7.12} $V_{(s-l)m,j+k}$ is decomposed further and  represents $$V^{+}_{\frac{d}{2}m,j+k}\oplus V^{-}_{\frac{d}{2}m,j+k}.$$
\end{itemize}
Similarly, one can work out the meaning of mark $(\ast)$ for the formula $V_{lm+i,j}\otimes V_{sm+t,k}$. That is,  whenever $(s+l)m+(i+t)\equiv 0$ (mod $dm$) or $(s+l)m+(i+t)\equiv \frac{d}{2}m$ (mod $dm$) the item $V_{(s+l)m+(i+t),j+k}$ will split further and whenever $(l-s)m+(i-t)\equiv 0$ (mod $dm$) or $(l-s)m+(i-t)\equiv \frac{d}{2}m$ (mod $dm$) the item $V_{(l-s)m+(i-t),j+k-t}$ will split further.

$\bullet$\emph{The series of nonsemisimple finite-dimensional Hopf algebras.} Using the Hopf algebra $D=D(\underline{m},d,\gamma)$, we also can get many nonsemisimple finite-dimensional Hopf algebras, which are knew up the author's knowledge. The main idea to construct these finite-dimensional Hopf algebras is to generalize the exact sequence \eqref{eq5.2} \begin{equation*} \k\longrightarrow H_{00}\longrightarrow H\longrightarrow \overline{H}\longrightarrow \k.
\end{equation*}
That is, we want to substitute $H$ by our Hopf algebra $D(\underline{m},d,\gamma)$ and thus get finite-dimensional quotients. One can realize this idea through showing that every Hopf subalgebra $\k[x^{\pm t}]$ for $t\in \N$ is a normal Hopf subalgebra of $D.$  Since by definition we know that the element $x$ commutes with $g, y_{m_i}\;(1\leq i\leq \theta)$, we only need to show that
$ad(u_j)(x^{t})=u_j'x^{t}S(u_j'')\in \k[x^{\pm t}]$
for all $0\leq j\leq m-1.$ Through direct computation, we have
$$ad(u_j)(x^{t})=x^{-t}u_j'S(u_j'')=x^{-t}\e(u_j)\in \k[x^{\pm t}].$$
So we have the following exact sequence of Hopf algebras
\begin{equation}\label{eq7.13} \k\longrightarrow \k[x^{\pm t}]\longrightarrow D\longrightarrow D/(x^{t}-1)\longrightarrow \k.
\end{equation}
We denote the resulted Hopf algebra $D/(x^{t}-1)$ by $D_t$, i.e., $D_t:=D/(x^{t}-1).$
\begin{lemma}
 The Hopf algebra $D_t$ is finite-dimensional and has dimension $2m^2t.$
\end{lemma}
\begin{proof} We also want to use the bigrading of $D$ to compute the dimension of $D_t.$ By equation \eqref{eq6.1}, we know that $D$ is a free $\k[x^{\pm 1}]$-module of rank $2m^2$. Now through this bigrading \eqref{eq6.1} and the relation modular $x^{t}-1$, $D_t$ is also bigraded and is a free $\k[x]/(x^t-1)$-module of rank $2m^2$. Therefore, $\dim_{k} D_t=2m^2t.$ Actually, the following elements $\{x^ig^jy_t, x^{i}g^{j}u_{t}|0\leq i\leq t-1, 0\leq j\leq m-1, 0\leq t\leq m-1\}$ (we use the same notations as $D$ for simple) is a basis of $D_t.$
\end{proof}
It is not hard to give the generators and relations of this Hopf algebra:  one just need add one more relation in the definition of the Hopf algebra $D$, that is the relation $x^t=1.$ The coproduct, counit and the antipode are the same as $D$. It seems that there is no need to repeat  them again.

About this Hopf algebra, it has the following properties.
\begin{proposition}\label{p7.14} Retain above notations.
\begin{itemize}\item[(1)] The Hopf algebra $D_t$ is not pointed unless $m=1$. And in case $m>1$, its coradical is not a Hopf subalgebra.
\item[(2)] The Hopf algebra $D_t$ is not semisimple unless $m=1$.
\item[(3)] The Hopf algebra $D_t$ is pivotal, that is, its representation category is a pivotal tensor category.
\end{itemize}
\end{proposition}
\begin{proof} (1) Using totally the same method given the proof of Proposition \ref{p7.6}, the subspace spanned by $\{(x^{-d}g)^{i}u_j|0\leq i,j\leq m-1\}$ is a simple coalgebra, where $x^{-d}$ means its image in $\k[x^{\pm 1}]/(x^t-1).$ So $D_t$ has a simple coalgebra of dimension $m^2$. Therefore it is not pointed. If $m=1$, then it is easy to see that $D_t$ is just a group algebra. Using the same arguments stated in the proof of Proposition \ref{p4.8} (3), its coradical is not a Hopf subalgebra.

(2) Assume $m>1$ and we want to show that $D_t$ is not semisimple. On the contrary, if $D_t$ is semisimple then it is cosemisimple \cite{LR}. This implies every $y_j$ should lie in the coradical. This is absurd since clearly $y_{m_i}$ does not due to it is a nontrivial skew primitive element.

(3) Actually we can prove a stronger result, that is, the Hopf algebra $D$ is pivotal. To prove this stronger result, by Lemma \ref{l2.16} we only need to set the following formula for $S^2$:
\begin{equation}\label{eq7.14} S^2(h)=(g^{\sum_{i=1}^{\theta}m_i}x^c)h(g^{\sum_{i=1}^{\theta}m_i}x^c)^{-1},\;\;\;\;h\in D,
\end{equation}
where $c=-\frac{\sum_{i=1}^{\theta}(e_i+1)m_id}{2}$. Note that by second equation of \eqref{eq4.7}, $\sum_{i=1}^{\theta}(e_i+1)m_id$ is always even. Our task is to prove above formula. Indeed, on one hand,
\begin{align*}S^2(u_j)&=S(x^{b}g^{m-1}\prod_{i=1}^{\theta}(-1)^{j_i}\xi_{m_i}^{-j_i}\gamma^{-
m_i^2\frac{j_i(j_i+1)}{2}}
x^{j_im_id}
g^{-j_im_i}u_j)\\
&=S(u_j)\prod_{i=1}^{\theta}(-1)^{j_i}\xi_{m_i}^{-j_i}\gamma^{-
m_i^2\frac{j_i(j_i+1)}{2}}
x^{-j_im_id}
g^{j_im_i}g^{1-m}x^{-b}\\
&=x^bg^{m-1}\prod_{i=1}^{\theta}\xi_{m_i}^{-2j_i}\gamma^{-
m_i^2j_i(j_i+1)}
x^{j_im_id}
g^{-j_im_i}u_jx^{-j_im_id}
g^{j_im_i}g^{1-m}x^{-b}\\
&=x^{2b}g^{m-1}\gamma^{(1-m)j}\prod_{i=1}^{\theta}\xi_{m_i}^{-2j_i}\gamma^{-
m_i^2j_i(j_i+1)}\gamma^{j(j_im_i)}
x^{-2(1-m)d}g^{1-m}u_j\\
&=x^{2b-2(1-m)d}\prod_{i=1}^{\theta}\gamma^{-m_i^2j_i}u_j,
\end{align*}
where recall that $b=(1-m)d-\frac{\sum_{i=1}^{\theta}(e_i-1)m_i}{2}d.$

On the other hand,
\begin{align*}(g^{\sum_{i=1}^{\theta}m_i}x^c)u_j(g^{\sum_{i=1}^{\theta}m_i}x^c)^{-1}
&=x^{2c}x^{2d\sum_{i=1}^{\theta}m_i}\gamma^{-j\sum_{i=1}^{\theta}m_i}u_j\\
&=x^{2c+2d\sum_{i=1}^{\theta}m_i}\prod_{i=1}^{\theta}\gamma^{-m_i^2j_i}u_j.
\end{align*}
Since $$2c+2d\sum_{i=1}^{\theta}m_i=-\sum_{i=1}^{\theta}(e_i-1)m_id=2b-2(1-m)d,$$ we have $S^2(u_j)=(g^{\sum_{i=1}^{\theta}m_i}x^c)u_j(g^{\sum_{i=1}^{\theta}m_i}x^c)^{-1}.$

So to show the formula \eqref{eq7.14}, we only need to check it for $y_{m_i}$ for $1\leq i\leq \theta$ now. This is not hard. In fact,
\begin{align*}S^2(y_{m_i})&=S(-y_{m_i}g^{-m_i})\\
&=g^{m_i}y_{m_i}g^{-m_i}=\gamma^{-m_i^2}y_{m_i}\\
&=\gamma^{-m_i(m_1+\cdots+m_{\theta})}y_{m_i}\\
&=(g^{\sum_{i=1}^{\theta}m_i}x^c)y_{m_i}(g^{\sum_{i=1}^{\theta}m_i}x^c)^{-1}
\end{align*}
due to $\gamma^{m_im_j}=1$ for $i\neq j$ and $x$ commutes with $y_{m_i}.$

Therefore, the representation category of $D$ is pivotal. As a tensor subcategory, the category of representations of $D_t$ is pivotal automatically.
\end{proof}
  In \cite{BGNR}, the authors posed the following sentence ``it remains unknown whether there exists any Hopf algebra $H$ of dimension 24 such that neither $H$ nor $H^{\ast}$ has the Chevalley property" (see \cite[Introduction, third paragraph]{BGNR}). With the helping of the Hopf algebra $D_t$, we can give one now. We will show that $D_3$ has dimension 24 and has no Chevalley property. However, its dual $(D_3)^{\ast}$ indeed has Chevalley property. That is, we still can't fix the question posed in \cite{BGNR}. Anyway, it seems that the following example didn't written out explicitly in \cite{BGNR} and should be implicated in their classification in a dual version.
 \begin{example}\emph{ Let $m=2$ (this implies that $m$ has no nontrivial fraction now, that is, $\theta=1$) and take $d=6$. The condition $d=6$ guarantees the condition \eqref{eq4.7} is fulfilled and thus Hopf algebra $D(m,d,\gamma)$ exists. So we take $t=3$ and then we find
 $$\dim_{\k} D_t=2m^2t=24.$$
 In order to understand this Hopf algebra well, we give the presentation of this $D_3$: as an algebra, it is generated by $g,x,y,u_0,u_1$ and satisfies
 \begin{align*} & x^3=1,\quad g^2=1,\quad xg=gx,\quad y^2=0,\quad xy=yx,\quad yg=-gy,\\
 & xu_0=u_0x^{-1},\quad  xu_1=u_1x^{-1},\quad yu_0=2u_1=\textrm{i}u_0y,\quad yu_1=0=u_1y,\\
 & u_0g=gu_0,\quad u_1g=-u_1g\\
 & u_0u_0=g,\quad u_0u_1=\frac{- \textrm{i}}{2} yg,\quad u_1u_0=\frac{1}{2}yg,\quad u_1u_1=0,
 \end{align*}
 where $\textrm{i}$ is the imaginary square root of $-1$, that is $\textrm{i}=\sqrt{-1}.$
 The coproduct $\D$, the counit $\epsilon$ and the antipode $S$ of $D_3$ are given by
\begin{eqnarray*}
&&\D(x)=x\otimes x,\;\; \D(g)=g\otimes g,\;\;\D(y)=1\otimes y+y\otimes g,\\
&&\D(u_0)=u_0\otimes u_0-u_1\otimes gu_1,\;\;\D(u_1)=u_0\otimes u_1+u_1\otimes gu_0,\\
&&\epsilon(x)=\epsilon(g)=\epsilon(u_0)=1,\;\;\epsilon(u_1)=\epsilon(y)=0;\\
&&S(x)=x^{-1},\;\; S(g)=g^{-1},\;\;S(y)=-yg^{-1}, \\
&& S(u_0)=gu_0,\;\;S(u_1)=-\textrm{i}u_1.
\end{eqnarray*}
Next we claim that $D_3$ has no Chevalley property while its $(D_3)^{\ast}$ does. Recall that a Hopf algebra is said to have Chevalley property if it's coradical is a Hopf subalgebra. So to show the claim, it is enough to prove that the coradical of $D_3$ is not a Hopf subalgebra and its (Jacobson) radical is a Hopf ideal.  In fact, by Proposition \ref{p7.14} (1), its coradical is not a Hopf subalgebra. Now let's prove that its radical is a Hopf ideal. As usual, denote its radical by $J$ and then it is not hard to see that $y\in J$ since $y$ generates a nilpotent ideal. Using the relation $yu_0=2u_1$, $u_1\in J$. Now consider the quotient $D_3/(y,u_1)$. It is not hard to see that $D_3/(y,u_1)\cong \k (\Z_4\times \Z_3)$. Therefore, $J=(y,u_1)$ and it is a Hopf ideal clearly.}
 \end{example}
 \subsection{The Hypothesis.} We point out that our final aim is to classify all prime Hopf algebras of GK-dimension one. So, as a natural step, we want to consider the question about the Hypothesis (Hyp1) and (Hyp2) listed in the introduction.

 $\bullet$\emph{ The Hypothesis (Hyp1).}  Let $H$ be a prime Hopf algebra of GK-dimension one, does $H$ satisfy (Hyp1) automatically? It is a pity that this is not true as we have the following counterexample.

\begin{example}\label{ex7.5} \emph{Let $n$ be a natural number. As an algebra, $\Lambda(n)$ is generated by $X_1,\ldots,X_n$ and $g$ subject to the following relations:
$$X_i^2=X_j^2,\;\;X_iX_j=-X_jX_i,\;\;g^2=1,\;\;-gX_i=X_ig$$
for all $1\leq i\neq j\leq n.$ The coproduct, counit and the antipode are given by
$$\D(X_i)=1\otimes X_i+X_i\otimes g,\;\;\D(g)=g\otimes g,$$
$$\e(X_i)=0,\;\;\e(g)=1$$
$$S(X_i)=-X_ig,\;\;S(g)=g^{-1}$$
for all $1\leq i\leq n.$ By the following lemma, we know that $\Lambda(n)$ is a prime Hopf algebra of GK-dimension one when $n$ is odd. Moreover, if $n=2m+1$, then the PI-degree of $\Lambda(n)$ is $2^{m+1}$. }

\emph{Now let $$\pi: \Lambda(n)\to \k$$ be a 1-dimensional representation of $\Lambda(n)$. Since $g^2=1$, $\pi(g)=1$ or $\pi(g)=-1$. From the relation $-gX_i=X_ig$, we get $\pi(X_i)=0$ for all $1\leq i\leq n$. This implies that $\ord(\pi)=1$ or $\ord(\pi)=2$. In general, we find that PI-deg$(\Lambda(n))>\ord(\pi)$ and the difference PI-deg$(H)-\ord(\pi)$ can be very large.}
\end{example}
\begin{lemma} Keep the notations and operations used in above example. Then
\begin{itemize}\item[(1)] The algebra $\Lambda(n)$ is a Hopf algebra of GK-dimension one.
\item[(2)] The algebra $\Lambda(n)$ is prime if and only if $n$ is odd.
\item[(3)] If $n=2m+1$ is an odd, then PI-deg $(\Lambda(n))=2^{m+1}$.
\end{itemize}
\end{lemma}
\begin{proof} (1) is clear.

(2) If $n$ is even, then we consider the element $g\prod_{i=1}^{n}X_i$. Direct computation shows that this element belongs to the center $C(\Lambda(n))$. Also we know that $X_{1}^{n}$ lives in the center too. Thus $$X_1^{n}-ag\prod_{i=1}^{n}X_i\in C(\Lambda(n))$$
for any $a\in \k$. Now, $(X_1^{n}-ag\prod_{i=1}^{n}X_i)(X_1^{n}+ag\prod_{i=1}^{n}X_i)
=X_1^{2n}-a^2(-1)^\frac{n(n+1)}{2}\prod_{i=1}^{n}X_i^2=
X_1^{2n}-a^2(-1)^\frac{n(n+1)}{2}X_1^{2n}.$ Taking $a$ such that $a^2(-1)^\frac{n(n+1)}{2}=1$, we see that the central element $X_1^{n}-ag\prod_{i=1}^{n}X_i$ has nontrivial zero divisor and thus $\Lambda(n)$ is not prime.

So the left task is to show that $\Lambda(n)$ is prime when $n$ is odd. To prove this, we give the following two facts about the algebra $\Lambda(n)$: 1) The center of $\Lambda(n)$ is $\k[X_1^2]$ ($=\k[X_i^2]$ for $1\leq i\leq n$); 2) $\Lambda(n)$ is a free module over its center with basis $\{g^{l}\prod_{i=1}^{n}X_{i}^{j_i}|0\leq l\leq 1, 0\leq j_i\leq 1\}$. Both of these two facts can be gotten through the following observation easily: As an algebra, one has $\Lambda(n)\cong \overline{U}(n)\# \k\Z_2$ where $\overline{U}(n)=U(n)/(X_i^2-X_j^2|1\leq i\neq j\leq n)$ and $U(n)$ is the enveloping algebra of the commutative Lie superalgebra of dimension $n$ with degree one basis $\{X_i|1\leq i\leq n\}$.

From above two facts about $\Lambda(n)$, every monomial generated by $g$ and $X_i$ ($1\leq i\leq n$) is not a zero divisor and in fact regular. Now to show the result, assume that $I,J$ be two nontrivial ideals of $\Lambda(n)$ satisfying $IJ=0$. We will show that $I$ contains a monomial and thus get a contradiction. For this, through setting $\deg (g)=0$ and $\deg (X_i)=1$ we find that $\Lambda(n)$ is a graded algebra. Let $a$ and $b$ be two nonzero element of $I$ and $J$ respectively. Since $\Lambda(n)$ is $\Z$-graded which is an order group, we can assume that both $a$ and $b$ are homogenous elements through $ab=0$. In particular, we can take $a$ to be a nonzero homogenous element. For simple, we assume that $a$ has degree one (for other degrees one can prove the result using the same way as degree one). So, $$a=\sum_{i=1}^{n} a_iX_i+\sum_{i=1}^{n}a_i'gX_{i},$$
for $a_i,a_i'\in \k.$ Now $a':=X_1a+aX_1=2a_1X_1^2-2\sum_{i\neq 1}a_i'gX_1X_i$. For any $i\neq 1$, we have $a'':=X_{i}a'+a'X_i=4a_1X_1^2X_i-4a_i'gX_1X_i^2$ and continue this process $a''':=X_ja''+a''X_j=-8a_i'gX_1X_i^2X_j\in I$ for any $j\neq 1,i$ (such $j$ exists unless $n=1$. But in case $n=1$, $\Lambda(n)$ is clear prime). This implies that we have a monomial in $I$ if $a_i'\neq 0$ for $i\neq 1$. We next consider the case $a_i'\equiv 0$ for all $i\neq 1.$ Looking back the element $a''$, we can assume that $a_1=0$ too. Repeat above precess through substituting $X_1$ by other $X_j$ and we can assume all $a_j=0$ and $a_t'=0$ with $t\neq j$. That's impossible since $0\neq a$ and in one word we must have a monomial in $I$.

(3) By the proof of the part (2), we know that $\Lambda(n)$ is a free module over its center with basis $\{g^{l}\prod_{i=1}^{n}X_{i}^{j_i}|0\leq l\leq 1, 0\leq j_i\leq 1\}$ and so the rank of $\Lambda(n)$ over its center is $2^{n+1}=2^{2(m+1)}$. Therefore, PI-deg$(\Lambda(n))=\sqrt{2^{2(m+1)}}=2^{m+1}.$
\end{proof}

$\bullet$ \emph{The Hypothesis (Hyp2).} We next want to consider the question about the second hypothesis (Hyp2): Let $H$ be a prime Hopf algebra of GK-dimension one, does $H$ has a one-dimensional representation $\pi: H\to \k$ such its invariant components are domains? This is also not true in general. In fact, by Example \ref{ex7.5}, we find that the left invariant component must contains the subalgebra generated by $X_i$ ($1\leq i\leq n$) for any one-dimensional representation and thus it is not a domain (if it is, it must be commutative by the proof of Lemma \ref{l2.7}).

$\bullet$ \emph{Relation between (Hyp1) and (Hyp2).} In the introduction, (Hyp2) is built on (Hyp1), i.e., they used the same one-dimensional representation. However, it is clear we can consider (Hyp1) and (Hyp2) individually, that is, for each hypothesis we consider a one-dimensional representation which may be different. Until now, we still don't know the exactly relationship between (Hyp1) and (Hyp2) for a prime Hopf algebra of GK-dimension one. So, we formulate the following question for further considerations.

\begin{question} \begin{itemize}
\item[(1)] \emph{Let $H$ be a prime Hopf algebra of GK-dimension one satisfying (Hyp1), does $H$ satisfy (Hyp2) automatically?}
\item[(2)] \emph{Let $H$ be a prime Hopf algebra of GK-dimension one satisfying (Hyp2), does $H$ satisfy (Hyp1) automatically?}
\end{itemize}
\end{question}

\subsection{A conjecture.} From all examples stated in this paper, it seems that prime Hopf algebras of GK-dimension one exist widely. However, we still can find some common points about them. Among of these points, we formulate a conjecture on the structure of prime Hopf algebras of GK-dimension in the following way.
\begin{conjecture}\label{con7.19} \emph{Let $H$ be a prime Hopf of GK-dimension one. Then we have an exact sequence of Hopf algebras:}
\begin{equation} \k\longrightarrow \emph{alg.gp} \longrightarrow H\longrightarrow \emph{f.d. Hopf}\longrightarrow \k,
\end{equation}
\emph{where ``alg.gp" denotes the coordinate algebra of a connected algebraic group of dimension one and ``f.d. Hopf" means a finite-dimensional Hopf algebra.}
\end{conjecture}

It is not hard to see that all examples given in this paper always satisfy above conjecture.
\begin{remark} \emph{Recently, professor Ken Brown showed the author one of his slides in which he introduced the definition so called \emph{commutative-by-finite} as follows: A Hopf algebra is commutative-by-finite if it is a finite (left or right) module over a commutative normal Hopf subalgebra. So our Conjecture \ref{con7.19} just says that every prime Hopf algebra of GK-dimension one should be a commutative-by-finite Hopf algebra.}
\end{remark}


\begin{thebibliography}{99}
\bibitem{AA} N.  Andruskiewitsch, I.E. Angiono, On Nichols algebras with generic braiding. Modules and comodules, 47-64, Trends Math., Birkh$\ddot{a}$user Verlag, Basel, 2008.

\bibitem{AAH} N.  Andruskiewitsch, I.E. Angiono, I. Heckenberger, On finite GK-dimensional Nichols algebras over abelian groups. arXiv:1606.02521.

\bibitem{AS} N. Andruskiewitsch, H.-J. Schneider, A characterization of quantum groups. J. Reine Angew. Math. 577 (2004), 81-104.

\bibitem{AS1} N. Andruskiewitsch, H.-J. Schneider,
On the classification of finite-dimensional pointed Hopf algebras. Ann. of Math. (2) 171 (2010), no. 1, 375-417.

\bibitem{Ang} I. Angiono, On Nichols algebras of diagonal type. J. Reine Angew. Math. 683 (2013), 189-251.
\bibitem{BGNR} M. Beattie, G. A. Garc\'ia, S.-H. Ng, J. Roat, Nonsemisimple Hopf algebras of dimension $8p$ with the Chevalley property, arXiv:1803.00271v1.

\bibitem{BG} K. A. Brown and K. R. Goodearl, Homological aspects of Noetherian PI Hopf algebras and Irreducible modules of maximal dimension, J. Algebra 198 (1997), 240-265.

\bibitem{Br} K. A. Brown, Representation theory of Noetherian Hopf algebras satisfying a polynomial identity. Trends in the representation theory of finite-dimensional algebras (Seattle, WA, 1997), 49-79, Contemp. Math., 229, Amer. Math. Soc., Providence, RI, 1998.

\bibitem{BZ} K. A. Brown and J. J. Zhang, Prime regular Hopf algebras of GK-dimension one,
Proc. London Math. Soc. (3) 101 (2010) 260-302.

\bibitem{EGNO} P. Etingof, S. Gelaki, D. Nikshych, V. Ostrik, Tensor categories. Mathematical Surveys and Monographs, 205. American Mathematical Society, Providence, RI, 2015. xvi+343 pp.

\bibitem{ENO}
P. Etingof, D. Nikshych, V. Ostrik, On fusion categories. Ann. of Math. (2) 162 (2005), no. 2, 581-642.

\bibitem{EO} P. Etingof, V. Ostrik, On semisimplification of tensor categories, arXiv: 1801.04409v1.

\bibitem{Go} K. R. Goodearl, Noetherian Hopf algebras. Glasg. Math. J. 55 (2013), no. A, 75-87.

\bibitem{GZ} K. R. Goodearl, J. J. Zhang, Noetherian Hopf algebra domains of Gelfand-Kirillov
dimension two, J. Algebra 324 (2010), no. 11, 3131-3168.

\bibitem{He} I. Heckenberger, The Weyl groupoid of a Nichols algebra of diagonal type. Invent. Math. 164 (2006), no. 1, 175-188.

\bibitem{He1} I. Heckenberger,  Classification of arithmetic root systems. Adv. Math. 220 (2009), no. 1, 59-124.

\bibitem{Hu} J. E. Humphreys, Linear algebraic groups, Graduate Texts in
Mathematics 21 (Springer, New York, 1975).

\bibitem{Kas} C. Kassel, Quatum groups, GTM 155, Springer-Verlag, New York, 1995.

\bibitem{Lam} T. Y. Lam,  Lectures on modules and rings, Graduate Texts in Mathematics, 189. Springer-Verlag, New York, 1999. xxiv+557 pp.

\bibitem{LR} R. G. Larson, D. E. Radford, Finite-dimensional cosemisimple Hopf algebras in characteristic 0 are semisimple. J. Algebra 117 (1988), no. 2, 267-289.

\bibitem{Liu} Gongxiang Liu, On Noetherian affine prime regular Hopf algebras of Gelfand-Kirillov
dimension 1, Proc. Amer. Math. Soc. 137(2009), 777-785.

\bibitem{LL} M. E. Lorenz and M. Lorenz, On crossed products of Hopf algebras, Proc. Amer. Math. Soc.
123(1) (1995), 33-38.

\bibitem{LWZ} D.-M. Lu, Q.-S. Wu and J. J. Zhang, Homological integral of Hopf algebras, Trans.
Amer. Math. Soc. 359 (2007), 4945-4975.

\bibitem{MR} J. C. McConnell and J. C. Robson, Noncommutative noetherian rings
(Wiley, Chichester, 1987).

\bibitem{Maj} S. Majid, Cross products by braided groups and bosonization. J. Algebra 163 (1994), no. 1, 165-190.
\bibitem{Mo} S. Montgomery, Hopf algebra and their actions on rings, CBMS Regional Conference
Series in Mathematics, 82, Providence, RI, 1993.

\bibitem{NS}  S.-H. Ng, P. Schauenburg, Higher Frobenius-Schur indicators for pivotal categories. Hopf algebras and generalizations, 63-90, Contemp. Math., 441, Amer. Math. Soc., Providence, RI, 2007.

\bibitem{Rad} D. Radford, The structure of Hopf algebras with a projection.
J. Algebra 92 (1985), no. 2, 322-347.

\bibitem{Ro} M. Rosso, Quantum groups and quantum shuffles. Invent. Math. 133 (1998), no. 2, 399-416.

\bibitem{Sm} L. W. Small, Prime affine algebras
of Gelfand-Kirillov dimension one, J. Algebra 91 (1984), 386-389.

\bibitem{SSW} L. W. Small, J. T. Stafford and R. B. Warfield Jr., Affine algebras
of Gelfand-Kirillov dimension one are PI, Math. Proc. Cambridge
Philos. Soc. 97 (1985) 407-414.

%\bibitem{Sw} Sweedler, M.E. Hopf Algebra. New York: Benjamin, 1969.

\bibitem{WZZ1} D.-G. Wang, J.J. Zhang and G.-B. Zhuang, Lower bounds of growth of Hopf
algebras, Trans. Amer. Math. Soc. 365 (2013), 4963-4986.

\bibitem{WZZ2} D.-G. Wang, J.J. Zhang and G.-B. Zhuang, Hopf algebras of GK-dimension two
with vanishing Ext-group, J. Algebra 388 (2013), 219-247.

\bibitem{WZZ3} D.-G. Wang, J.J. Zhang and G.-B. Zhuang, Connected Hopf algebras of Gelfand-
Kirillov dimension four, Trans. Amer. Math. Soc. 367 (2015), 5597-5632.

\bibitem{Wu} Jingyong Wu, Note on the coradical filtration of $D(m,d,\xi).$ Comm. Algebra 44 (2016), no. 11, 4844-4850.

\bibitem{WLD}  J. Wu, G. Liu, N. Ding, Classification of affine prime regular Hopf algebras of GK-dimension one. Adv. Math. 296 (2016), 1-54.

\bibitem{WZ} Q.-S. Wu and J.J. Zhang, Noetherian PI Hopf algebras are Gorenstein, Trans.
Amer. Math. Soc. 355 (2003), no. 3, 1043-1066.

\bibitem{Zh} G.-B Zhuang, Properties of pointed and connected Hopf algebras of finite Gelfand-Kirillov dimension. J. Lond. Math. Soc. (2) 87 (2013), no. 3, 877-898.
\end{thebibliography}
\end{document}